%% file: main.tex
\documentclass[12pt]{amsart}

\usepackage[utf8]{inputenc}
\usepackage{amssymb,mathrsfs, mathtools,cancel}
\usepackage{xcolor}
\definecolor{darkblue}{rgb}{0,0,0.4} 
\usepackage[colorlinks=true, citecolor=darkblue, filecolor=darkblue, linkcolor=darkblue,urlcolor=darkblue]{hyperref}
\usepackage[weather]{ifsym}
\usepackage[all]{xy}
\usepackage[lmargin=1in,rmargin=1in,tmargin=1in,bmargin=1in]{geometry}
\usepackage{enumitem}
\usepackage{color}
\usepackage{tikz}
\usetikzlibrary{cd}
\graphicspath{{draws/}{}}

\usepackage{tensor}
\newcommand{\phipsiphi}{\!\tensor*[_{\phi}]{\cup}{_{\phi\circ\psi}}\!}
\newcommand{\RcupL}{\tensor*[_{\bdy_R}]{\cup}{_{\bdy_L}}}

\input{defs.tex}

\begin{document}
\title[Bordered Floer, compression bodies, and diffeomorphisms]{Bordered Floer homology, handlebody detection, and compressing diffeomorphisms}

\author{Akram Alishahi}
\thanks{\texttt{AA was partly supported by NSF Grants DMS-2000506 and DMS-2238103 and partly
    supported by NSF Grant DMS-1928930 while in residence at SLMath}}
\address{Department of Mathematics, University of Georgia, Athens, GA 30602}
\email{\href{mailto:akram.alishahi@uga.edu}{akram.alishahi@uga.edu}}

\author{Robert Lipshitz}
\thanks{\texttt{RL was partly supported by NSF Grant DMS-2204214 and partly
    supported by NSF Grant DMS-1928930 while in residence at SLMath}}
\address{Department of Mathematics, University of Oregon, Eugene, OR 97403}
\email{\href{mailto:lipshitz@uoregon.edu}{lipshitz@uoregon.edu}}


\date{\today}

\begin{abstract}
  We show that, up to connected sums with integer homology $L$-spaces,
  bordered Floer homology detects handlebodies,
  as well as whether a mapping class extends over a given
  handlebody or compression body. Using this, we combine ideas of Casson-Long with the theory of train tracks to give an
  algorithm using bordered Floer homology to detect whether a mapping
  class extends over any compression body.
\end{abstract}

\maketitle

\tableofcontents
\section{Introduction}
To apply a numerical or algebraic invariant to topological problems,
it is helpful to know what geometric information the invariant
contains. For Heegaard Floer homology, some of the most useful
information comes from its famous detection properties:
Ozsv\'ath and Szab\'o's theorems that it detects the genus of knots and
the Thurston norm of 3-manifolds~\cite{OS04:ThurstonNorm} and Ni's
theorems that it detects fiberedness of knots and
3-manifolds~\cite{Ni07:FiberedKnot,Ni09:FiberedMfld} (see
also~\cite{Ghiggini08:FiberedGenusOne}). A useful cousin of these
properties is that various twisted forms of Heegaard Floer homology
detect the existence of homologically essential
2-spheres~\cite{HeddenNi13:detects,AL19:incompressible}.

This last result extends easily to show that twisted Heegaard Floer
homology can also be used to detect the presence of homologically
linearly independent 2-spheres (Lemma~\ref{lem:sphere-support}, below).
In this paper, we use that extension to give two new phenomena that
bordered Heegaard Floer homology detects: it detects handlebodies
(Theorem~\ref{thm:detect-hb}), and also whether a diffeomorphism extends
over a given handlebody or compression body
(Theorem~\ref{thm:detect-over-specific}). A related question is whether
a diffeomorphism extends over any handlebody. Casson and Long give an
algorithm for answering this in the
1980s~\cite{CassonLong85:compression}. (One reason for interest was a
theorem of Casson and Gordon that the monodromy of a fibered ribbon knot
extends over a handlebody~\cite{CG83:fibered-ribbon}.) We modify their
algorithm to show that bordered-sutured Floer homology can be used to
detect whether a diffeomorphism extends over some compression body, and
give explicit bounds on the complexity of the bimodules involved
(Theorem~\ref{thm:phi-extend-bound}). In particular, the proof involves
replacing some bounds in terms of the lengths of geodesics with bounds
in terms of train tracks and giving a connection between train track
splitting sequences and bordered-sutured bimodules associated to mapping
classes, both of which may be of independent use.

We state these results in a little more detail, starting with detection of handlebodies:
\begin{theorem}\label{thm:detect-hb} Let $Y$ be an irreducible homology
  handlebody. Fix $\phi\co F\to \bdy Y$ making $Y$ into a bordered
  $3$-manifold and let $\lsup{\Alg(F)}\tCFDAa(\Id)_{\Alg(F)}$ be the
  twisted identity bimodule of $F$. Then the support of
  \[
    \Mor_{\Alg(F)}\bigl(\lsup{\Alg(F)}\CFDa(Y),\lsup{\Alg(F)}\tCFDAa(\Id)_{\Alg(F)}\DT\lsup{\Alg(F)}\CFDa(Y)\bigr)
  \]
  over $\FF_2[H_2(Y,F)]$ is $0$-dimensional if and only if $Y$ is a
  handlebody.
\end{theorem}

\begin{corollary}\label{cor:detect-hb}
  Let $Y$ be a bordered $3$-manifold so that $\CFDa(Y)$ is homotopy equivalent to $\CFDa(H,\phi)$ for some bordered handlebody $H$, as (relatively) graded type $D$ structures. Then $Y$ is a connected sum of a handlebody and an integer homology sphere $L$-space.
\end{corollary}

Given a surface $\Sigma$, a $3$-manifold $Y$ with $\Sigma\subset \bdy Y$, and a homeomorphism
$\psi\co \Sigma\to \Sigma$, $\psi$ \emph{extends over} $Y$ if there is a
homeomorphism $\Psi\co Y\to Y$ so that $\Psi|_{F}=\psi$. 
With a little more work, the techniques
used to prove Theorem~\ref{thm:detect-hb} also show that bordered Floer
homology detects whether a homeomorphism extends over a given handlebody
or compression body:

\begin{theorem}\label{thm:detect-over-specific}%
  Let $C$ be a compression body with outer boundary $\Sigma$ and $k$ components
  of its inner boundary (none of which are spheres), and let $\psi\co
  \Sigma\to \Sigma$ be a homeomorphism. Let $g'$ be the sum of the genera of the components of the inner boundary of $C$. Make $C$ into a special
  bordered-sutured manifold with outer bordered boundary $F$, inner
  bordered boundary $F'$, and $m$ sutures on the inner boundary, and
  choose a strongly based representative for $\psi$ (i.e., a
  representative respecting the sutured structure on $\bdy F$). Then
  $\psi$ extends over $C$ if and only if $\psi$ preserves the
  kernel of the map $H_1(\Sigma)\to H_1(C)$ and the support of
  \begin{equation}\label{eq:detect-over-specific-formula}
    \Mor_{\Alg(F)}\bigl(\lsup{\Alg(F'),\Alg(F)}\BSD(C),\lsup{\Alg(F)}\BSDA(\psi)_{\Alg(F)}\DT\lsup{\Alg(F),\Alg(F')}\tBSD(C)\bigr)
  \end{equation}
  over $\FF_2[H_2(C,F\cup F')]$ is $(2g'+k-m)$-dimensional.
\end{theorem}
In particular, if there is one suture on each inner boundary component, then the question is whether the support is $2g'$-dimensional. 

If $C$ is a handlebody, Formula~\eqref{eq:detect-over-specific-formula} reduces to whether 
\[
  \dim\Supp\Mor_{\Alg(F)}\bigl(\lsup{\Alg(F)}\CFDa(C),\lsup{\Alg(F)}\CFDAa(\psi)_{\Alg(F)}\DT\lsup{\Alg(F)}\tCFDa(C)\bigr)=0.
\]
This can also be interpreted as 
\[
  \dim\Supp\Mor_{\Alg(F)}\bigl(\lsup{\Alg(F)}\CFDa(C),\lsup{\Alg(F)}\tCFDAa(\Id)_{\Alg(F)}\DT\lsup{\Alg(F)}\CFDAa(\psi)_{\Alg(F)}\DT\lsup{\Alg(F)}\CFDa(C)\bigr),
\]
as in Theorem~\ref{thm:detect-hb}.

Since the bordered Floer algebras categorify the exterior
algebra on $H_1(F)$ or, more precisely, bordered Floer homology
categorifies the Donaldson TQFT~\cite{HLW17:chi,Petkova18:decat}, the
bordered condition in Theorem~\ref{thm:detect-over-specific} is in some
sense a categorification of the obvious necessary condition that $\psi$
preserve the kernel of $H_1(\Sigma)\to H_1(C)$. So, like with the Thurston
norm, Floer homology detects a phenomenon which classical topology
merely obstructs. There are other interesting obstructions to partially
extending diffeomorphisms over 3-manifolds with boundary, e.g., in terms
of laminations~\cite{BJM:partially-extend}; in light of the third part
of the paper, it might be interesting to compare the two kinds of
techniques.

Theorems~\ref{thm:detect-hb} and~\ref{thm:detect-over-specific} are
effective: the bordered invariants are computable~\cite{LOT1} (see
also~\cite{Zhan:bfh}), and the dimension of the support of these
modules can also be computed (Section~\ref{sec:compute-support}).

A related problem to asking whether a given homeomorphism $\psi$ of
$\Sigma$ extends over a given compression body filling of $\Sigma$ is to ask if
$\psi$ extends over any compression body filling of $\Sigma$. In 1985,
Casson-Long showed that this problem is algorithmic, using a bound in
terms of how $\psi$ interacts with geodesics on $\Sigma$, which they call
the intercept length~\cite{CassonLong85:compression}. In particular, they show that if $\psi$ extends
over a handlebody, it also extends over a compression body containing
a disk whose boundary is relatively short (with respect to a metric on
$\Sigma$). Combining some of their ideas with results about train tracks,
we show that if $\psi$ extends over a handlebody then $\psi$ extends
over a compression body whose bordered Floer bimodule is relatively
small. More precisely:

\begin{theorem}\label{thm:phi-extend-bound}
  Let $\psi\co \Sigma\to \Sigma$ be a mapping class of a closed surface $\Sigma$ with
  genus $g$. Let $(m_{i,j})$ be the
  incidence matrix for $\psi$ with respect to some train track $\tau$
  carrying $\psi$, 
  $\kappa$ be the number of connected components of $\Sigma\setminus\tau$,
  $s$ be the number of switches of $\tau$, and $M(\psi)$ be as in
  Formula~\eqref{eq:def-M-psi}. Then there is a bordered Heegaard
  diagram $\HD$ for a compression body $C_{\HD}$ with boundary $\Sigma$ so
  that $\psi$ extends over $C_{\HD}$ and $\HD$ has at most
  \[
  (20(g+s)-18)^s\left((2M(\psi))^{2g}+(2M(\psi)+8)^{2(g+s-1)}\right)
  \]  
  many generators.
\end{theorem}
(This is re-stated and proved as Theorem~\ref{thm:phi-extend-bound-v2}.)

\begin{corollary}\label{cor:alg-extend} Bordered-sutured Floer homology
  gives an algorithm to test whether $\psi$ extends over some
  compression body, or over some handlebody.
\end{corollary}

The proof of Theorem~\ref{thm:phi-extend-bound} has several
ingredients. One is a construction of bordered Heegaard diagrams from
trivalent train tracks. Another is to define a notion of length of a
curve in terms of its intersections with stable and unstable train
tracks for $\psi$, and to use Agol's periodic splitting sequences to
give bounds on this length for some curve bounding a disk in the
compression body.

\begin{convention}
  Because we are working mostly with compression bodies and
  handlebodies, throughout this paper by \emph{3-manifold} we mean a
  compact, connected, orientable 3-manifold with boundary; when
  3-manifolds are required to be closed, we will say that explicitly.
\end{convention}

This paper is organized as follows. Section~\ref{sec:traintracks} is a quick review of the theory of train tracks for pseudo-Anosov maps, collecting the material needed to prove Theorem~\ref{thm:phi-extend-bound}. Section~\ref{sec:compression-bodies} has some general results about compression bodies and decompositions of 3-manifolds along spheres, needed for Theorems~\ref{thm:detect-hb} and~\ref{thm:detect-over-specific}. Terminology about compression bodies is also explained in Section~\ref{sec:compression-bodies}. Section~\ref{sec:HF-background} recalls twisted Heegaard Floer homology and sutured Floer homology and untwisted bordered-sutured Floer homology, and then introduces twisted bordered-sutured Floer homology. It also introduces the notion of special bordered-sutured manifolds. None of the material in that section will be surprising to experts, but some has not yet appeared in the literature. Section~\ref{sec:detect-handlebodies} starts by recalling the definition of the support and then proves Theorem~\ref{thm:detect-hb}. It also indicates what we mean by the support of the bordered invariants (needed for later sections) and discusses how one can compute these supports. The proof of Theorem~\ref{thm:detect-over-specific} is in Section~\ref{sec:HF-detect}. Section~\ref{sec:tt-arc-diagram-slides} connects the material from Sections~\ref{sec:traintracks} and~\ref{sec:HF-background}, showing how train tracks give arc diagrams and splitting sequences give factorizations of mapping classes into arcslides. Section~\ref{sec:compress-diffeo} builds on these ideas to prove Theorem~\ref{thm:phi-extend-bound} and Corollary~\ref{cor:alg-extend}.

\subsubsection*{Acknowledgments}
We thank Nick Addington, Mark Bell, Spencer Dowdall, Dave Gabai, Tye Lidman, Dan Margalit, and Lee Mosher for
helpful conversations. We thank Chi Cheuk Tsang for pointing out an error in a previous version and discussions about correcting it. Finally, we thank the Simons Laufer Mathematical Sciences
Institute for its hospitality in Fall 2022, while much of this work
was conducted.

\section{Background on pseudo-Anosov maps and train-tracks}\label{sec:traintracks}

Let $\Sigma$ be an oriented surface. For the most part, this section follows \cite{PennerHarrer92:book}, so we assume $\Sigma$ is not the once punctured torus. A diffeomorphism $\psi$ of $\Sigma$  is called \emph{pseudo-Anosov} if there exist transverse measured foliations $(\mathcal{F}^s,\mu_s)$ and  $(\mathcal{F}^u,\mu_u)$ on $\Sigma$ and a real number $\lambda>1$ such that $\mathcal{F}^s$ and $\mathcal{F}^u$ are preserved by $\psi$, while the measures $\mu_s$ and $\mu_u$ are multiplied by $1/\lambda$ and $\lambda$, respectively. The constant $\lambda(\psi)=\lambda$ is called the \emph{dilatation number} of $\psi$.

Thurston introduced \emph{measured train tracks} as combinatorial tools to encode the measured foliations $(\mathcal{F}^s,\mu_s)$ and  $(\mathcal{F}^u,\mu_u)$. A \emph{train track} $\tau\subset \Sigma$ is an embedded graph on $\Sigma$ satisfying the following conditions: 
\begin{enumerate}
\item Each edge is $C^1$ embedded, and at each vertex there is a well-defined tangent line to all of the edges adjacent to it as in Figure \ref{fig:switch-basics}.
\item For every connected component $S$ of $\Sigma\setminus \tau$, the Euler characteristic of the double of $S$ along its boundary with cusp singularities on $\bdy S$ removed is negative.
\end{enumerate}
Edges and vertices of a train track are called \emph{branches} and \emph{switches}, respectively. A train track is called \emph{generic} if every switch is trivalent. 
In this paper, we work with generic train tracks, unless stated otherwise. 

\begin{figure}
  \centering
  \includegraphics{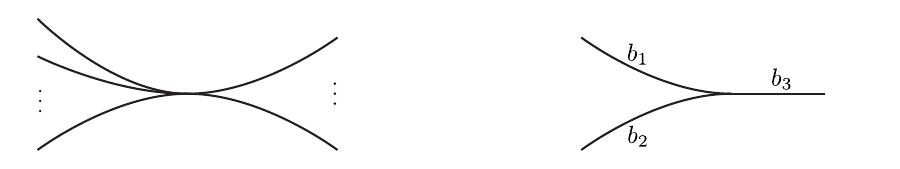}
  \caption[Switches]{\textbf{Switches}. Left: a general switch. Right:
    a trivalent switch, with incident half-branches $b_1$, $b_2$, and
    $b_3$. The half-branches $b_1$ and $b_2$ are small and $b_3$ is large. The switch condition for a measure $\mu$ is that
    $\mu(b_3)=\mu(b_1)+\mu(b_2)$.}
  \label{fig:switch-basics}
\end{figure}

Given a branch $b$ and some point $p$ in the interior of $b$, the components of $\mathrm{int}(b)\setminus p$ are called \emph{half-branches} of $b$. Moreover, two half-branches of $b$ are equivalent if their intersection is a half-branch as well. Whenever we talk about half-branches, we mean an equivalence class of half-branches. Every switch is in the closure of three half-branches, two \emph{small} half-branches bounding the cusp region and one \emph{large} half-branch on the other side. Every branch contains two half-branches, and is called \emph{large} if both of its half-branches are large. 

We require train tracks to be \emph{filling}, i.e., every
connected component of $\Sigma\setminus \tau$ is either a polygon or a or once-punctured polygon.
For every connected component $A$ of $\Sigma\setminus \tau$, an \emph{edge} of $A$ is a maximal smooth arc $e\subset \tau\cap\bdy A$. 

Given train tracks $\tau$ and $\tau'$ on $\Sigma$, we say $\tau$ \emph{carries} $\tau'$, and write $\tau'<\tau$, if there exists a smooth map $f\co \Sigma\to\Sigma$ homotopic to the identity such that $f(\tau')\subset \tau$ and the restriction of $f$ to the tangent spaces of $\tau'$ is non-singular. (This condition does not imply that $f$ sends switches to switches.) Similarly, we say a simple closed curve $\gamma$ is \emph{carried} by a train track $\tau$, and write $\gamma<\tau$, if $\gamma$ is smoothly homotopic to a curve in $\tau$ via a map whose restriction to the tangent space of $\gamma$ is non-singular.

Roughly, the \emph{incidence matrix} $M=\left(m_{ij}\right)$ for the carrying $\tau'<\tau$ using the map $f$ is defined so that $m_{ij}$ is the number of times $f(b_j')$ traverses $b_i$ in either direction. Here, $b_i$ and $b_j'$ denote the $i$-th and $j$-th branches of $\tau$ and $\tau'$, respectively. More precisely, to define the incidence matrix we fix a regular value $p_i$ of $f$ in the interior of each $b_i$ and let $m_{ij}$ be the number of preimages of $p_i$ in $b_j'$. (Since switches may not map to switches, the incidence matrix may depend on the choice of $p_i$.)

A \emph{measure} on a train track $\tau$ is a function $\mu$ that assigns a weight $\mu(b)\ge 0$ to each branch $b$ of $\tau$ and satisfies the \emph{switch condition} as illustrated in Figure \ref{fig:switch-basics} at every switch of $\tau$. The pair $(\tau,\mu)$ is called a \emph{measured train track}. If $\tau'<\tau$, every measure $\mu'$ on $\tau'$ will induce a measure on $\tau$ where
\[\mu(b_i)=\sum_{j}m_{ij}\mu(b_j').\]

Two measured train tracks are called \emph{equivalent} if one is obtained from the other by a finite sequence of  isotopies and the following moves.
\begin{enumerate}
\item \emph{Split}. $(\tau',\mu')$ is obtained from $(\tau,\mu)$ by splitting a large branch $b$ if it is obtained from $(\tau,\mu)$ as in Figure~\ref{fig:tt-split}. In this case, we write  $(\tau,\mu)\rightharpoonup_b(\tau',\mu')$.
\item \emph{Shift}. $(\tau',\mu')$ is obtained from $(\tau,\mu)$ by a
  shift if it is obtained by sliding one switch past another as in
  Figure~\ref{fig:tt-shift}. Note that the weights are not important in a shift move, so we do not specify them in the figure.
\item \emph{Fold}. This is the inverse of the split move.
\end{enumerate}

\begin{figure}
  \centering
  \includegraphics{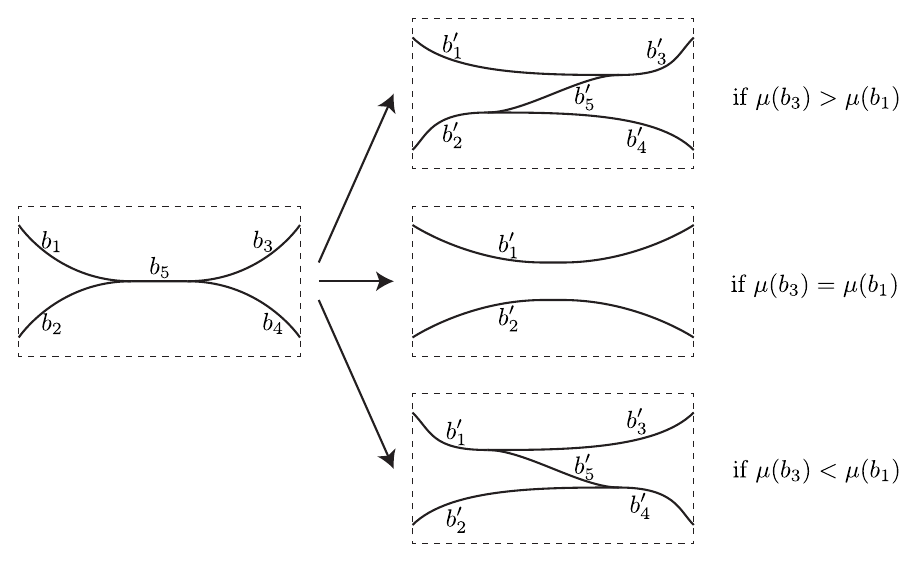}
  \caption[Splitting a train track]{\textbf{A split}. There are three cases, depending on the relative weights of the branches.}
  \label{fig:tt-split}
\end{figure}

\begin{figure}
  \centering
  \includegraphics{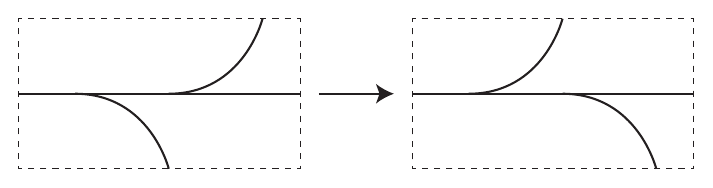}
  \caption[Shifting a train track]{\textbf{A shift}. The weight of the inner edge is the sum of the weights of the two edges to its right; the other weights are unchanged by the move.}
  \label{fig:tt-shift}
\end{figure}

If $\tau'$ is obtained from $\tau$ by a split or a shift then $\tau'$ is carried by $\tau$. In the case of a split, the carrying map can be chosen to send switches to switches.

Following Penner-Harer~\cite[Section 1.7]{PennerHarrer92:book}, a \emph{measured lamination} on $\Sigma$ is a measured foliation of a closed subset of $\Sigma$. (There is a related concept of a measured geodesic lamination, which we will not explicitly use.) The \emph{space of measured laminations} on $\Sigma$ is denoted by $\mathcal{ML}(\Sigma)$. Positive real numbers $\RR_+$ act on $\mathcal{ML}(\Sigma)$ by multiplication and the quotient of $\mathcal{ML}(\Sigma)\setminus\{0\}$ by this action is called the \emph{space of projective measured laminations} on $\Sigma$ and is denoted by $\mathcal{PML}(\Sigma)$. Here, $0$ denotes the empty lamination. Every measured train track $(\tau,\mu)$ specifies a well-defined measured lamination on $\Sigma$, and if two positively measured train tracks are equivalent then their corresponding measured laminations are isotopic \cite[Theorem 2.7.4]{PennerHarrer92:book}. Conversely, if the corresponding measured laminations are isotopic, the measured train tracks are equivalent \cite[Theorem 2.8.5]{PennerHarrer92:book}. That is, there is a bijection between the equivalence classes of measured train tracks and measured laminations. We say a measured lamination $(L,\mu_0)$ on $\Sigma$ is \emph{suited} to the train track $\tau$ if there exists a measure $\mu>0$ on $\tau$ such that $(L,\mu_0)$ is isotopic to the measured lamination specified by $(\tau,\mu)$. If $(\tau',\mu')$ is obtained from $(\tau,\mu)$ by a shift or a split move, then $\tau'$ is carried by $\tau$. So, measured laminations suited to $\tau'$ are suited to $\tau$ as well. 

A train track is called \emph{recurrent} if it supports a positive measure (i.e., a measure satisfying $\mu(b)>0$ for every branch $b$). 

A train track $\tau$ is called \emph{maximal} if 
it is not a proper subtrack of another train track. A \emph{diagonal} for $\tau$ is a smooth arc in one of the complementary regions of $\tau$ whose endpoints terminate tangentially at cusps and such that the union of $\tau$ with this arc is a train track. A train track $\tau'$ is called a \emph{diagonal extension} of $\tau$ if it is obtained from $\tau$ by adding pairwise disjoint diagonals. Note that if a train track $\tau$ is not maximal, one can construct a maximal diagonal extension $\tau'$ of $\tau$, by adding diagonals. This maximal diagonal extension $\tau'$ is not unique (and diagonal extensions are not generic train tracks). 

\begin{lemma}\label{lem:unstableconv} Let $\tau$ be a train track suited to the unstable foliation of $\psi$ and let $\gamma\subset\Sigma$ be a simple closed curve. Then there is an integer $N$, such that for all $n>N$, $\psi^n(\gamma)$ is carried by some maximal diagonal extension of $\tau$ (possibly depending on $n$).
 \end{lemma} 

\begin{proof}
  By \cite[Corollary 12.3]{FLP79:TravauxThurston}, as $n$ goes to infinity, $\psi^n(\gamma)$ converges to  $[\mathcal{F}^u,\mu_u]$, the image of the unstable foliation in the space of projective measured laminations $\mathcal{PML}(\Sigma)$. By \cite[Proposition 1.4.9]{PennerHarrer92:book}, there exists a maximal birecurrent (in the sense of~\cite[Section 1.3]{PennerHarrer92:book}) train track $\tau_c$ with $\tau<\tau_c$ and such that $\tau_c$ is obtained from $\tau$ by a sequence of trivial collapses along admissible arcs, as in \cite[Figure 1.4.14]{PennerHarrer92:book}.  By \cite[Lemma 2.1.2]{PennerHarrer92:book}, the measured laminations corresponding to the measures on $\tau_c$ form a polyhedron $U(\tau_c)$. By \cite[Lemma 3.1.2]{PennerHarrer92:book}, the interior of $U(\tau_c)$ is an open neighborhood of $[\mathcal{F}^u,\mu_u]$ in $\mathcal{PML}(\Sigma)$. So, there exists an integer $N>0$ such that $\psi^n(\gamma)$ is in  $U(\tau_c)$ for all $n>N$.
 
  On the other hand, for any maximal diagonal extension $\tau'$ of $\tau$, the measured laminations corresponding to the measures on $\tau'$ will form a polyhedron $U(\tau')$ in $\mathcal{PML}(\Sigma)$, as well (see~\cite[Theorem 1.3.6, Lemma 2.1.2]{PennerHarrer92:book}). Moreover, \cite[Proposition 2.2.2]{PennerHarrer92:book} implies that $U(\tau_c)=\bigcup_{\tau'}U(\tau')$ where the union is over all maximal diagonal extensions $\tau'$ of $\tau$. Therefore, for all $n>N$, $\psi^n(\gamma)$ is carried by some maximal diagonal extension $\tau'$ of $\tau$ (not necessarily independent of $n$).
\end{proof}

For a measured train track $(\tau,\mu)$ a \emph{maximal split} is splitting $\tau$ simultaneously along all the large branches with maximum measure. If $(\tau',\mu')$ is obtained from $(\tau,\mu)$ by a maximal split, we write $(\tau,\mu)\rightharpoonup (\tau',\mu')$.

\begin{theorem} \cite[Theorem 3.5]{Agol2010IdealTO} \label{thm:Agol-cycle}Let $\psi$ be a pseudo-Anosov diffeomorphism of $\Sigma$ and $(\tau,\mu)$ be a measured train track suited to the unstable measured foliation $(\mathcal{F}^u,\mu_u)$ of $\psi$. Then there exist positive integers $n$ and $m$ such that 
\[(\tau,\mu)\rightharpoonup (\tau_1,\mu_1)\rightharpoonup(\tau_2,\mu_2)\rightharpoonup\cdots\rightharpoonup(\tau_n,\mu_n)\rightharpoonup\cdots\rightharpoonup (\tau_{n+m},\mu_{n+m})\]
and $\tau_{n+m}=\psi(\tau_{n})$ and $\mu_{n+m}=\lambda(\psi)^{-1}\psi(\mu_n)$.
\end{theorem}

A \emph{periodic splitting sequence} for a pseudo-Anosov diffeomorphism $\psi$ is a sequence of train tracks suited to the unstable foliation
formed by maximal splittings 
\[(\tau_n,\mu_n)\rightharpoonup (\tau_{n+1},\mu_{n+1})\rightharpoonup\cdots\rightharpoonup (\tau_{n+m},\mu_{n+m})\]
such that $\tau_{n+m}=\psi(\tau_{n})$ and $\mu_{n+m}=\lambda(\psi)^{-1}\psi(\mu_n)$. A periodic splitting sequence for $\psi$ is unique up to applying powers of $\psi$ and cyclic permutations (changing where in the loop one starts), so we will often abuse terminology and refer to the periodic splitting sequence for $\psi$.

Suppose $(\tau,\mu)$ is a train track from the periodic splitting sequence of some pseudo-Anosov diffeomorphism $\psi$. Then $\psi(\tau)<\tau$ with a carrying map induced by the sequence of maximal splits. Moreover, we may assume that this maps sends switches to switches. 

The following lemma is due to Agol and Tsang~\cite{AgolTsang24:flow-graphs}; it and its proof were communicated to us by Tsang.
\begin{lemma}\cite{AgolTsang24:flow-graphs}\label{lemma:AT24}
  There exists a (possibly empty) collection of branches $b_I(\tau)$ of $\tau$ such that:
\begin{enumerate}[label=(\arabic*)]
  \item\label{item:AT-1} For any sufficiently large integer $n$, the image of $\psi^n(b)$ under the carrying map goes over every branch of $\tau$ if and only if $b\notin b_I(\tau)$. 
  \item\label{item:AT-2} If $b\in b_{I}(\tau)$, then $\psi(b)$ is equal to $b'$ for some branch $b'\in b_I(\tau)$. 
  \item\label{item:AT-3} The union of $b_I(\tau)$ is a collection of disjoint train paths of $\tau$.
\end{enumerate}
\end{lemma}
\begin{proof}
  Agol and Tsang associate a directed graph to the periodic splitting sequence, its \emph{flow graph} $G$, with vertices given by branches in the splitting sequence and, for each split, an edge from the branches coming from the split to the branch that split into them. (When a train track $\tau$ splits to $\tau'$, if $b$ is a branch of $\tau$ not involved in the split, then $b$ and its image in $\tau'$ specify the same vertex of $G$. Also, the branches of the last train track in the periodic splitting sequence are identified with the branches of the first, via $\psi$.) So, given branches $b$ and $b'$ of $\tau$, there is a path in the flow graph from $b$ to $b'$ if and only if the image of $\psi^n(b)$ under the carrying map goes over $b'$ (for some $n$).

  Recall that vertices $v,w$ of a directed graph lie in the same \emph{strongly connected component} if there is a directed path from $v$ to $w$ and one from $w$ to $v$; this defines an equivalence relation $\sim$ on the flow graph. They show that $G/\sim$ is a directed tree with a vertex $v_0$ that has edges to all the other vertices~\cite[Theorem 3.5]{AgolTsang24:flow-graphs}. Let $b_I(\tau)$ be the branches of $\tau$ not corresponding to $v_0$. Property~\ref{item:AT-1} is immediate.

  For Property~\ref{item:AT-2}, ~\cite[Theorem 3.5]{AgolTsang24:flow-graphs} implies that the equivalence classes of  branches in $b_I(\tau)$ correspond to \emph{infinitesimal cycles of walls}. Each of these  infinitesimal cycles is an oriented cycle in the  flow graph with only incoming edges. Moreover, corresponding to each  infinitesimal cycle there are branches $b$ and $b'$ of $b_I(\tau)$ such that the image of $\psi(b)$ under the carrying map is $b'$.
  
  For Property~\ref{item:AT-3}, they show that the unions of branches in $b_I(\tau)$ have particular local forms~\cite[Section 3.3]{AgolTsang24:flow-graphs}. (The local forms are cases I--V in~\cite[Figure 7]{AgolTsang24:flow-graphs}.) It is immediate that the union of branches in $b_I(\tau)$ is a disjoint union of intervals.
\end{proof}

The branches in $b_I(\tau)$ are called \emph{infinitesimal}.

\begin{corollary}\label{cor:almost-PF}
  If we list the branches $b_i$ of $\tau$ such that the elements of $b_I(\tau)$ are listed first, $b_I(\tau)=\cup_{i=1}^kb_i$, then for a sufficiently  large $n$, the incidence matrix of $\psi^n$ has the block form $\left[\begin{smallmatrix}P&N\\0&M\end{smallmatrix}\right]$ where $P$ is a $k\times k$ permutation matrix and $M$ and $N$ have strictly positive entries.
\end{corollary}

\begin{lemma}\label{lem:carry-diag-ext}
  Let $\tau$ be a train track from the periodic splitting sequence for $\psi$.
  For any maximal diagonal extension $\widetilde{\tau}$ of $\tau$, there exists a carrying of $\psi(\widetilde{\tau})$ by some (possibly distinct) maximal diagonal extension $\widetilde{\tau}'$ of $\tau$ such that the restriction of this carrying map to the subtrack $\psi(\tau)\subset \psi(\widetilde{\tau})$ coincides with the carrying map $\psi(\tau)<\tau$.
\end{lemma}

\begin{proof}
The train track $\tau$ is filling, so every connected component of $\Sigma\setminus \tau$ is a polygon or once-punctured polygon. Let $B$ be one of these complementary regions and suppose $B$ is an $n$-gon, so the boundary of $B$ contains $n$ cusps. Denote the corresponding switches by $s_1,s_2,\dots, s_n$, indexed counterclockwise. Denote the $n$ train paths that give the edges of $B$ by $\rho_1,\rho_2,\dots, \rho_n$, where $\rho_i$ connects $s_i$ to $s_{i+1}$ for all $i$ (and $s_{n+1}\coloneqq s_1$).

\begin{figure}
  \includegraphics{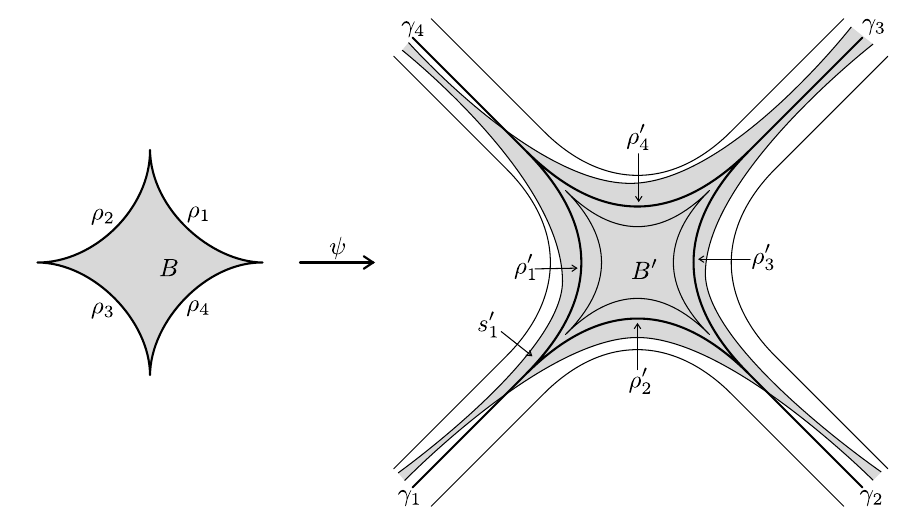}
  \caption[The image of a 4-gon under $\psi$]{\textbf{The image of a 4-gon $B$ under $\psi$}. Thick lines indicate the train track $\tau$. Thin lines indicate the boundary of a fibered neighborhood of $\tau$. The shaded region is $B$ (left) and the image of $B$ (right). The image of $B$ and the train paths $\gamma_i$ extend beyond the picture. (In the drawing, near the singular point $\psi$ sends $B$ by a $135^\circ$ twist.)}
  \label{fig:map-polygon}
\end{figure}

 The diffeomorphism $\psi$ maps each singular point of the unstable foliation (respectively puncture of $\Sigma$) to a singular point (respectively puncture).
 Moreover, since $\psi(\tau)$ is obtained from $\tau$ by a sequence of splits, we may assume the carrying map $\psi(\tau)<\tau$ sends switches to switches.
 Each connected component $B$ of $\Sigma\setminus \tau$ contains either one singular point or one puncture. If $B$ contains a singular point (respectively puncture), then $\psi(B)$ will contain a singular point (respectively puncture). Denote the component of $\Sigma\setminus \tau$ containing the image of that point under $\psi$ by $B'$. Note that $B'$ is also an $n$-gon and might coincide with $B$. The boundary of $B'$ is a union of $n$ train paths $\rho_1', \rho_2',\dots, \rho_n'$, indexed counterclockwise.  Denote the switches corresponding to the cusps on the boundary of $B'$ by $s_1', s_2',\dots, s_n'$ such that $\rho_i'$ connects $s_i'$ to $s_{i+1}'$ (and $s_{n+1}'\coloneqq s_1'$). Then there exist train paths $\gamma_1,\gamma_2,\dots,\gamma_n$ such that $\gamma_i$ ends in the large branch of $s_i'$ and the image of $\psi(\rho_i)$ under the carrying map is equal to $\gamma_{i+k}\circ\rho_{i+k}'\circ\gamma_{i+k+1}^{-1}$, where indices are taken modulo $n$. Here, $0\le k\le n-1$ is a constant, depending on the singular point in $B$. See Figure~\ref{fig:map-polygon}. 

With this description in hand, we are ready to describe $\widetilde{\tau}'$. For each diagonal $d$ in $B$ of $\widetilde{\tau}$ that connects $s_i$ to $s_j$ we add the diagonal $d'$ in $B'$ that connected $s_i'$ to $s_j'$. These new diagonals give the maximal diagonal extension $\widetilde{\tau}'$. It is straightforward that the carrying map $\psi(\tau)<\tau$ extends to a carrying map $\psi(\widetilde{\tau})<\widetilde{\tau}'$. For instance, under this extension a diagonal $d$ in $B$ connecting $s_i$ to $s_j$ is mapped to $\gamma_i^{-1}d'\gamma_j$. 
\end{proof}

Let $\widetilde{\tau}$ be a maximal diagonal extension of $\tau$. By Lemma~\ref{lem:carry-diag-ext}, for any $n>0$, $\psi^{n}(\widetilde{\tau})$ is carried by some maximal diagonal extension $\widetilde{\tau}_n'$ of $\tau$, and the carrying map for $\psi^n(\widetilde{\tau})<\widetilde{\tau}'_n$ extends the carrying $\psi^n(\tau)<\tau$. A diagonal $d$ in $\widetilde{\tau}$ is called \emph{infinitesimal} if, for every $n>0$, the image of $\psi^n(d)$ under this carrying map does not go over any non-infinitesimal branches of $\tau$. In particular, $\psi^n(d)$ under the  carrying map is of the form $\gamma d' \gamma'$ where $d$ is an infinitesimal diagonal in $\widetilde{\tau}_n$ and $\gamma$ and $\gamma'$ are train paths in $\tau$ consisting of only infinitesimal branches. That is, $\gamma$ and $\gamma'$ are subsets of train paths in $b_{I}(\tau)$. Note that, for sufficiently large $n$, the image $\psi^n(d)$ of any non-infinitesimal diagonal $d$ for $\tau$ will go over  every branch of $\tau$ by the carrying map.

Let $l_I(\tau)$ and $d_{I}(\tau)$ denote the number of infinitesimal branches of $\tau$ and diagonals of $\widetilde{\tau}$, respectively.

\begin{lemma}\label{lem:long-paths}
Let $\widetilde{\tau}$ be a maximal diagonal extension of $\tau$. Every train path that goes over  $l_{I}(\tau)+d_I(\tau)+1$ branches of $\widetilde{\tau}$ will contain at  least one non-infinitesimal branch or diagonal.
\end{lemma}
\begin{proof} 
First, we show that every loop carried by $\widetilde{\tau}$ contains at  least one non-infinitesimal branch or diagonal. Let $\gamma$ be a loop carried by $\widetilde{\tau}$, and suppose that $\gamma$ is a union of infinitesimal branches and diagonals. Then, for every  $n>0$, $\psi^{n}(\gamma)$ is carried by some maximal diagonal extension  of $\tau$ and goes over the same number of diagonals as $\gamma$. Thus, by part~\ref{item:AT-3} of Lemma~\ref{lemma:AT24}, $\psi^n(\gamma)$ will go over at most $d_{I}(\tau)(1+l_I(\tau))$  branches of $\widetilde{\tau}$. For any maximal diagonal extension $\widetilde{\tau}$ there are finitely many loops carried by $\widetilde{\tau}$ with such bounded length, and so there are $0<n<n'$ such that $\psi^n(\gamma)$ and $\psi^{n'}(\gamma)$ are isotopic.
But this can not happen, since $\psi^{n'-n}$ is pseudo-Anosov.

\begin{figure}
  \centering
  \includegraphics{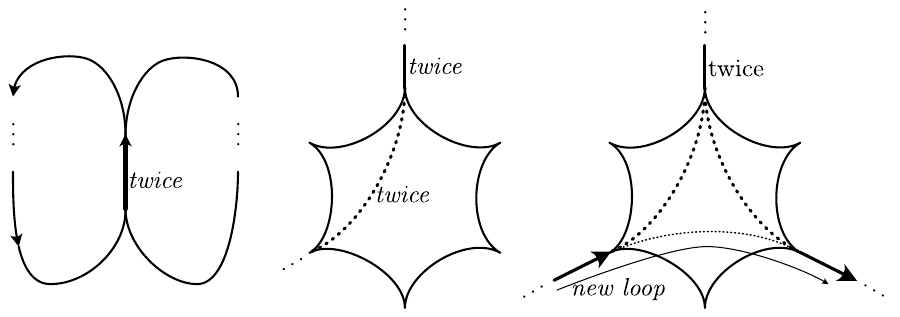}
  \caption{\textbf{Finding a loop.} Left: if $\gamma$ goes over a non-infinitesimal branch twice in the same direction then it contains a loop. Center: if $\gamma$ goes over the same diagonal twice then it goes over the same non-diagonal edge twice. Right: if $\gamma$ goes over two diagonals incident to the same switch and goes over the adjacent branch twice in opposite directions, then $\gamma$ can be short-circuited by a different infinitesimal diagonal to give a loop. In the center and right pictures, dotted edges indicate diagonals and solid edges the original train track $\tau$; the original train path under consideration is thick. In the left picture, edges could either be diagonals or in $\tau$.}
  \label{fig:make-a-loop}
\end{figure}

Let $\gamma$ be a train path that goes over $l_{I}(\tau)+d_I(\tau)+1$ branches of $\widetilde{\tau}$. If $\gamma$ does not contain any non-infinitesimal branch or diagonal then it must go over at least one infinitesimal branch or diagonal twice. By part~\ref{item:AT-3} of Lemma~\ref{lemma:AT24}, if $\gamma$ goes over a diagonal twice then $\gamma$ either contains a loop or goes over an infinitesimal branch of $\tau$ twice, as well, in opposite directions. (See Figure~\ref{fig:make-a-loop}.) If $\gamma$ contains a loop, the previous argument implies that it must contain a non-infinitesimal branch or diagonal. Thus, assume $\gamma$ does not contain a loop, and so $\gamma$ goes over an infinitesimal branch of $\tau$ twice, in opposite directions. Again, by part~\ref{item:AT-3} of Lemma~\ref{lemma:AT24}, $\gamma$ contains two diagonals with one endpoint the same switch of $\tau$. If both of these diagonals are infinitesimal, then the branch or diagonal of $\tau$ that connects the two other end points of these diagonals is infinitesimal. (Again, see Figure~\ref{fig:make-a-loop}.) So, the sub-path of $\gamma$ between these switches along with this branch or diagonal is a loop carried by a maximal diagonal extension of $\tau$. Therefore, it must contain a non-infinitesimal branch or diagonal, and thus $\gamma$ must contain a non-infinitesimal branch or diagonal, as well.
\end{proof}

Associated to any filling train track $\tau$ on $\Sigma$, there is a dual triangulation $\mathcal{T}$ of $\Sigma$, defined so that the dual edge has one vertex in every component of $\Sigma\setminus \tau$ and one edge dual to every branch of $\tau$
 such that it intersects the branch transversely and connects the vertices in the components of  $\Sigma\setminus \tau$ adjacent to the branch. Moreover, if a component of $\Sigma\setminus \tau$ is punctured then the vertex in that component coincides with the puncture. If $\tau'$ is obtained from $\tau$ by a split, its dual triangulation $\mathcal{T}'$ is obtained from $\mathcal{T}$ by a \emph{Whitehead move} as depicted in Figure~\ref{fig:Whitehead-move}.

\begin{figure}
  \centering
  \includegraphics{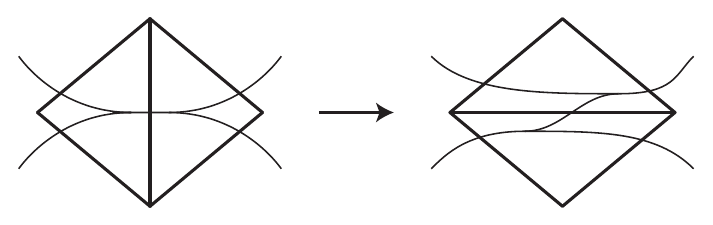}
  \caption[A Whitehead move]{\textbf{A Whitehead move}. The train track is drawn in a
    thin line, and the dual triangulation in a thick line. When the
    train track changes by a split, the dual triangulation changes by
    a Whitehead move.}
  \label{fig:Whitehead-move}
\end{figure}

A \emph{bigon track} is a branched $1$-dimensional submanifold of $\Sigma$ that fails to be a train track because some of its complementary regions are bigons. A bigon track is called \emph{generic} if all of its switches are trivalent. 

\begin{figure}
  \centering
  \includegraphics{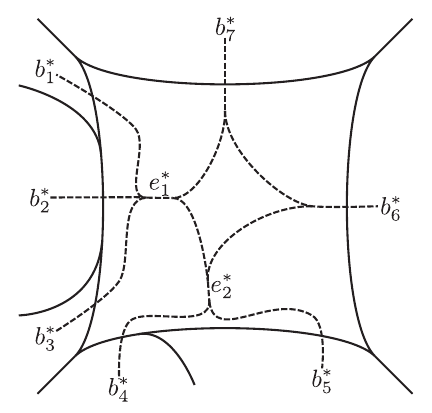}
  \caption[Dual bigon track]{\textbf{Dual bigon track.} A piece of a train track is shown (solid) together the corresponding piece of its dual bigon track (dashed). }
  \label{fig:dual-bigon-track}
\end{figure}

Given a train track $\tau$, the \emph{dual bigon track} $\tau^*$ is defined as follows. For each branch $b_i$ of $\tau$ consider a small arc $b_i^*$ transversely intersecting $b_i$ at one point and disjoint from every $b_j$ with $j\neq i$. Moreover, arrange that the arcs $b_1^*,b_2^*,\dots, b_{l}^*$ are pairwise disjoint. Let $A$ be a connected component of $\Sigma\setminus \tau$. For each edge $e$ of $A$, join the dual arcs to the branches contained in $e$ by a switch in $A$ with the dual arcs on one side and a new branch $e^*$ on the other side. So, for example, if $e=\overline{b_{i_1}}\cup \overline{b_{i_2}}\cup\cdots\cup\overline{b_{i_k}}$, the arcs $b_{i_1}^*, b_{i_2}^*,\dots, b_{i_k}^*$ will merge together and form a switch of valence $k+1$. For any cusp singularity in $\bdy A$ with adjacent edges $e_1$ and $e_2$, connect the new branches $e_1^*$ and $e_2^*$ with a new smooth branch, as in Figure~\ref{fig:dual-bigon-track}. Note that these branches are chosen such that for each switch of $\tau$ there is a bigon region in $\Sigma\setminus\tau^*$ that contains the switch. 
Therefore, if $A$ is a once-punctured $n$-gon the corresponding  $n$-gon component of $\Sigma\setminus \tau^*$ inside $A$ is once-punctured as well.  Note that $\tau^*$ is not necessarily generic. Moreover, $\tau^*$ and $\tau$ intersect \emph{efficiently} i.e.  there is no embedded bigon on $\Sigma$ whose boundary is the union of a smooth arc on $\tau$ and a smooth arc on $\tau^*$. See \cite[Section 3.4]{PennerHarrer92:book} for more detailed discussion of dual bigon tracks.

\begin{lemma}\label{lem:stableconv} Let $\tau$ be a train track suited to the unstable foliation of $\psi$ and let $\gamma\subset\Sigma$ be a simple closed curve. Then there is an integer $N$ such that for all $n>N$, $\psi^{-n}(\gamma)$ is carried by some maximal diagonal extension of the dual bigon track $\tau^*$ (possibly depending on $n$).
\end{lemma} 

\begin{proof}
  By \cite[Corollary 12.3]{FLP79:TravauxThurston}, as $n$ goes to infinity, $\psi^{-n}(\gamma)$ converges to  $[\mathcal{F}^s,\mu_s]$, the image of the stable foliation in the space of projective measured laminations $\mathcal{PML}(\Sigma)$. By 
  \cite[Proposition 1.4.9]{PennerHarrer92:book} one can obtain a maximal birecurrent train track ${\tau}_c$ from $\tau$ by a sequence of trivial collapses along admissible arcs. Then, a corresponding sequence of split moves on ${\tau}_c$ will result in a generic maximal and birecurrent train track $\tau'$ such that $\tau'$ is the result of combing a maximal diagonal extension of $\tau$ \cite[Figure 1.4.2]{PennerHarrer92:book}. 
  
\begin{figure}
  \centering
  \includegraphics[scale=.7]{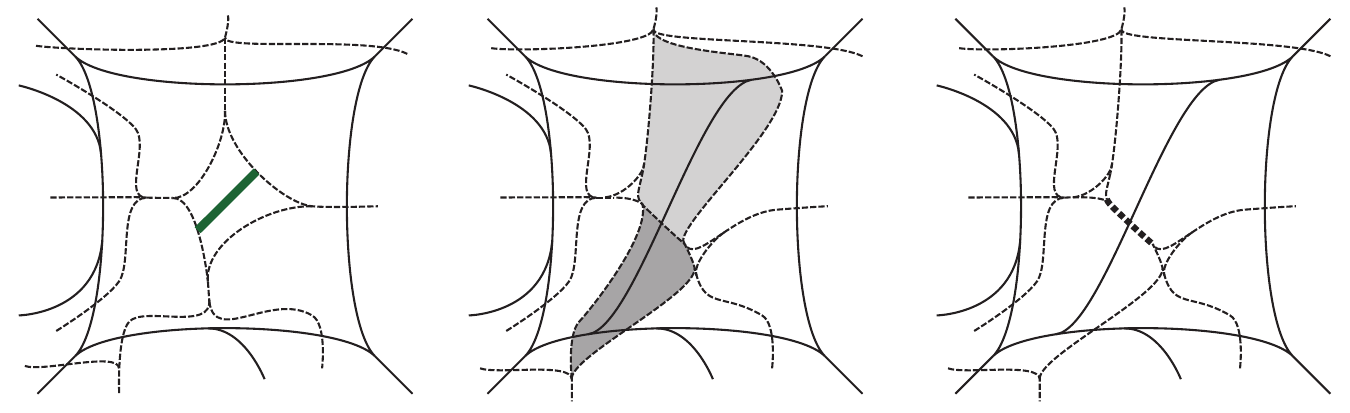}
  \caption{\textbf{Collapsing admissible arcs}. Left: the train track $\tau$ (solid), its dual bigon track $\tau^*$ (dashed), and an admissible arc (\textcolor{darkgreen}{thick}). Center: the train track $\tau'$  obtained from adding a diagonal to $\tau$ and combing it to make the result generic, its dual bigon track $(\tau')^*$. Right: the (dashed) bigon track $\tau^*_c$ obtained by collapsing the two shaded bigons, which coincides with the bigon track obtained from $\tau^*$ by a trivial collapse along the admissible \textcolor{darkgreen}{thick} arc in left figure.}
  \label{fig:trivial-collapse}
\end{figure}

The dual bigon track $(\tau')^*$ is birecurrent and maximal \cite[Proposition 3.4.5]{PennerHarrer92:book}. Collapsing the bigons corresponding to the endpoint switches of the added diagonals results in a bigon track $\tau^*_c$, which can also be obtained from $\tau^*$ by trivial collapses along admissible arcs (see Figure~\ref{fig:trivial-collapse}). Note that the measured laminations corresponding to the measures on 
$(\tau')^*$ and $\tau^*_c$ are the same.  There is a positive measure $\mu^*$ on $\tau^*$ such that $(\tau^*,\mu^*)$ is suited to $[\mathcal{F}^s,\mu_s]$ \cite[Epilogue]{PennerHarrer92:book}. Since $\tau^*<\tau^*_c$, the positive measure $\mu^*$ induces a positive measure $\mu^*_c$ on $\tau^*_c$ so that $(\tau^*_c,\mu^*_c)$ is suited to $[\mathcal{F}^s,\mu_s]$. Therefore, the measured laminations corresponding to the measures on $\tau^*_c$ form a polytope whose interior is an open neighborhood of $[\mathcal{F}^s,\mu_s]$ \cite[Proposition 3.4.1]{PennerHarrer92:book}. Then, as in the proof of Lemma~\ref{lem:unstableconv}, by \cite[Proposition 2.2.2]{PennerHarrer92:book} this neighborhood is the union of measured laminations corresponding to the maximal diagonal extensions of $\tau^*$.  Therefore, there exists an $N>0$ such that for all $n>N$, $\psi^{-n}(\gamma)$ is carried by some maximal diagonal extension of $\tau^*$ (not necessarily independent of $n$). 
\end{proof}

\section{Spheres and compression bodies}\label{sec:compression-bodies}
In this section, we prove four elementary lemmas about decompositions of
3-manifolds, handlebodies, and compression bodies, needed for the
detection results later. The first is a specific version of the prime
factorization, in terms of the Hurewicz homomorphism. The second lemma
uses the first to give a simple criterion for a 3-manifold to be a
handlebody. The
third is a criterion for a collection of disks to generate a compression
body. The last is a generalization of a result of Haken's about how
spheres intersect Heegaard surfaces to the case of compression body
splittings.

\begin{lemma}\label{lem:disj-sphere}
Let $Y$ be a compact, connected, oriented
$3$-manifold. Suppose that the image of the Hurewicz
homomorphism $\pi_2(Y)\to H_2(Y)$ generates an $n$-dimensional subspace
of $H_2(Y;\QQ)$. Then there is an integer $0\leq k\leq n$, a closed 3-manifold
$Y'$, and 3-manifolds $Y_1,\dots,Y_{n-k+1}$ with non-empty boundary so that
\[
  Y=\begin{cases}
      Y'\#\overbrace{(S^2\times S^1)\#\cdots\# (S^2\times S^1)}^k\#Y_1\#\cdots\#Y_{n-k+1}
      & \text{if }\bdy Y\neq\emptyset\\
      Y'\#\overbrace{(S^2\times S^1)\#\cdots\# (S^2\times S^1)}^n
      & \text{if }\bdy Y = \emptyset.
    \end{cases}
\]
Further, $Y'$ and every $Y_i$ does not contain any homologically
essential $2$-spheres.
\end{lemma}
\begin{proof}
  Choose a prime decomposition
  \begin{equation}\label{eq:decompose-Y}
    Y=Y'_1\#\cdots\#Y'_j\#\overbrace{(S^2\times S^1)\#\cdots\# (S^2\times S^1)}^k\#Y_1\#\cdots\#Y_{\ell}
  \end{equation}
  of $Y$, where each $Y_i$ and $Y'_i$ is irreducible (every
  sphere bounds a ball or equivalently, by the Sphere Theorem, each $Y_i$ and $Y'_i$ has
  trivial $\pi_2$), each $Y_i$ has nonempty boundary, and each $Y'_i$ is closed.
  If $S$ is a sphere in $Y$ then, intersecting $S$ with the connect
  sum spheres in Formula~\eqref{eq:decompose-Y} and using an innermost
  disk argument, $[S]\in H_2(Y)$ is a linear combination of spheres in
  the $Y'_i\setminus B^3$, the copies of $(S^2\times S^1)\setminus B^3$,
  and the $Y_i\setminus B^3$. Since $Y'_i$ is closed, from the diagram
  \[
    \begin{tikzcd}
      \pi_2(Y_i'\setminus B^3) \arrow[r]\arrow[d] & \pi_2(Y_i')=0\arrow[d]\\
      H_2(Y_i'\setminus B^3)\arrow[r,"\cong"] & H_2(Y_i')
    \end{tikzcd}
  \]
  the image of the Hurewicz map to $H_2(Y_i'\setminus B^3)$
  vanishes. Similarly, the image of the Hurewicz map to $(S^2\setminus
  S^1)\setminus B^3$ is generated by $[S^2\times\{pt\}]$, and for
  $Y_i\setminus B^3$, the analogous diagram
  \[
    \begin{tikzcd}
    & \pi_2(S^2) \arrow[r]\arrow[d,"\cong"] & \pi_2(Y_i\setminus B^3)\arrow[r]\arrow[d]
    & \pi_2(Y_i)=0\arrow[d]\\
    0\arrow[r] &\ZZ\arrow[r] & H_2(Y_i\setminus B^3)\arrow[r] & H_2(Y_i)
    \end{tikzcd}
  \]
  the image of the Hurewicz map to $H_2(Y'_i\setminus B^3)$ is
  generated by the boundary sphere. So, $[S]$ is a linear combination
  of the spheres $[S^2\times\{pt\}]$ for the $S^2\times S^1$ factors
  and the boundary spheres for the $Y_i$ factors. Further, the sum of
  all the boundary spheres vanishes in $H_2(Y)$. So,  $k+\ell=n+1$,
  and letting $Y'=Y'_1\#\cdots\#Y'_j$ gives the desired factorization.
  
  The proof for the closed case is the same, except that $\ell=0$ and $k=n$.
\end{proof}

A \emph{homology handlebody of genus $g$} is a compact, connected,
orientable $3$-manifold $Y$ with boundary a connected surface of genus
$g$ so that the map $H_1(\bdy Y)\to H_1(Y)$ is surjective. (We will
generally assume that $g>0$. Also, $H_1(\bdy Y)\to H_1(Y)$ being
surjective implies that $H_1(Y,\bdy Y)=0$ so, since the torsion
subgroup of $H_1(Y)=H^2(Y,\bdy Y)$ is $\Ext^1(H_1(Y,\bdy Y),\ZZ)$,
this implies that $H_1(Y)$ is free.)

We can use Lemma~\ref{lem:disj-sphere} to give a criterion for
recognizing handlebodies among homology handlebodies, which is key to
the proof of Theorem~\ref{thm:detect-hb}:
\begin{lemma}\label{lem:double-handlebody}
  Let $Y$ be an irreducible homology handlebody of genus $g$. Then $Y$
  is a handlebody of genus $g$ if and only if the homology
  $H_2(D(Y);\QQ)$ of the double $D(Y)=Y\cup_\bdy (-Y)$ is generated by
  $2$-spheres.
\end{lemma}
\begin{proof}
  If $Y$ is a handlebody then certainly $H_2(D(Y))$ is generated by
  $2$-spheres.
  For the converse, suppose $H_2(D(Y);\QQ)$ is
  generated by $2$-spheres. By Lemma~\ref{lem:disj-sphere}, $H_2(D(Y);\QQ)$ is
  generated by disjoint, embedded $2$-spheres. We first reduce to the case that $Y$ is not a boundary sum of two manifolds. If $Y=Y_1\#_bY_2$ then $D(Y)=D(Y_1)\# D(Y_2)$. It follows from the proof of Lemma~\ref{lem:disj-sphere} that both $H_2(D(Y_1);\QQ)$ and $H_2(D(Y_2);\QQ)$ are generated by $2$-spheres. Moreover, if $Y_1$ and $Y_2$ are handlebodies, then $Y$ is also a handlebody.

  So, assume that $Y$ is not a boundary sum. Let $\Sigma=\bdy Y$. Consider an embedded $2$-sphere $S$ in $DY$ that is a generator of $H_2(D(Y);\QQ)$. It follows from the irreducibility of $Y$ that $S$ intersects $\Sigma$. Further, we may arrange for $S$ to intersect $\Sigma$ minimally, i.e., that none of the disks in $S\cap Y$ or $S\cap -Y$ are homotopic relative boundary to disks in $\Sigma$. Choose a disk $D$ in $S\cap Y$ or $S\cap -Y$. Since $Y$ is not a boundary sum, $D$ is homologically essential. Since  $Y$ is a homology handlebody, the map $H^1(Y)\to H^1(\bdy Y)$ is injective, so there is a circle in $\bdy Y$ intersecting $\bdy D$ in a single point. This gives a decomposition of $Y$ as the boundary sum of a solid torus with a $3$-manifold $Y'$. Hence, $Y'$ is a $3$-ball and $Y$ is a solid torus.
\end{proof}

Before giving the next two results, we introduce some terminology
related to compression bodies. Let $\Sigma$ be a closed, orientable
surface of genus $g$. Given disjoint curves
$\gamma_1,\dots,\gamma_n\subset \Sigma$ there is an associated
3-manifold $\Sigma[\gamma_1,\dots,\gamma_n]$ obtained from
$[0,1]\times\Sigma$ by attaching 2-handles along the
$\gamma_i\times\{1\}$ and filling any $S^2$ boundary components of the
result with 3-balls. A \emph{compression body} $C$ is a manifold homeomorphic to some $\Sigma[\gamma_1,\dots,\gamma_n]$.
We refer to (the image of) the boundary component
$\{1\}\times\Sigma$ of the result as the \emph{outer boundary} $\bdy_{\out}C$ and the
remaining boundary components as the \emph{inner boundary} $\bdy_{\mathit{in}}C$. Since we do not
require the attaching circles for the 2-handles to be homologically
linearly independent, the inner boundary may not be connected. A
\emph{handlebody} is the special case that the inner boundary is
empty. 
A \emph{basis}
for $C$ is a set of pairwise disjoint simple closed curves
$\{\gamma_1,\cdots,\gamma_{n}\}$ on $\bdy_{\mathit{out}}C$ so that $C\cong (\bdy_{\mathit{out}}C)[\gamma_1,\cdots,\gamma_{n}]$.

Since we are interested in bordered Floer theory, we will be
interested in compression bodies whose boundaries are parameterized by
surfaces associated to pointed matched circles or arc diagrams. Define a
\emph{half-bordered compression body} to be a compression body $C$
together with a diffeomorphism $\phi$ from a reference surface $\Sigma$ to
$\bdy_{\mathit{out}}C$.  An essential simple closed curve
$\gamma\subset \Sigma$ is a \emph{meridian} for $C$ if $\phi(\gamma)$
bounds a disk in $C$.

A \emph{compression body splitting} of a 3-manifold $Y$
is a decomposition $Y=C\cup_{\Sigma} C'$ as a union of two compression bodies
glued along their outer boundaries $\Sigma$.

The following lemma gives a criterion for when a set of meridians for a
compression body is large enough to determine the compression body; this
will be used in detecting when diffeomorphisms extend over specific
compression bodies.
\begin{lemma} \label{lem:max-set} Let $(C,\phi\co \Sigma\to\bdy_{\mathit{out}}C)$ be a half-bordered compression body. Let $\amalg_{i=1}^n\gamma_i\subset \Sigma$ be a collection of pairwise disjoint meridians for $C$ and consider pairwise disjoint, properly embedded disks $\amalg_{i=1}^n D_i\subset C$ such that $\bdy D_i=\phi(\gamma_i)$ for all $i$. Then $\Sigma[\gamma_1,\gamma_2,\dots,\gamma_n]\cong C$ if and only if the homology classes $[D_1], [D_2], \dots, [D_n]$ generate $H_2(C,\bdy_{\mathit{out}}(C))$. 
\end{lemma}
\begin{proof}
First, it is easy to see that the homology classes of the cores of the attached $2$-handles in $\Sigma[\gamma_1,\gamma_2,\dots,\gamma_n]$ generate $H_2(\Sigma[\gamma_1,\gamma_2,\dots,\gamma_n], \Sigma)$, and so if $\Sigma[\gamma_1,\gamma_2,\dots,\gamma_n]\cong C$ then $[D_1], [D_2], \dots, [D_n]$ generate $H_2(C,\bdy_{\mathit{out}}(C))$.

On the other hand, suppose $[D_1], [D_2], \dots, [D_n]$ generate $H_2(C,\bdy_{\mathit{out}}(C))$. If we show that $\amalg_{i=1}^n\gamma_i$ is a maximal set of meridians for $C$, then the claim follows from \cite[Lemma 2.1]{BiringerVlamis17} (which states that attaching 2-handles along any maximal set of meridians for a compression body $C$ gives $C$). Let $\gamma\subset \Sigma$ be a simple closed curve disjoint from $\amalg_{i=1}^n\gamma_i$ such that $\phi(\gamma)=\bdy D$ for a properly embedded disk $D$ in $C$. The homology class $[D]\in H_2(C,\bdy_{\mathit{out}}(C))$ is equal to a linear combination of $[D_1], [D_2], \dots, [D_n]$ and so $\nbd(\bdy_{\mathit{out}}(C))\cup(\amalg_{i=1}^n \nbd(D_i))\cup \nbd(D)$ has an $S^2$ boundary component that intersects $\nbd(D)$ nontrivially. Consequently, $\gamma$ is homotopic to a homotopically trivial curve in the inner boundary of $\Sigma[\gamma_1,\gamma_2,\cdots,\gamma_n]$ and so $\amalg_{i=1}^n\gamma_i$  is maximal.
\end{proof}

The following result shows that prime decompositions can be chosen to
be compatible with compression body splittings. The case of
Heegaard splittings (handlebodies) is due to Haken~\cite{Haken68:surfaces} (but we
learned it from Ghiggini-Lisca~\cite[Lemma 3.4]{GL15:OB-prime}):
\begin{lemma}\label{lem:cb-Haken}
  Let $Y=C\cup_{\Sigma} C'$ be a compression body splitting of a
  $3$-manifold $Y$ with prime factorization of the form
  \[
    Y=Y_1\#\cdots\# Y_k\#(S^2\times S^1)^{l}\#Y'_1\#\cdots\#Y'_{h},
  \]
  where each $Y_i$ has nonempty boundary and every
  $Y'_j$ is closed.  Moreover, assume that
  $\rank(H_2(Y)/i_*(H_2(\bdy Y)))=l$, where $i$ denotes the inclusion
  map.  Then there exist pairwise disjoint, embedded $2$-spheres
  $S_1,\dots, S_{k+l-1}$ in $Y$ satisfying the following conditions:
  \begin{enumerate}[label=(\arabic*)]
  \item Each $S_i$ intersects the surface $\Sigma$ in a single circle $C_i$.
  \item\label{item:Haken-lin-indep} The homology classes
    $[S_1], [S_2],\dots,[S_{k+l-1}]$ are linearly independent in
    $H_2(Y)$.
  \item\label{item:Haken-span} The homology classes
    $[S_1],[S_2],\dots,[S_l]$ span $H_2(Y;\QQ)/i_*(H_2(\bdy Y;\QQ))$.
  \end{enumerate}
\end{lemma}
(In the case that $Y$ is closed, Condition~\ref{item:Haken-lin-indep}
is redundant.)
 \begin{proof}
 The prime factorization of $Y$ implies that for $n=k+l-1$ there exist
 embedded $2$-spheres $S_1,S_2,\dots, S_{n}$ in $Y$ such that
 conditions~\ref{item:Haken-lin-indep} and~\ref{item:Haken-span} hold for them. Following an analogous argument
 of Haken's~\cite[pp. 84--86]{Haken68:surfaces}, we show that one can
 construct a collection of pairwise disjoint, essential embedded $2$-spheres $S'_1,S'_2,\dots, S'_{n'}$ so that $n'\ge n$, each $S'_i$ intersects $\Sigma$ in a single circle and the subspace of $H_2(Y)$ generated by the homology classes $[S'_1], \dots,[S'_{n'}]$ contains the subspace generated by $[S_1], [S_2],\dots, [S_n]$. As a result, we can find $\{j_1,\dots, j_n\}\subset\{1,2,\dots,n'\}$ such that the homology classes $[S'_{j_1}], [S_{j_2}'], \dots, [S'_{j_l}]$ span $H_2(Y;\QQ)/i_*(H_2(\bdy Y;\QQ))$ and $[S'_{j_1}], [S'_{j_2}],\dots, [S'_{j_n}]$ are linearly independent in $H_2(Y)$, and we are done.

\emph{Step 1.} We show that the spheres can be changed via isotopy so that every connected component of $S_i\cap C'$ is a disk representing a nontrivial homology class in $H_2(C',\bdy_{\mathit{out}}C')$ for all $1\le i\le n$. If $\bdy_{\mathit{in}}C'=\emptyset$, then $C'$ is a handlebody and this is Step $1$ of Haken's proof. Suppose $\bdy_{\mathit{in}}C'\neq\emptyset$. Consider pairwise disjoint, properly embedded arcs $\gamma_1,\dots,\gamma_m$ in $C'$ so that $\bdy\gamma_i\subset \bdy_{\mathit{in}}C'$ for all $i$ and the homology classes $[\gamma_1],\dots,[\gamma_m]$ form a basis for $H_1(C',\bdy_{\mathit{in}}C')$. Moreover, assume $S_i$ and $\gamma_j$ intersect transversely for all $i,j$. Let $C''\subset C'$ be the compression body defined as a small neighborhood of $\bdy_{\mathit{in}}C'\cup\gamma_1\cup\cdots\cup\gamma_m$, so that every component of $S_i\cap C''$ is a disk,
for all $i$. There is an ambient isotopy $h$ that maps $C''$ to $C'$.  Let $S_{1,i}=h(S_i)$ for $1\le i\le n$.

Denote the number of connected components in $\left(\coprod_{i=1}^nS_{1,i}\right)\cap\Sigma$ by $c_1$. 

\emph{Step 2.} We transform the spheres $S_{1,1},\dots, S_{1,n}$ from
Step 1 into a collection of pairwise disjoint embedded spheres $S_{2,1},\dots, S_{2,n'}$ so that they still satisfy conditions (2) and (3),   $\left(\coprod_{i=1}^nS_{2,i}\right)\cap C$ is incompressible in $C$ and 
\[\left(\coprod_{i=1}^{n'}S_{2,i}\right)\cap \Sigma\subset\left(\coprod_{i=1}^nS_{1,i}\right)\cap \Sigma. \]
If we define $c_2$ to be the number of connected components of $\left(\coprod_{i=1}^{n'}S_{2,i}\right)\cap\Sigma$, the last condition is $c_2\le c_1$.

If $\left(\coprod_{i=1}^nS_{1,i}\right)\cap C$ is incompressible in
$C$, then let $n=n'$ and $S_{2,i}=S_{1,i}$ for all $i$. So, suppose
$\left(\coprod_{i=1}^nS_{1,i}\right)\cap C$ is not incompressible. Let
$\gamma\subset (S_{1,i}\cap C)$ be a circle that bounds a compressing
disk $D$ in $C$. By an innermost disk argument, we can assume that the interior of $D$ is disjoint from the spheres $S_{1,i}$. 
Remove a small tubular neighborhood $\nbd(\gamma)$ of
$\gamma$ from $S_{1,i}$ and add two parallel copies of $D$ to $\bdy
\nbd(\gamma)$. Denote the resulting spheres by $S_{11,i}$ and
$S_{12,i}$.  If one of $S_{11,i}$ or $S_{12,i}$, say $S_{11,i}$,
bounds a ball in $Y$, then replace $S_{1,i}$ with $S_{12,i}$. Otherwise, if both $S_{11,i}$ and $S_{12,i}$ are essential in $Y$, then replace $S_{1,i}$ with the two spheres $S_{11,i}$ and $S_{12,i}$. Repeat this process until the intersection of our spheres with $C$ is incompressible in $C$. Denote the resulting spheres by $S_{2,1},\dots, S_{2,n'}$. It is obvious that
\[\left\langle [S_{1,1}],\dots, [S_{1,n}]\right\rangle \subset \left\langle [S_{2,1}],\dots,[S_{2,n'}]\right\rangle.\]

Before going to the next step, we introduce some more notation.
Let
\[a_2=c_2-\left|\left(\coprod_{i=1}^{n'}S_{2,i}\right)\cap C\right|\]
and let $b_2$ be $a_2$ minus the number of non-disk components in
$\left(\coprod_{i=1}^{n'}S_{2,i}\right)\cap C$. Then $a_2\ge 0$,
because otherwise one of the spheres $S_{2,i}$ is disjoint from
$\Sigma$ and lies in $C$, which contradicts the incompressibility of
$S_{2,i}$. So, if $a_2=0$, then all connected components of
$S_{2,i}\cap C$ are disks,
and we are done. If $a_2\neq 0$ then $a_2>b_2\ge 0$. 
 
\emph{Step 3.} Suppose $a_2>0$. We use an isotopy to transform $S_{2,1},\dots, S_{2,n'}$ into $S_{3,1},\dots, S_{3,n'}$ such that $c_3<c_2$ and $a_3=b_3=0$, where $a_3, b_3$ and $c_3$ are defined similar to $a_2, b_2$ and $c_2$. However, the connected components of $\left(\coprod_{i=1}^{n'}S_{3,i}\right)  \cap C'$ might no longer be disks. 

Consider a set of pairwise disjoint, properly embedded disks $D_1, D_2,\dots, D_m$ in $C$ such that $\bdy D_i\subset\bdy_{\mathit{out}}C$ and 
\[
  \begin{cases}
    C\setminus\left(\coprod_{i=1}^m\nbd(D_i)\right)\cong
    [0,1]\times\bdy_{\mathit{in}}C & \text{if }\bdy_{\mathit{in}}C\neq\emptyset\\
    B^3 & \text{if } \bdy_{\mathit{in}}C=\emptyset.
  \end{cases}
\]
Moreover, assume every $D_i$ intersects all the spheres $S_{2,j}$
transversely. So, $D_i\cap S_{2,j}$ is a collection of arcs and
circles. First, every disk $D_i$ can be transformed via isotopies so
that $D_i\cap\left(\coprod_{j=1}^{n'}S_{2,j}\right)$ has no circle
components, as follows. Consider an innermost circle component
$\gamma$ on $D_i$, i.e., a circle $\gamma$ which bounds a disk $a$ in $D_i$ disjoint from $\coprod_{j=1}^{n'}S_{2,j}$. Then the incompressibility of $C\cap\left(\coprod_{j=1}^{n'}S_{2,j}\right)$ in $C$ implies that $\gamma$ bounds a disk $b$ in $\coprod_{j=1}^{n'}S_{2,j}$. The union of $a$ and $b$ is an embedded $2$-sphere in $C$ and so bounds a $3$-ball. Pushing $D_i$ (and possibly other disks that have nonempty intersection with this $3$-ball) via an isotopy through this $3$-ball will remove the intersection circle $\gamma$. Repeat this process until every connected component of $D_i\cap S_{2,j}$ is an arc, for all $i$ and $j$. 

Second, say that an arc $\gamma\subset D_i\cap S_{2,j}$ \emph{splits off
a disk} from $S_{2,j}\cap C$ if one of the components of $(S_{2,j}\cap
C)\setminus \gamma$ is a disk $D$ with $\gamma$ on its boundary. We
transform our collection of embedded disks $D_1,\dots, D_m$
such that all the aforementioned properties hold, and none of the
arc components in $D_i\cap S_{2,j}$ splits off a disk from
$S_{2,j}\cap C$, for every $i$ and $j$. Consider an
intersection arc $\gamma$ in $S_{2,j}\cap C$ that splits off a disk $D$
and which is innermost in the sense that
the interior of $D$ is disjoint from $\coprod_{i=1}^m D_i$. Remove a small
neighborhood $\nbd(\gamma)$ of $\gamma$ from $D_i$ and add two
parallel copies of $D$ to $D_i\setminus\nbd(\gamma)$ to construct two
properly embedded disks $D_i'$ and $D_i''$ with boundary on
$\bdy_{\mathit{out}}(C)$ in $C$. After small isotopies we may assume
$D_i$, $D_i'$ and $D_i''$ are pairwise disjoint. Since
$C\setminus\left(\coprod_{i=1}^m\nbd(D_i)\right)$ is a product, the
disks $D_i'$ and $D_i''$ are boundary-parallel in
$C\setminus\left(\coprod_{i=1}^m\nbd(D_i)\right)$, so there is a
$3$-ball $B_i'$ (respectively $B_i''$) with part of its boundary $D_i'$
(respectively $D_i''$) and the rest on the boundary of $C\setminus\left(\coprod_{i=1}^m\nbd(D_i)\right)$. After an isotopy, we may assume
that the $3$-balls $B_i'$ and $B_i''$ are either disjoint or one is
contained in the other; but if $B_i'\cap B_i''=\emptyset$ then their
union together with the region between $D_i$ and $D'_i\cup D''_i$ is a
$3$-ball, showing that $D_i$ is boundary parallel, a contradiction. So,
without loss of generality, we may assume that $B_i''\supset B_i'$.
Then it is easy to check that replacing $D_i$ with $D_i''$ will result
in a set of disks that still splits $C$ into a $3$-manifold homeomorphic
to $[0,1]\times \bdy_{\mathit{in}}C$ or into $B^3$.
Repeat this process until none of the arc components in $D_i\cap
S_{2,j}$ splits off a disk from $S_{2,j}\cap C$, for every $i$ and
$j$.

Third, we remove the rest of the intersection arcs by isotoping $\coprod_{i=1}^{n'}S_{2,i}$ as follows. Consider an innermost intersection arc $\gamma\subset D_i\cap \left(\coprod_{i=1}^{n'}S_{2,i}\right)$ on $D_i$, in the sense that one of the two components in $D_i\setminus \gamma$, denoted by $D_i'$, is disjoint from $\coprod_{i=1}^{n'}S_{2,i}$. Suppose $\gamma\subset S_{2,j}$. Then transform $S_{2,j}$ by an isotopy that pushes $S_{2,j}$ along $D_i'$ through $\Sigma$ and removes the intersection arc $\gamma$. Depending on $\gamma$ one of the following happens:
\begin{itemize}
\item If $\gamma$ connects distinct components of $\Sigma\cap\left(\coprod_{i=1}^{n'}S_{2,j}\right)$, then $c_2$ and $a_2$ decrease by $1$, and $b_2$ does not increase.
\item If $\bdy \gamma$ lies on a single component of $\Sigma\cap\left(\coprod_{i=1}^{n'}S_{2,j}\right)$, then $c_2$ increases by $1$, $a_2$ does not change, and $b_2$ decreases by $1$. 
\end{itemize}
Repeat this process until we get a collection of embedded $2$-spheres $S_{2,1}', S_{2,2}',\dots, S_{2,n'}'$ disjoint from $\coprod_{i=1}^mD_i$. 
If $\bdy_{\mathrm{in}}C=\emptyset$, so $C$ is a handlebody, incompressibility of $C\cap\left(\coprod_{i=1}^{n'}S_{2,i}'\right)$ in $C$ implies that every connected component of $S'_{2,j}\cap C$ is a disk, and we define $S_{3,j}=S_{2,j}'$. Then $a_3=b_3=0$ and since $a_2>b_2$, we have $c_3<c_2$. 

Suppose $C$ is not a handlebody. If every connected component of $C\cap\left(\coprod_{i=1}^{n'}S'_{2,i}\right)$ is a disk then again we let $S_{3,j}=S_{2,j}'$ for all $1\le j\le n'$, and as before $a_3=b_3=0$ and $c_3<c_2$. Therefore, assume $C\cap\left(\coprod_{i=1}^{n'}S'_{2,i}\right)$ contains non-disk components. Waldhausen showed that any incompressible surface in a product is parallel to the
boundary. That is, if $F$ is a closed surface and $(S,\bdy
S)\subset(F\times[0,1],F\times\{0\})$ then $S$ is isotopic,
relative boundary, to a subset of $F\times\{0\}$~\cite[Corollary
3.2]{Waldhausen68}. Thus, every non-disk component of
$C\cap\left(\coprod_{j=1}^{n'}S_{2,j}'\right)$ is parallel to
$\bdy_{\mathrm{out}}C$. Let $S\subset S'_{2,j}$ be a non-disk,
innermost component of $C\cap\left(\coprod_{j=1}^{n'}S_{2,j}'\right)$,
i.e., the interior of the component of $C\setminus S$ bounded between
$S$ and $\bdy_{\mathrm{out}}C$ is disjoint from
$\coprod_{j=1}^{n'}S_{2,j}'$. Assume $|\bdy S|=r$ and consider arcs
$\gamma_1,\dots, \gamma_{r-1}$ on $S$ so that
$S\setminus\left(\coprod_{i=1}^{r-1}\nbd(\gamma_i)\right)$ is a
disk. Under the isotopy that transforms $S$ to a subset of
$\bdy_{\mathrm{out}}C$ each $\gamma_i$ would traverse a disk
$D'_i$. Push $S$ along the disks $D'_1, D'_2,\dots, D'_{r-1}$ and through
$\Sigma$. This operation reduces $c_2$ and $a_2$ by
$r-1$, while reducing $b_2$ by $r-2$. Repeat this process until every
connected component of $C\cap\left(\coprod_{i=1}^{n'}S'_{2,i}\right)$ is a disk. Then we let $S_{3,j}=S_{2,j}'$. As before, $a_3=b_3=0$ and $c_3<c_2$.

\emph{Step 4.} Repeat Steps 2 and 3
by exchanging the role of $C$ and $C'$ to obtain a collection of essential spheres $S_{4,1},S_{4,2},\dots, S_{4,n''}$ such that every connected component of $S_{4,j}\cap C'$ is a disk for all $j$ and $c_4=n''$ or $c_4<c_3$.  

\emph{Step 5.} Repeat Steps 2, 3, and 4 until we are done.  
\end{proof}

\section{Invariants}\label{sec:HF-background} 

In this section, we introduce twisted
bordered-sutured Floer homology, and also collect some properties of
twisted sutured Floer homology that we need later in the paper. Twisted
bordered-sutured Floer homology is a relatively straightforward
adaptation of twisted $\HFa$, so we review that theory first. While
discussing twisted bordered-sutured Floer homology, we also introduce
the notion of special bordered-sutured manifolds and some related
constructions that make sense in both the twisted and untwisted setting.
The section assumes some familiarity with bordered Floer homology, but
not with bordered-sutured Floer homology.

\subsection{Twisted Heegaard Floer homology}\label{sec:twisted-HF} Here,
we recall briefly the construction and key properties of twisted
Heegaard Floer homology. Throughout, we will work with
$\FF_2$-coefficients, as that suffices for the applications in this
paper.

Given a 3-manifold $Y$, the
\emph{totally twisted Heegaard Floer complex} of $Y$ is a chain
complex $\tCFa(Y)$ over the group ring $\FF_2[H_2(Y)]$~\cite[Section
8]{OS04:HolDiskProperties}. If we fix an isomorphism
$H_2(Y)\cong \ZZ^n$ then there is an induced identification
$\FF_2[H_2(Y)]\cong \FF_2[x_1^{\pm 1},\dots,x_n^{\pm 1}]$. Given any
other module $M$ over $\FF_2[H_2(Y)]$, the tensor product
$\tCFa(Y)\otimes_{\FF_2[H_2(Y)]} M=\tCFa(Y;M)$ is the \emph{twisted
  Heegaard Floer complex with coefficients in $M$}. The homologies of
$\tCFa(Y)$ and $\tCFa(Y;M)$ are denoted $\tHFa(Y)$ and
$\tHFa(Y;M)$. Like untwisted Heegaard Floer homology, the twisted
Heegaard Floer complex decomposes as a direct sum along
$\SpinC$-structures on $Y$.

To construct $\tCFa(Y)$, one fixes a pointed Heegaard diagram
$\HD=(\Sigma,\alphas,\betas,z)$, a sufficiently generic almost complex
structure, a base generator $\x_0\in T_\alpha\cap T_\beta$ for
$\CFa(\HD)$ for each $\SpinC$-structure, and for each other generator
$\x$ representing that $\SpinC$-structure a homotopy class of disks
$B_\x\in\pi_2(\x_0,\x)$. Recall that
$\pi_2(\x_0,\x_0)\cong\ZZ\oplus H_2(Y)$; given $B\in \pi_2(\x_0,\x_0)$ let
$[B]$ denote its image in $H_2(Y)$. (We are implicitly assuming that
$\Sigma$ has genus at least $2$; in the genus $1$ case, there is a map
to $\ZZ\oplus H_2(Y)$, but it is not surjective; this makes no
difference for the constructions below. In the cylindrical formulation
of Heegaard Floer homology, the isomorphism holds in any genus.) Then
$\tCFa(Y)$ is freely generated over $\FF_2[H_2(Y)]$ by
$T_\alpha\cap T_\beta$ and the differential is given by
\begin{equation}\label{eq:tHF-diff}
  \bdy(\x)=\sum_{\y\in T_\alpha\cap
    T_\beta}\sum_{\substack{B\in\pi_2(\x,\y)\\\mu(B)=1,\ n_z(B)=0}}\bigl(\#\cM^B\bigr)e^{[B_\x*B*B_\y^{-1}]}\y 
\end{equation}
where we are writing $e^h$, for $h\in H_2(Y)$, to denote the
corresponding group ring element and $\cM^B$ is the moduli space of
holomorphic disks in $(\Sym^g(\Sigma),T_\alpha,T_\beta)$ in the
homotopy class $B$. One then shows that the homotopy type of $\tCFa(Y)$
is independent of the choices made in its construction.

Twisted Floer homology has a number of useful properties, including:
\begin{enumerate}
\item K\"unneth Theorem:
  $\tCFa(Y_1\# Y_2)\simeq \tCFa(Y_1)\otimes_{\FF_2}\tCFa(Y_2)$. More
  generally, for modules $M_i$ over $\FF_2[H_2(Y_i)]$,
  $\tCFa(Y_1\#Y_2;M_1\otimes M_2)\simeq
  \tCFa(Y_1;M_1)\otimes_{\FF_2}\tCFa(Y_2;M_2)$. This is immediate from the
  definition, if one takes the connected sum near the basepoints.
\item Non-vanishing Theorem: Let $\omega\co H_2(Y)\to\RR$. The map
  $\omega$ makes the universal Novikov field
  $\Lambda=\{\sum_{i=1}^\infty n_ie^{a_i}\mid n_i\in\FF_2,\ a_i\in\RR,\ 
  \lim_{i\to\infty}a_i=\infty\}$ into an algebra $\Lambda_\omega$ over
  $\FF_2[H_2(Y)]$. Then $\tHFa(Y;\Lambda_\omega)=0$ if and only if
  there is a 2-sphere $S\subset Y$ so that $\omega([S])\neq 0$.
\end{enumerate}
The second point is proved in our previous
paper~\cite{AL19:incompressible}, building on results of Ni and
Hedden-Ni~\cite{Ni:spheres,HeddenNi10:small,HeddenNi13:detects}.

\subsection{A quick review of bordered-sutured Floer homology}\label{sec:bs-background}
To define the bordered Floer complexes
of a 3-manifold with multiple boundary components, one fixes a single
basepoint on each boundary component and a tree in the 3-manifold
connecting the basepoints. For the constructions below, it is more
convenient not to fix such a tree, and to allow multiple basepoints on a
single boundary component. This fits nicely as a special case of Zarev's
bordered-sutured Floer homology~\cite{Zarev09:BorSut}, so we review that
theory.
We assume the reader is already somewhat familiar with bordered Floer
homology.

A \emph{sutured surface} is an oriented surface $F$ with no closed
components together with a 0-manifold $\Lambda$ in its boundary
dividing the boundary into two collections of intervals, $S_+$ and $S_-$~\cite[Defintion
1.2]{Zarev09:BorSut}. (In particular, $\Lambda$ is required to
intersect every component of $\bdy F$.) We call $\Lambda$ the \emph{sutures}, and $S_+$ and $S_-$ the \emph{positive} and \emph{negative} arcs in the boundary, respectively. The relevant combinatorial
model for a sutured surface is an \emph{arc diagram}, which consists
of a collection of oriented arcs $Z$, an even number of points $\mathbf{a}$ in
the interior of $Z$, and a fixed point-free involution $M$ of
$\mathbf{a}$, which we think of as identifying the points in
$\mathbf{a}$ in pairs, satisfying a compatibility condition that we
state presently. An arc diagram $\PMC$ specifies a sutured surface
$F(\PMC)$ by thickening the arcs to rectangles $Z\times[0,1]$,
attaching 1-handles to $Z\times\{0\}$ via the matching, and declaring
the sutures to be $(\bdy Z)\times\{1/2\}$, $S_+$ to be
$Z\times\{1\}\cup (\bdy Z)\times[1/2,1]$, and $S_-$ to be the rest of
the boundary~\cite[Section 2.1]{Zarev09:BorSut}. The compatibility
condition for an arc diagram is that $S_-$ has no closed components.

In addition to abstract arc diagrams, we will also consider arc diagrams
embedded in surfaces. That is, given a surface $\Sigma$ (not necessarily closed), an \emph{arc diagram} for $\Sigma$ is an arc diagram $\PMC$ together with an embedding $F(\PMC)\into \Sigma$ so that each component of $\Sigma\setminus F(\PMC)$ is either a disk or an annulus around a component of $\bdy\Sigma$. An arc diagram $\PMC$ is called \emph{special} if every component of $\bdy F(\PMC)$ contains exactly one positive and one negative arc.

A \emph{bordered-sutured manifold} is a cobordism with corners between
sutured surfaces. That is, a bordered-sutured manifold is a $3$-manifold
$Y$, an embedding $\phi\co F(\PMC)\to \bdy Y$ for some arc diagram
$\PMC$, and a 1-dimension submanifold $\Gamma\subset \bigl(\bdy
Y\setminus \phi(F(\PMC))\bigr)$ with $\bdy\Gamma=\Lambda$, dividing
$\bdy Y\setminus \phi(F(\PMC))$ (the \emph{sutured boundary}) into
regions $R_+$ and $R_-$, with $\bdy R_\pm=S_\pm\cup\Gamma$. We will
often abbreviate the data $(Y,\phi,\Gamma,R_\pm)$ of a bordered-sutured
manifold simply as $Y$ or $(Y,\phi)$. An example of a bordered-sutured
manifold is depicted in Figure \ref{fig:bs-eg}. (If we are thinking of
$Y$ as a cobordism then some part of the bordered boundary is viewed as
on the left and some on the right.)

\begin{figure}
  \centering
  \includegraphics[width=5in]{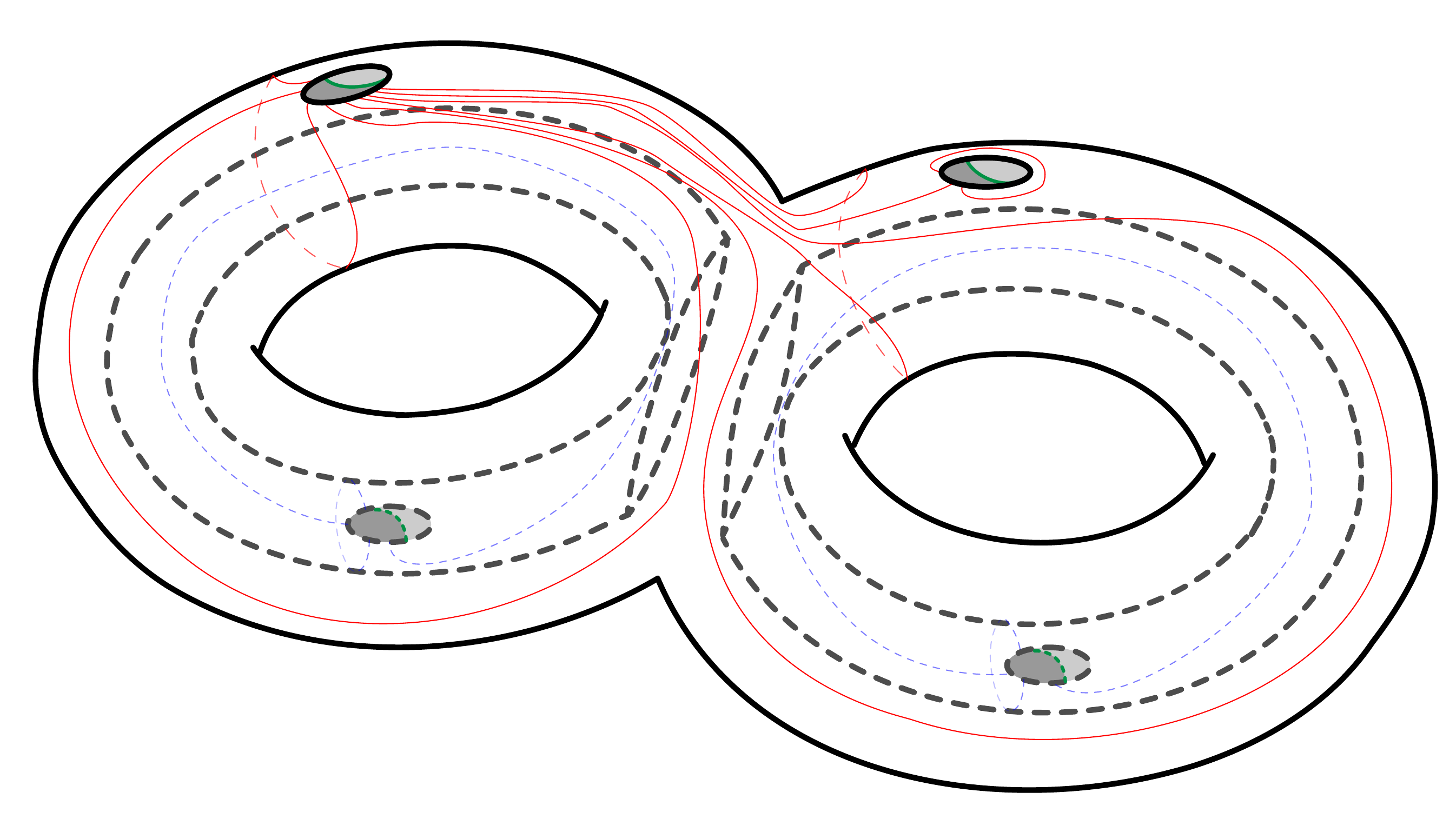}\qquad\includegraphics[scale=.8]{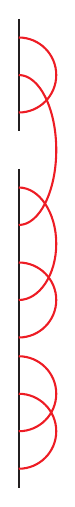}
  \caption[A bordered-sutured manifold]{\textbf{A bordered-sutured
  manifold.} The figure shows a compression body with outer boundary a
  genus $2$ surface and inner boundary the union of two tori. The
  bordered boundary is the complement of four disks. The sutured
  boundary is shaded, with $R_-$ darker than $R_+$. On the genus $2$
  boundary we have drawn the cores of the $1$-handles of $F(\PMC)$ as
  \textcolor{red}{thin} curves, where $\PMC$ is the arc diagram shown at
  the right. On the two genus $1$ boundaries, the cores of the
  $1$-hanles of $F(\PMC')$ are drawn with \textcolor{darkblue}{thin,
  dashed} curves.}
  \label{fig:bs-eg}
\end{figure}

To each arc diagram $\PMC$, bordered-sutured Floer homology associates
a \dg algebra $\Alg(\PMC)$, defined combinatorially in terms of chords
in the diagram. One can take disjoint unions of arc diagrams, and both
the construction of
sutured surfaces and the bordered-sutured algebras behave well with
respect to this operation:
\begin{align}
  F(\PMC_1\amalg \PMC_2)&\cong F(\PMC_1)\amalg F(\PMC_2) \\
  \Alg(\PMC_1\amalg\PMC_2)&\cong\Alg(\PMC_1)\otimes_{\FF_2}\Alg(\PMC_2),\label{eq:alg-is-tens-prod}
\end{align}
canonically. Orientation reversal also has a simple effect: reversing
the orientation of the arcs in $\PMC$ gives an arc diagram $-\PMC$,
$F(-\PMC)=-F(\PMC)$, and $\Alg(-\PMC)=\Alg(\PMC)^{\op}$.

Associated to a bordered-sutured manifold $Y$ with bordered boundary
$F(\PMC)$ is a left type $D$ structure (twisted complex) $\BSD(Y)$
over $\Alg(-\PMC)$ and a right $\Ainf$-module $\BSA(Y)$ over
$\Alg(\PMC)$. If $F(\PMC)$ has two components, by
Formula~\eqref{eq:alg-is-tens-prod}, we can view $\BSD(Y)$ and
$\BSA(Y)$ as bimodules; more generally, if $F(\PMC)$ has many
components, $\BSD(Y)$ and $\BSA(Y)$ can be viewed as
multi-modules. For $\BSD$, this just uses the identification between
type $D$ structures over
$\Alg(-\PMC_1)\otimes\cdots\otimes\Alg(-\PMC_n)$ and type $D^n$
multi-modules over $\Alg(-\PMC_1),\dots,\Alg(-\PMC_n)$. For $\BSA$,
this uses the equivalence of categories between $\Ainf$-modules over
$\Alg(\PMC_1)\otimes\cdots\otimes\Alg(\PMC_n)$ and
$\Ainf$-multi-modules over $\Alg(\PMC_1),\dots,\Alg(\PMC_n)$ (see,
e.g.,~\cite[Section 2.4.3]{LOT2}).

Given an arc diagram $\PMC$, we can view the identity cobordism
$\Id_\PMC$ of $F(\PMC)$ as a bordered-sutured manifold. The invariants
$\BSA(\Id_\PMC)$ and $\BSD(\Id_\PMC)$ are then bimodules over
$\Alg(\PMC)$ and $\Alg(-\PMC)$. If $Y$ has bordered boundary $F(\PMC)$
then
\begin{align*}
  \BSA(Y)&\simeq \BSA(\Id_\PMC)\DT_{\Alg(-\PMC)}\BSD(Y)\\
  \BSD(Y)&\simeq \BSA(Y)\DT_{\Alg(\PMC)}\BSD(\Id_\PMC).
\end{align*}
In fact, for corresponding choices of Heegaard diagrams, the second
of these homotopy equivalences is an isomorphism, and could be taken
as the definition of $\BSD(Y)$.

If the set of bordered boundary components of $Y$ is partitioned
into two parts, $F(Z_D)\amalg F(Z_A)$, there is a mixed-type invariant
\[
  \BSDA(Y)=\BSA(Y)\DT_{\Alg(Z_D)}\BSD(\Id_{Z_D}).
\]
We will call $F(Z_D)$ (respectively $F(Z_A)$) the \emph{type $D$}
(respectively \emph{type $A$}) \emph{bordered boundary} of $Y$.

Given bordered-sutured manifolds $(Y_1,\phi_1\co F_1\to \bdy Y_1)$ and
$(Y_2,\phi_2\co F_2\to \bdy Y_2)$ so that $F_1$ and $F_2$ have a common
sutured subsurface $F$, their \emph{gluing along $F$} is
\[
Y_1\cup_F Y_2=
  Y_1\thinspace\lsub{\phi_1|_F\!\!}{\cup}_{\phi_2|_F}Y_2=(Y_1\amalg Y_2)/\bigl(\phi_1(F)\ni p\sim \phi_2(\phi_1^{-1}(p))\in \phi_2(F)\bigr).
\]
The \emph{pairing theorem} for bordered-sutured Floer homology states that when gluing
along a collection of components $F$ of their boundary which are
parameterized by the same arc diagram (up to an orientation reversal), where those components are treated as type $A$ boundary for $Y_1$ and type
$D$ boundary for $Y_2$, the bordered invariants assemble as
\[
\BSDA(Y_1\cup_F Y_2)\simeq \BSDA(Y_1)\DT_{\Alg(F)}\BSDA(Y_2)
\]
\cite[Theorem 8.7]{Zarev09:BorSut}. If not taken as definitions, the
previous three formulas are special cases of this pairing theorem.

An important class of the bordered-sutured manifold is the following:
\begin{definition}
  A diffeomorphism $\psi$ from a sutured surface $(F,\Lambda)$ to
  another sutured surface $(F',\Lambda')$ is called \emph{strongly
  based} if $\psi(S_+)=S'_+$ and $\psi(S_-)=S'_-$. Given a strongly
  based diffeomorphism $\psi\co F(\PMC_L)\to F(\PMC_R)$, the
  \emph{mapping cylinder} of $\psi$ is $Y_\psi=([0,1]\times
  F(\PMC),\psi,\Id)$.
\end{definition}
We often abbreviate $\BSDA(Y_\psi)$ as $\BSDA(\psi)$. 

\begin{figure}
  \centering
  \includegraphics{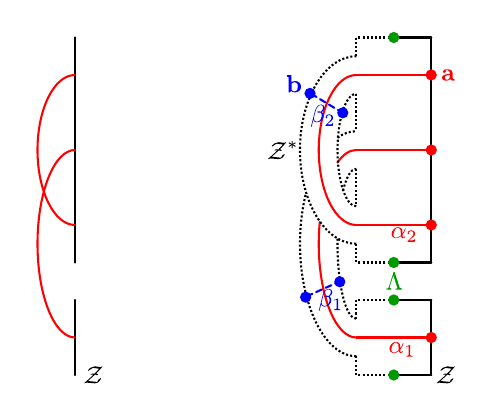}
\caption{\textbf{The half-identity diagram.} Left: an arc diagram $\PMC$. Right: the associated half-identity diagram. The $\alpha$-arcs are \textcolor{red}{solid}, the $\beta$-arcs are \textcolor{blue}{dashed}, $S_+$ is solid, and $S_-$ is dotted.}
  \label{fig:half-id-diagram}
\end{figure}

The constructions in Sections~\ref{sec:tt-arc-diagram-slides}
and~\ref{sec:compress-diffeo} will use a particular way of associating
Heegaard diagrams to mapping cylinders. These Heegaard diagrams come in
two halves. We start with the case of the identity map:
\begin{definition}\label{def:id-HD}
  Fix an arc diagram $\PMC=(Z,\mathbf{a},M)$. On $F(\PMC)$, identify  $\mathbf{a}$ with a subset of $S_+$ and let $\alpha_1,\dots,\alpha_n$ denote
  the (extended) cores of the handles glued to the points $\mathbf{a}$, so
  $\bdy(\alpha_1\cup\cdots\cup\alpha_n)=\mathbf{a}$. Let $\beta_i$ be an arc
  with boundary in $S_-$ and so that $\beta_i$ intersects $\alpha_i$ in a single
  point and is disjoint from $\alpha_j$ for $i\neq j$. Let
  $\mathbf{b}=\bdy(\beta_1\cup\dots\cup\beta_n)$. Define
  $M'\co \mathbf{b}\to\mathbf{b}$ to exchange the endpoints of $\beta_i$. Then
  $\PMC^*=(S_-,\mathbf{b},M')$ is the \emph{dual arc diagram} to $\PMC$.  The
  triple $(F(\PMC),\{\alpha_1,\dots,\alpha_n\},\{\beta_1,\dots,\beta_n\})$ is
  the \emph{standard half Heegaard diagram for the identity map of $\PMC$},
  denoted $\HalfHD(\Id_\PMC)$ or $\HalfHD(\Id)$. See Figure~\ref{fig:half-id-diagram}.

  The \emph{standard Heegaard diagram for the identity map of $\PMC$},
  $\HD(\Id)$, is the result of gluing $\HalfHD(\Id_\PMC)$ to
  $\HalfHD(\Id_{\PMC^*})^\beta$, where $\HalfHD(\Id_{\PMC^*})^\beta$ is the
  result of exchanging the $\alpha$- and $\beta$-arcs in
  $\HalfHD(\Id_{\PMC^*})$.
\end{definition}

More generally, by a \emph{half Heegaard diagram} we mean a sutured
surface $F$, arc diagrams $\PMC_L$ and $\PMC_R$, and diffeomorphisms
$\phi_L\co F(\PMC_L)\to F$ and $\phi_R\co F(\PMC_R)\to F$, so that on
the boundary, $\phi_L(S_\pm(\PMC_L))=S_\pm$ and
$\phi_R(S_\mp(\PMC_R))=S_\pm$. The images of the cores of the
$1$-handles of $F(\PMC_L)$ under $\phi_L$ give arcs
$\alpha_1,\dots,\alpha_n$ in $F$, and the images of the cores of the
$1$-handles of $F(\PMC_R)$ under $\phi_R$ give arcs
$\beta_1,\dots,\beta_n$ in $F$. Further, we can recover the
diffeomorphisms $\phi_L$ and $\phi_R$, up to isotopy, from these
arcs. So, we will sometimes refer to $(S,\PMC_L,\phi_L,\PMC_R,\phi_R)$
as the half Heegaard diagram, and sometimes refer to
$(S,\{\alpha_1,\dots,\alpha_n\},\{\beta_1,\dots,\beta_n\})$ as the
half Heegaard diagram.  In particular, $\HalfHD(\Id_{\PMC})$
corresponds to the case that $\PMC_L=\PMC$, $\PMC_R=\PMC^*$,
$\phi_L=\Id$, and $\phi_R=\Id$.

Given a half Heegaard diagram $\HD=(F,\phi_L,\phi_R)$, we get a map
$\phi_R^{-1}\circ\phi_L\co F(\PMC_L)\to F(\PMC_R)$ sending
$S_\pm(\PMC_L)\to S_\mp(\PMC_R)$. The diagram $\HalfHD(\Id_{\PMC_R})$
induces a diffeomorphism $\Id'\co F(\PMC_R)\to F(\PMC_R^*)$ which maps $S_{\mp}(\PMC_R)$ to $S_{\pm}(\PMC_R^*)$. So,
$\psi(\HD)=\Id'\circ\phi_R^{-1}\circ\phi_L\co F(\PMC_L)\to
F(\PMC_R^*)$ is a strongly based diffeomorphism. We call this diffeomorphism the \emph{mapping class associated to the half Heegaard
  diagram}.  Equivalently, if we glue the half Heegaard diagram to the
standard half identity diagram $\HalfHD(\Id_{\PMC_R})$, we obtain a bordered-sutured Heegaard diagram for the mapping cylinder of $\psi(\HD)$. Consequently, given a strongly based diffeomorphism $\psi\co F(\PMC')\to F(\PMC)$ a bordered-sutured Heegaard diagram for $Y_{\psi}$ is obtained from gluing the half-identity diagram $\HalfHD(\Id_{{\PMC}^*})^{\beta}$ to the half Heegaard diagram $\left(F(\PMC),\{\psi(\alpha_1),\cdots,\psi(\alpha_n)\},\{\beta_1,\cdots,\beta_n\}\right)$, where $\HalfHD(\Id_{\PMC})=(F(\PMC),\{\alpha_1,\dots,\alpha_n\},\{\beta_1,\dots,\beta_n\})$.

\begin{figure}
  \centering
  \includegraphics[scale=.95]{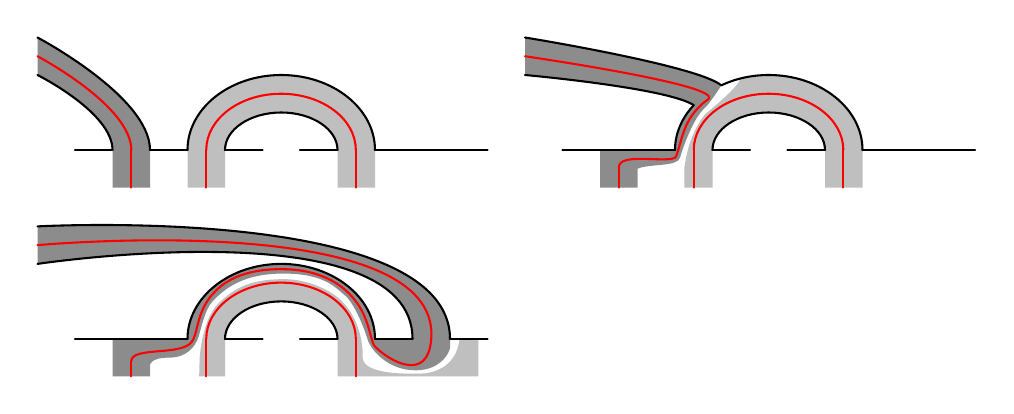}
  \caption{\textbf{An arcslide diffeomorphism.} Top left: the diagram before the arcslide. The two relevant handles are shaded, and their cores ($\alpha$-arcs) are \textcolor{red}{indicated}. Top right: an intermediate stage. The images of the handles under the diffeomorphism are shaded and the images of their cores \textcolor{red}{indicated}. Bottom right: the diagram after the arcslide. Again, the images of the handles are shaded and the images of their cores \textcolor{red}{indicated}.}
  \label{fig:arcslide}
\end{figure}

A key tool for bordered Floer theory (e.g.,~\cite{LOT4}) is a particular class of diffeomorphisms, the arcslides:
\begin{definition}\label{def:arcslide}
  Given an arc diagram $\PMC=(\{Z_i\},\mathbf{a},M)$ and a pair of
  adjacent points $a,a'\in Z_i\cap\mathbf{a}$, let $a''$ be a point
  adjacent to $M(a)$, so that $a''$ is above (respectively below)
  $M(a)$ if $a'$ is below (respectively above) $a$. There is a new arc
  diagram $\PMC'$ obtained by replacing $a'$ by $a''$ (and defining
  $M(a'')=M(a')$). We say that $\PMC'$ is \emph{obtained from
    $\PMC$ by an arcslide}.

  The surface $F(\PMC')$ is obtained from $F(\PMC)$ by sliding one
  foot of the handle corresponding to $\{a',M(a')\}$ over the handle
  corresponding to $\{a,M(a)\}$. In particular, this 1-parameter
  family of sutured surfaces induces a strongly based diffeomorphism
  $F(\PMC')\to F(\PMC)$, the \emph{arcslide diffeomorphism}
  corresponding to sliding $a'$ over $a$. See
  Figure~\ref{fig:arcslide} (as well as~\cite[Section
  6.1]{AndersenBenePenner09:MCGroupoid} and~\cite[Figure 3]{LOT4}).
\end{definition}

Given a half Heegaard diagram $\HD=(F,\phi_L\co F(\PMC_L)\to
F,\phi_R\co F(\PMC_R)\to F)$ and an arcslide diffeomorphism
$\psi\co F(\PMC'_L)\to F(\PMC_L)$ there is a new half Heegaard diagram
$(F,\phi_L\circ\psi,\phi_R)$, which we call the result of performing
the arcslide to $\HD$. In terms of $\alpha$- and $\beta$-arcs, this
corresponds to performing an embedded arcslide of the $\alpha$-arc
corresponding to $a'$ over the $\alpha$-arc corresponding to $a$.

The bordered-sutured manifolds of interest in this paper come from
compression bodies, with simple kinds of sutures. We will give a name
to that class:
\begin{definition}
  A \emph{special bordered-sutured manifold} is a bordered-sutured
  manifold so that each component of $R_\pm$ is a bigon with one edge
  a component of $\Gamma$ and one edge a component of $S_\pm$.
\end{definition}
The example in Figure~\ref{fig:bs-eg} is a special bordered-sutured manifold.

Given a bordered 3-manifold with connected boundary, if we view the
bordered Heegaard diagram as a bordered-sutured Heegaard diagram then
the corresponding bordered-sutured manifold is special, with each $R_\pm$ a single bigon. An arced bordered Heegaard diagram with two
boundary components (such as the diagrams representing
diffeomorphisms~\cite[Section 5.3]{LOT2}) represents a pair
$(Y,\gamma)$ of a cobordism between closed surfaces and an arc $\gamma$
connecting the two boundary components, but does not directly
represent a special bordered-sutured manifold. Specifically, we obtain
a bordered-sutured diagram by deleting a neighborhood of the arc and
viewing the newly-created boundary as sutured arcs; see
Figure~\ref{fig:arced-HD}. The corresponding bordered-sutured manifold
is $Y\setminus\nbd(\gamma)$ where $R_+$ and $R_-$ consist of a
rectangle each on the cylinder $\bdy\nbd(\gamma)$, and $\Gamma$
consists of two arcs along $\bdy\nbd(\gamma)$.  In particular, this is
not a special bordered-sutured manifold. To obtain a special
bordered-sutured manifold, we attach a 2-handle to a meridian of
$\gamma$ and modify $R_+$ and $R_-$ to be bigons. At the level of
Heegaard diagrams, this can be accomplished by gluing on the
tube-cutting diagram shown in Figures~\ref{fig:arced-HD} and~\ref{fig:TC-again}. 
(The tube-cutting diagram appeared previously
in~\cite{Hanselman16:graph,LT:hoch-loc,AL19:incompressible}.)

\begin{figure}
  \centering
  \includegraphics{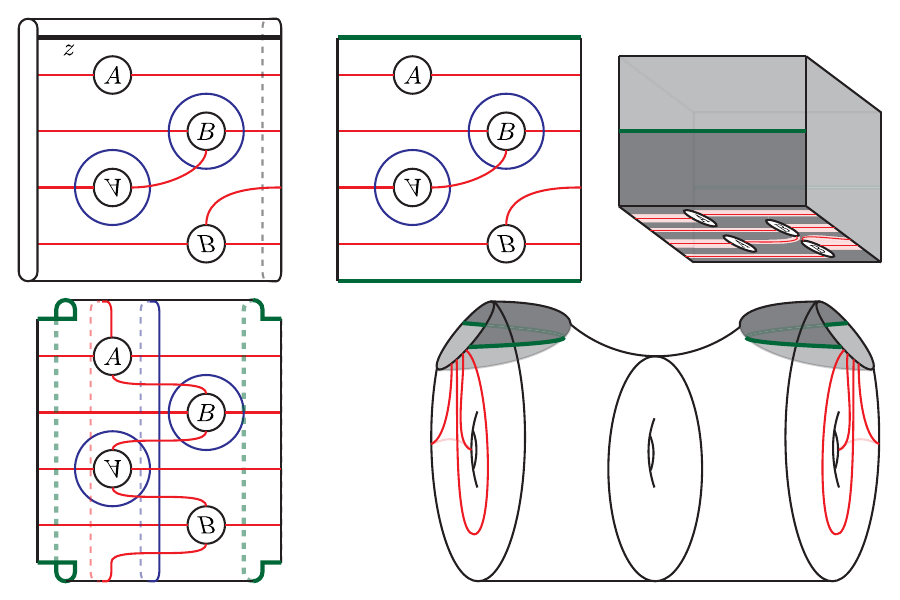}
  \caption[Turning a bordered Heegaard diagram into a
  sutured Heegaard diagram]{\textbf{Turning an arced bordered Heegaard diagram into a
      sutured Heegaard diagram.} Top-left: an arced bordered Heegaard
    diagram for the mapping cylinder of a Dehn twist on a
    torus. Top-center: the corresponding bordered-sutured
    diagram. Top-right: the corresponding bordered-sutured manifold,
    which is not special.  Bottom-left: the tube-cutting
    diagram. Bottom-right: the bordered-sutured manifold represented
    by the tube-cutting diagram.  Gluing this to the manifold on the
    top-right gives a special bordered-sutured manifold.}
  \label{fig:arced-HD}
\end{figure}

There are analogous constructions for bordered $3$-manifolds with more
than two boundary components. For example, Figure~\ref{fig:bs-handlebody-diag}
shows a bordered Heegaard diagram for $(Y,\gamma)$ where $Y$ is the
compression body with outer boundary of genus $2$ and inner boundary
$T^2\amalg T^2$ and $\gamma$ is a Y-shaped graph in $Y$ connecting the
three boundary components. The figure also shows the corresponding
bordered-sutured diagram (which is not special). Gluing on
tube-cutting pieces to two of the
boundary gives a bordered-sutured Heegaard diagram for a special
bordered-sutured structure on $Y$. One can also introduce extra
sutured disks, by gluing on the bordered-sutured diagram shown on the
right of Figure~\ref{fig:bs-handlebody-diag}. In particular, from
these pieces it is easy to construct a bordered-sutured diagram for
any special bordered-sutured structure on a compression body.

\begin{figure}
  \centering
  \includegraphics{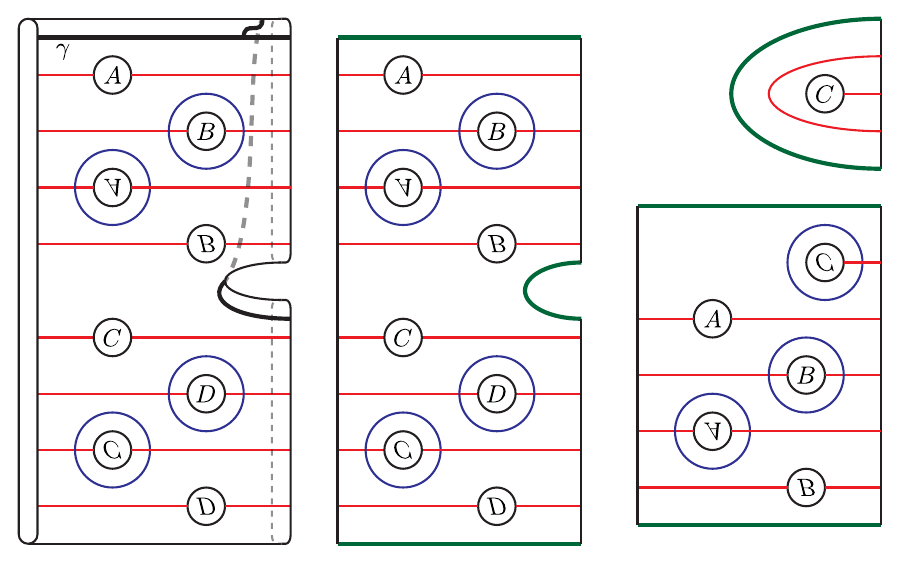}
  \caption[Bordered-sutured diagrams for compression bodies]{\textbf{Bordered-sutured diagrams for compression bodies.}
    Left: a bordered Heegaard diagram for $(Y,\gamma)$ where $\gamma$
    is a Y-shaped graph connecting the boundary components of
    $Y$. Center: a bordered-sutured Heegaard diagram $\HD$ for a
    special bordered-sutured structure on $Y$. Right: a
    bordered-sutured diagram $\HD'$ so that gluing $\HD'$ along its
    left boundary to a
    boundary component of $\HD$ introduces an extra disk suture.}
  \label{fig:bs-handlebody-diag}
\end{figure}

The bordered-sutured modules have several duality properties which
allow one to state the pairing theorem in terms of morphism spaces
instead of tensor products (cf.~\cite{LOTHomPair}). In the general
case, the duality properties have the effect of twisting the sutures
(see~\cite{Zarev:JoinGlue}), so we will state them only for special
bordered-sutured manifolds:
\begin{theorem}\label{thm:bs-dual}
  Let $Y$ be a special bordered-sutured manifold and $F(\PMC)$ a
  component of the bordered boundary of $Y$. Then there is a homotopy
  equivalence
  \[
    \Mor_{\Alg(-\PMC)}\bigl(\Alg(-\PMC)\DT\BSD(Y),\Alg(-\PMC)\bigr)
    \simeq \BSDA(-Y)
  \]
  where, on the right side, $F(\PMC)$ is treated as type $A$ boundary
  and the other components as type $D$ boundary.

  In particular, given another bordered-sutured manifold $Y'$ and an
  identification of $F(\PMC)$ with one of the boundary components of
  $Y'$,
  \[
    \Mor_{\Alg(-\PMC)}\bigl(\BSD(Y),\BSD(Y')\bigr)
    \simeq \BSD((-Y)\cup_{F(\PMC)}Y').
  \]
\end{theorem}
\begin{proof}
  The proof is the same as the proof of~\cite[Theorem
  2.6]{AL19:incompressible} (most of the work of which is
  in~\cite{Zarev:JoinGlue}), observing that the twisting slice acts
  trivially on a special bordered-sutured manifold. The second
  statement follows from the first by tensoring both sides with
  $\BSD(Y')$.
\end{proof}
Versions of Theorem~\ref{thm:bs-dual} where the other
boundary components of $Y$ or $Y'$ are treated as type $A$ rather than
$D$ boundary follow from the same proof or, in many cases, by tensoring with the type \AAm\ identity bimodule.

\subsection{Twisted bordered-sutured Floer homology}\label{sec:twisted-bs}
Bordered-sutured Floer homology with twisted coefficients is not developed in Zarev's
papers, but is a straightforward combination of his theory with
Ozsv\'ath-Szab\'o's construction of $\HFa$ with twisted coefficients. 
The group of periodic domains $\pi_2(\x,\x)$ in a bordered-sutured
Heegaard diagram representing $(Y,\Gamma,\phi\co F(\PMC)\to \bdy Y)$
is isomorphic to
$H_2(Y,\phi(F(\PMC)))$~\cite[p. 24]{Zarev09:BorSut}. Note that $\phi(F(\PMC))$ is
the bordered boundary of $Y$, so in particular has no closed
components; abusing notation, we will use $F$ to denote the image
$\phi(F)\subset\bdy Y$. Thus, one can
define the twisted bordered-sutured modules $\tBSD(Y)$ and $\tBSA(Y)$,
similarly to the construction of $\tCFa$ above. (The bordered case is
discussed in~\cite[Sections 6.4 and 7.4]{LOT1}.) Specifically, we can
view $\Alg(\PMC)\otimes\FF_2[H_2(Y,F)]$ as a \dg algebra over
$\FF_2[H_2(Y,F)]$. 
Then the invariant
$\tBSD(Y)$ is a type $D$ structure over $\Alg(-\PMC)\otimes\FF_2[H_2(Y,F)]$, while
$\tBSA(Y)$ is a (strictly unital) $\Ainf$-module over
$\Alg(\PMC)\otimes\FF_2[H_2(Y,F)]$. The fact that $\Alg(\PMC) \otimes\FF_2[H_2(Y,F)]$ is viewed as a \dg algebra over
the ground ring $\FF_2[H_2(Y, F)]$ means, in particular, that the
operations on $\tBSA(Y)$ satisfy
$m_{1+n}(x,a_1,\dots,e^ha_i,\dots,a_n)=e^hm_{1+n}(x,a_1,\dots,a_i,\dots,a_n)$
for any $e^h\in \FF_2[H_2(Y,F)]$, so strict unitality implies that
$m_{1+n}(x,a_1,\dots,e^h,\dots,a_n)=0$ if $n>1$. The formulas for the
twisted operations are obtained easily from the untwisted case. For
example, for $\tBSA(Y)$, one replaces Zarev's formula~\cite[Definition
7.12]{Zarev09:BorSut} by
\[
  m_k(\x,a_1,\dots,a_{k-1})=\sum_{\substack{\y\in \mathcal{G}(x,y)\\
      \vec{a}(\x,\y,\vec{\rho})=a_1\otimes\dots\otimes a_{k-1}\\\ind(B,\vec{\rho})=1}}
  \Bigl(\#\cM^B_{\mathit{emb}}(\x,\y;\vec{\rho})\Bigr)e^{[B_\x*B*B_\y^{-1}]}\y,
\]
where $B_\x$ and $B_\y$ are domains connecting the generators to the
basepoint, as in the definition of $\tCFa$. (Also as in the definition
of $\tCFa$, one fixes one choice of base generator for each
$\SpinC$-structure.)

One can equivalently view $\tBSA(Y)$ as an $\Ainf$-bimodule over
$\FF_2[H_2(Y,F)]$ and $\Alg(\PMC)$, and $\tBSD(Y)$ as a type \DA\
bimodule over $\Alg(-\PMC)$ and $\FF_2[H_2(Y,F)]$. In both cases,
the bimodule structure is quite simple: for $\tBSA(Y)$, $m_{m,1,n}$
vanishes if $m$ and $n$ are both positive, or if $m>1$ and $n=0$. For
$\tBSD(Y)$, $\delta^1_{1+n}$ is only non-zero for $n=0,1$ and the
operation $\delta^1_{2}$ is given by $\delta^1_2(\x,e^h)=\iota\otimes
(e^h\x)$ (where $\iota$ is the basic idempotent with $\iota \x=\x$).

There are also partially twisted versions of the bordered-sutured
modules. Focusing on $\tBSA$, say, given a module $M$ over
$\FF_2[H_2(Y,F)]$, the bordered-sutured complex twisted by $M$ is
\[
  \tBSA(Y;M)=M\otimes_{\FF_2[H_2(Y,F)]}\tBSA(Y).
\]
A particularly interesting case is when $F_0$ is a union of components
of the bordered boundary of $Y$ and $M=\FF_2[H_1(F_0)]$ is an
algebra over $\FF_2[H_2(Y,F)]$ via the connecting homomorphism
$H_2(Y,F)\to H_1(F)$ and projection
$H_1(F)\to H_1(F_0)$.

The twisted versions $\tBSDA(Y)$ and $\tBSDA(Y;M)$ of $\BSDA(Y)$ are
defined similarly.

We state the twisted-coefficient pairing theorem for the invariant
$\BSDA$, since the other versions are special cases:
\begin{theorem}\label{thm:twisted-pairing}
  Let $Y$ be the result of gluing bordered-sutured $3$-manifolds $Y_1$
  and $Y_2$ along a union of bordered boundary components
  $F(\PMC')$. Let $F_i$ denote the full bordered boundary of $Y_i$ and $F$
  the bordered boundary of $Y$. View $F(Z')$ as type $A$ bordered
  boundary of $Y_1$ and
  view $F(-Z')$ as type $D$ bordered boundary of $Y_2$. Let $M_i$
  be a module over $\FF_2[H_2(Y_i,F_i)]$. Then
  \begin{equation}\label{eq:twisted-pairing}
    \tBSDA(Y_1;M_1)\lsup{\Alg(Z')}\DT_{\Alg(Z')}\tBSDA(Y_2;M_2)\simeq
    \tBSDA(Y;M_1\otimes M_2)
  \end{equation}
  where $M_1\otimes M_2$ is a module over $\FF_2[H_2(Y,F)]$ via
  the ring homomorphism
  $
  \FF_2[H_2(Y,F)]\to \FF_2[H_2(Y_1,F_1)]\otimes\FF_2[H_2(Y_2,F_2)]
  $
  induced by the homomorphism
  $
  H_2(Y,F)\to H_2(Y_1,F_1)\oplus H_2(Y_2,F_2)
  $
  induced by the maps of pairs $(Y,F)\to (Y,F_i\cup Y_j)$ and excision
  $H_2(Y,F_i\cup Y_j)\cong H_2(Y_i,F_i)$, where $\{i,j\}=\{1,2\}$.
\end{theorem}
The proof is the same as, for instance,~\cite[Theorem 9.44]{LOT1} (see
also~\cite[Theorem 12]{LOT2}). A key point is that a periodic domain
for $Y$ restricts to periodic domains for $Y_1$ and $Y_2$; this
corresponds to the map $H_2(Y,F)\to H_2(Y_1,F_1)\oplus H_2(Y_2,F_2)$
above. Choosing paths $B_\x$ for $Y$ as in
Section~\ref{sec:twisted-HF} gives a partial choice of paths for
the $Y_i$, compatible along the boundary; these can then be extended
to complete choices for the $Y_i$ and used to define the twisted
coefficient complexes.

For example, if $M_i=\FF_2[H_2(Y_i,F_i)]$, then
Formula~\eqref{eq:twisted-pairing} computes
$\tBSDA(Y;\FF_2[H_2(Y,F_1\cup F_2)])$, not $\tBSDA(Y;\FF_2[H_2(Y,F)])$.
(The difference is that $F_1\cup F_2$ includes the part of $\bdy Y_i$
where $Y_1$ and $Y_2$ are glued together, which is not part of $\bdy
Y$.) However, there is an evident split injection
$H_2(Y,F)\hookrightarrow H_2(Y,F_1\cup F_2)$. If one chooses a splitting $p\co H_2(Y,F_1\cup F_2)\to H_2(Y,F)$ then we can form the tensor product
\begin{equation}\label{eq:drop-extra-twisting}
  \tBSDA(Y;\FF_2[H_2(Y,F_1\cup F_2)])\otimes_{\FF_2[H_2(Y,F_1\cup F_2)]}\FF_2[H_2(Y,F)]\simeq \tBSDA(Y).
\end{equation}
So, the pairing theorem determines the standard, totally twisted
bimodule of $Y$ via the two-step process of taking the box product and
then extending scalars under $p_*$.

As a special case of Theorem~\ref{thm:twisted-pairing}, if some component(s) $F_0$ of $\bdy Y$ are parameterized
by $F(\PMC_0)$ then
\begin{equation}\label{eq:twist-pair-Id}
  \tBSDA(Y;\FF_2[H_1(F_0)])\simeq \BSDA(Y)\DT\tBSDA(\Id_{\PMC_0}).
\end{equation}
So, one can reconstruct the boundary twisted version of $\BSDA$ from
the untwisted version. (In particular, the boundary twisting is, in some sense,
redundant with the bordered module structure.)

Finally, we state the K\"unneth theorem for bordered-sutured Floer
homology. (Note that one cannot apply the statement below directly to
an arced bordered manifold with two boundary components: one must
first replace it by a bordered-sutured manifold, which then has
connected boundary.)

\begin{theorem}\label{thm:bs-Kunneth}
  Consider an (internal) connected sum of two bordered-sutured
  3-manifolds, $(Y_1,F_1)$ and $(Y_2,F_2)$.
  \begin{itemize}
  \item If $\bdy Y_1\neq\emptyset$ but $\bdy Y_2=\emptyset$ then
    \begin{align*}
      \BSD(Y_1\#Y_2)&\simeq \BSD(Y_1)\otimes_{\FF_2}\CFa(Y_2)\\
      \tBSD(Y_1\#Y_2)&\simeq \tBSD(Y_1)\otimes_{\FF_2}\tCFa(Y_2).
    \end{align*}
    Here, the second homotopy equivalence uses the evident isomorphism
    $H_2(Y_1\#Y_2,F_1)\cong H_2(Y_1,F_1)\oplus H_2(Y_2)$.
  \item If $\bdy Y_1\neq\emptyset$ and $\bdy Y_2\neq\emptyset$ then
    \begin{align*}
      \BSD(Y_1\#Y_2)&\simeq
      \BSD(Y_1)\otimes_{\FF_2}\BSD(Y_2)\otimes_{\FF_2}\CFa(S^2\times S^1)\\
      \tBSD(Y_1\#Y_2,F_1\amalg F_2)&\simeq
      \tBSD(Y_1)\otimes_{\FF_2}\tBSD(Y_2)\otimes_{\FF_2}\tCFa(S^2\times S^1).
    \end{align*}
    Here, the second homotopy equivalence uses the evident isomorphism
    $H_2(Y_1\#Y_2,F_1\amalg F_2)\cong H_2(Y_1,F_1)\oplus H_2(Y_2,F_2)\oplus\ZZ\cong
    H_2(Y_1,F_1)\oplus H_2(Y_2,F_2)\oplus H_2(S^2\times S^1)$. (Recall
    that the $F_i$ are just the bordered parts of the boundary, so
    have no closed components.)
  \end{itemize}
\end{theorem}
\begin{proof}
  All of these statements follow easily by taking the (internal)
  connected sum of Heegaard diagrams for $Y_1$ and $Y_2$ near a suture
  (basepoint). In the second case, to obtain a Heegaard diagram for
  $Y$, we then have to add a pair of $\alpha$- and $\beta$-circles
  inside the connected sum neck, which gives the extra
  $\CFa(S^2\times S^1)$ factor.
\end{proof}

\subsection{Twisted sutured Floer homology}
Here, we collect some properties we need of twisted sutured Floer homology.
Twisted sutured Floer homology can be viewed as a special case of
twisted bordered-sutured Floer homology, in which the bordered
boundary is empty. More directly, given a balanced sutured manifold
$(Y,\Gamma)$, there is a twisted sutured Floer complex
$\tSFC(Y,\Gamma)$ over $\FF_2[H_2(Y)]$, extending the untwisted
case~\cite{Juhasz06:Sutured}, defined as follows. Given a sutured
Heegaard diagram $\HD=(\Sigma,\alphas,\betas)$ for $(Y,\Gamma)$ and a
point $\x_0\in T_\alpha\cap T_\beta$, there is a map
$\pi_2(\x_0,\x_0)\to H_2(Y)$, which we denote
$B\mapsto [B]$~\cite[Definition 3.9]{Juhasz06:Sutured}. (If the
Heegaard diagram is sufficiently stabilized, or we work in the
cylindrical setting, this map is an isomorphism.)  The complex
$\tSFC(Y,\Gamma)$ is freely generated over $\FF_2[H_2(Y)]$ by
$T_\alpha\cap T_\beta$ with differential as in
Formula~\eqref{eq:tHF-diff} (without the requirement on $n_z$, since
sutured Heegaard diagrams have punctures instead of basepoints). Given
an $\FF_2[H_2(Y)]$-module $M$, $\tSFC(Y,\Gamma;M)$, $\tSFH(Y,\Gamma)$,
and $\tSFH(Y,\Gamma;M)$ are defined as in the closed case.

The key property of sutured Floer homology is its behavior under
surface decompositions. The twisted version of Juh\'asz's surface
decomposition theorem~\cite[Theorem 1.3]{Juhasz08:SuturedDecomp} is:
\begin{proposition}\label{prop:twist-surf-decomp}
  Let $(Y,\Gamma)$ be a balanced sutured manifold, $S\subset Y$
  a good decomposing surface~\cite[Defintion
  4.6]{Juhasz08:SuturedDecomp}, and $M$ a module over
  $\FF_2[H_2(Y)]$. Let $(Y',\Gamma')$ be the result of decomposing
  $(Y,\Gamma)$ along $S$. Then there is a module $M'$ over
  $\FF_2[H_2(Y')]$ so that
  \[
    \tSFH(Y',\Gamma';M')\cong \bigoplus_{\spinc \text{ outer w.r.t.\ } S}
    \tSFH(Y,\Gamma;\spinc;M).
  \]
  Moreover, as abelian groups, $M'\cong M$.
\end{proposition}
\begin{proof}
  We adapt Zarev's proof of the sutured decomposition
  theorem~\cite[Theorem 10.5]{Zarev09:BorSut}, and will assume some
  familiarity with it. (Zarev's figure~\cite[Figure
  10]{Zarev09:BorSut} is perhaps especially enlightening.) Decompose
  \begin{equation}\label{eq:tw-surf-decomp-pf-1}
    Y=(S\times [-2,2])\cup_{S\times\{\pm 2\}}W.
  \end{equation}
  Viewing $S\times[-2,2]$ as the identity cobordism of the sutured
  surface $S$ makes it into a bordered-sutured manifold; make $W$ into
  a bordered-sutured manifold so that
  Equation~\eqref{eq:tw-surf-decomp-pf-1} holds as sutured
  manifolds. There is a bordered-sutured structure on
  $S\times([-2,-1]\cup[1,2])$ so that
  \begin{equation}\label{eq:tw-surf-decomp-pf-2}
    Y'=(S\times \left([-2,-1] \cup[1,2]\right))\cup_{S\times\{\pm 2\}}W
  \end{equation}
  as sutured manifolds; see Zarev's paper for an explicit
  description. Moreover, for these bordered-sutured structures,
  \begin{equation}\label{eq:tw-surf-decomp-pf-3}
    \BSD(S\times \left([-2,-1] \cup[1,2]\right),\spinc_k')\cong \BSD(S\times[-2,2],\spinc_k)
  \end{equation}
  where $\spinc_k$ is the $\SpinC$-structure on $S\times[-2,2]$ so
  that in the corresponding idempotents for $\BSD(S\times[-2,2])$, all
  the $\alpha$-arcs corresponding to $S\times\{-2\}$ are occupied,
  and none corresponding to $S\times\{2\}$ are; and $\spinc_k'$ is the $\SpinC$-structure on $S\times [-2,-1] \cup[1,2]$ which agrees with $\spinc_k$ on $S\times\{\pm 2\}$. (See Zarev's paper
  for a little more discussion.)
  
  Combining the long exact sequence for the pair $(Y,S\times[-2,2])$
  and the excision isomorphism
  $H_*(Y,S\times[-2,2])\cong H_*(W,S\times\{-2,2\})$ gives
  \[
    0=H_2(S)\to H_2(Y)\to H_2(W,S\times\{-2,2\})\to H_1(S)\to\cdots.
  \]
  Since $H_1(S)$ is free, we can choose a splitting
  $p\co H_2(W,S\times\{-2,2\})\to H_2(Y)$. Let $M_W=p^*M$ be the
  module over $\FF_2[H_2(W,S\times\{-2,2\})]$ obtained by restricting
  scalars. Equation~\eqref{eq:tw-surf-decomp-pf-1} and the twisted
  pairing theorem, Theorem~\ref{thm:twisted-pairing}, give
  \[
    \tSFC(Y,\Gamma;\spinc_W\cup\spinc_S;M)\simeq \tBSA(W;\spinc_W;M_W)\DT\BSD(S\times[-2,2],\spinc_S).
  \]
  The outer $\SpinC$-structures on $Y$ are the ones that restrict to
  $\spinc_k$ on $S\times[-2,2]$, so in particular
  \begin{equation}
    \label{eq:tw-surf-decomp-4}
    \bigoplus_{\spinc \text{ outer w.r.t.\ }
      S}\tSFC(Y,\Gamma;\spinc;M)\simeq \bigoplus_{\spinc_W}\tBSA(W,\spinc_W;M_W)\DT\BSD(S\times[-2,2],\spinc_k).
  \end{equation}
  Since $\BSD(S\times[-2,2],\spinc_k)$ is supported on the idempotent
  where all the arcs corresponding to $S\times\{-2\}$ are occupied,
  and none corresponding to $S\times\{2\}$ are,
  the box tensor product on the right vanishes unless
  $\spinc_W$ has the same property. So, we could restrict the direct sum to these
  $\SpinC$-structures, which we could again call outer.

  The homeomorphism $Y'\cong W$ induces a map
  $i\co H_2(Y')\to H_2(W,S\times\{-2,2\})$. Let
  $M'=i^*M_W$. By Equation~\eqref{eq:tw-surf-decomp-pf-2} and the
  twisted pairing theorem,
  \begin{equation}
    \label{eq:tw-surf-decomp-5}
    \bigoplus_{\spinc'}\tSFC(Y',\Gamma';\spinc';M')\simeq \bigoplus_{\spinc_W}\tBSA(W,\spinc_W;M_W)\DT\BSD(S\times([-2,-1]\cup[1,2]),\spinc_k').
  \end{equation}

  Combining
  Equations~\eqref{eq:tw-surf-decomp-pf-3},~\eqref{eq:tw-surf-decomp-4},
  and~\eqref{eq:tw-surf-decomp-5} gives the result.
\end{proof}

Recall that an irreducible balanced sutured manifold is \emph{taut} if $R(\gamma)$ is Thurston-norm minimizing in $H_2(Y,\Gamma)$.
\begin{corollary}\label{cor:taut-twist-SFH-nontriv}
  If $(Y,\Gamma)$ is a taut balanced sutured manifold and $M$ is a
  nontrivial module over $\ZZ[H_2(Y)]$ then $\tSFH(Y,\Gamma;M)$ is
  nontrivial. In fact, $\tSFH(Y,\Gamma;M)$ has a summand isomorphic to
  $M$ as an abelian group.
\end{corollary}
\begin{proof}
  The proof is the same as Juh\'asz's proof in the untwisted
  case~\cite[Theorem 1.4]{Juhasz08:SuturedDecomp}, using
  Proposition~\ref{prop:twist-surf-decomp} in place of the untwisted
  surface decomposition theorem.
\end{proof}

The K\"unneth theorem holds for balanced sutured manifolds for either
disjoint unions or boundary connected sums. (These operations differ
by a disk decomposition.) For ordinary connected sums, we have:

\begin{lemma}\label{lem:sut-Kunneth}\cite[Proposition 9.15]{Juhasz06:Sutured}
  \begin{enumerate}
  \item Suppose that $Y_1$ is a closed $3$-manifold and
    $(Y_2,\Gamma_2)$ is a balanced sutured 3-manifold. Then
    $\SFH(Y_1\#Y_2,\Gamma_2)\cong
    \HFa(Y_1)\otimes\SFH(Y_2,\Gamma_2)$. Moreover, given modules $M_i$
    over $\ZZ[H_2(Y_i)]$,
    $\tSFH(Y_1\#Y_2,\Gamma_2;M_1\otimes M_2)\cong
    \tHFa(Y_1;M_1)\otimes\tSFH(Y_2,\Gamma_2;M_2)$ as modules over
    $\ZZ[H_2(Y_1\#Y_2)]\cong \ZZ[H_2(Y_1)]\otimes\ZZ[H_2(Y_2)]$.
  \item Suppose $(Y_i,\Gamma_i)$, $i=1,2$, is a balanced sutured
    manifold. Then $\SFH(Y_1\#Y_2,\Gamma_1\cup\Gamma_2)\cong
    \SFH(Y_1,\Gamma_1)\otimes\SFH(Y_2,\Gamma_2)\otimes\HFa(S^2\times S^1)$. Moreover, given modules $M_i$ over $\ZZ[H_2(Y_i)]$ and a module $M$ over
    $\ZZ[t,t^{-1}]$, if we identify $H_2(Y_1\#Y_2)\cong H_2(Y_1)\oplus
    H_2(Y_2)\oplus\ZZ$ where the third summand is generated by the
    connected sum sphere, then
    $\tSFH(Y_1\#Y_2,\Gamma_1\cup\Gamma_2;M_1\otimes M_2\otimes M)\cong
    \tSFH(Y_1,\Gamma_1;M_1)\otimes\tSFH(Y_2,\Gamma_2;M_2)\otimes
    \tHFa(S^2\times S^1;M)$.
  \end{enumerate}
\end{lemma}
\begin{proof}
  The untwisted cases were proved by Juh\'asz~\cite[Proposition
  9.15]{Juhasz06:Sutured}; the twisted cases follow from the same
  arguments.
\end{proof}

Juh\'asz showed (combining~\cite[Proposition 9.18]{Juhasz06:Sutured}
and~\cite[Theorem 1.4]{Juhasz08:SuturedDecomp}) that in the untwisted
case, for an irreducible, balanced sutured manifold $(Y,\Gamma)$,
$\SFH(Y,\Gamma)\neq 0$ if and only if $(Y,\Gamma)$ is taut.
We give an indirect proof that, given $Y$, $\Gamma$ can always be
chosen satisfying this property:
\begin{proposition}\label{prop:choose-SFH-nontriv}
  For any 3-manifold $Y$ with boundary there is a choice of sutures
  $\Gamma\subset\bdy Y$ so that $\SFH(Y,\Gamma)\neq 0$. Moreover,
  $\Gamma$ can be chosen so that for each component of $\bdy Y$,
  $\chi(R_+)=\chi(R_-)$.
\end{proposition}
\begin{proof}
  Fix some parameterization of $\bdy Y$ by $F(\PMC)$ for some arc
  diagram $\PMC$ making $Y$ into a special bordered-sutured manifold
  with a single suture on each boundary, say. If we fix some filling $Y'$
  of the boundary components of $Y$ by handlebodies $H_1,\dots,H_n$
  then
  \begin{equation}\label{eq:BSA-nontriv}
    \BSA(Y)\DT\bigl(\BSD(H_1)\otimes\cdots\otimes\BSD(H_n)\bigr)
    \simeq \CFa(Y')\otimes (\FF_2\oplus\FF_2)^{\otimes (n-1)}
  \end{equation}
  (since the right side corresponds to a manifold with $n$ $S^2$ boundary components).
  Since $\HFa$ is always nontrivial (by the computation of its Euler
  characteristic if $b_1=0$ and detection of the Thurston norm if
  $b_1>0$), it follows that $\BSA(Y)\not\simeq 0$. Thus, there is some
  idempotent $\iota$ so that $\BSA(Y)\cdot\iota\not\simeq
  0$. Moreover, in the tensor product~\eqref{eq:BSA-nontriv}, the only
  $\SpinC$-structure on $\bdy H_i$ which extends over $H_i$ is the
  middle $\SpinC$-structure (the one with
  $\langle c_1(\spinc),[\bdy H_i]\rangle=0$). So, we can assume the
  idempotent $\iota$ occupies half of the $\alpha$-arcs corresponding
  to each boundary component of $Y$.
  
  Zarev showed that for each idempotent $\iota$ there is a choice of
  sutures $\Gamma_\iota$ on $\bdy Y$ so that $\BSA(Y)\cdot \iota\simeq
  \SFC(Y,\Gamma_\iota)$~\cite[Section 6.1]{Zarev:JoinGlue}. The region 
  $R_+$ on each boundary component is the union of a disk and a strip
  corresponding to each $\alpha$-arc occupied by $\iota$. Since half
  the $\alpha$-arcs are occupied in $\iota$, the Euler characteristic
  is half of the Euler characteristic of the boundary component.
\end{proof}

Notice that Proposition~\ref{prop:choose-SFH-nontriv} is constructive:
given $\PMC$, there is an explicit, finite list of possibilities for $\Gamma$.

\section{Support of Heegaard Floer homology and detection of handlebodies}\label{sec:detect-handlebodies}
\subsection{Definitions of the support}\label{sec:support-of-mod}
In this section, we recall the classical definition of the support, and then give several equivalent definitions of the support of bordered-sutured Floer homology. Only the classical case is needed for Section~\ref{sec:detect-handlebodies}, so the reader interested only in Theorem~\ref{thm:detect-hb} might read the next paragraph and then skip to Section~\ref{sec:HF-support}.

We recall the classical definition of the support of a module.
Given a module $M$ over a commutative ring $R$, let
$\Ann(M)=\{r\in R\mid rm=0\ \forall m\in M\}$ denote the
\emph{annihilator} of $M$, and let $V(\Ann(M))$ be the set of prime
ideals in $R$ containing $\Ann(M)$ (a subvariety of $\Spec(R)$). The
\emph{support} of $M$, $\Supp(M)$, is the set of prime ideals
$I\subset R$ so that $M_I\neq 0$ (where $M_I$ denotes $M$ localized at
$I$, that is, with all elements of $R$ not lying in $I$ inverted). If
$M$ is finitely generated then
$\Supp(M)=V(\Ann(M))$~\cite[pp. 25--26]{Matsumura89:crt}. In
particular, since $\tHFa(Y)$ is finitely generated over
$R=\FF_2[H_2(Y)]\cong \FF_2[x_1^{\pm1},\dots,x_n^{\pm1}]$,
$\Supp(\tHFa(Y))=V(Ann \tHFa(Y))$, and similarly for twisted sutured
Floer homology.

The rest of this section, about the support of the bordered-sutured
modules, is not needed until Section~\ref{sec:HF-detect}.
Given a
bordered-sutured 3-manifold $(Y,\Gamma,\phi\co F(\PMC)\to \bdy Y)$
with $H_2(Y,F)\cong \ZZ^n$, the totally
twisted bordered-sutured module $\tBSD(Y)$ is a module over
$\Alg(\PMC)\otimes \FF_2[H_2(Y,F)]\cong \Alg(\PMC)\otimes
\FF_2[x_1^{\pm1},\dots,x_n^{\pm1}]$. The \emph{support} of $\tBSD(Y)$ is the set
of prime ideals in
\[
  R=\FF_2[H_2(Y,F)]\cong\FF_2[x_1^{\pm1},\dots,x_n^{\pm1}]
\]
so that $R_I\otimes_R\tBSD(Y)$ is not chain homotopy equivalent to the
trivial module. This condition has several equivalent formulations:
\begin{lemma}\label{lem:support-homology}
  The following conditions on a prime ideal $I\subset R$ are
  equivalent:
  \begin{enumerate}[label=(\arabic*)]
  \item\label{item:supp-TFAE1} The module $R_I\otimes_R\tBSD(Y)$ is
    chain homotopy equivalent to the trivial module.
  \item\label{item:supp-TFAE2} The homology of $R_I\otimes_R\tBSD(Y)$
    vanishes.
  \item\label{item:supp-TFAE3} The tensor product
    $R_I\otimes_R H_*\tBSD(Y)$ vanishes.
  \end{enumerate}
\end{lemma}
\begin{proof}
  Obviously Condition~\ref{item:supp-TFAE1} implies Condition~\ref{item:supp-TFAE2}. Quasi-isomorphism and homotopy equivalence agree for type $D$ structures which are homotopy equivalent to bounded ones (as $\BSD(Y)$ and $\tBSD(Y)$ are)~\cite[Corollary 2.4.4]{LOT2}. So, Condition~\ref{item:supp-TFAE2} implies Condition~\ref{item:supp-TFAE1}, since Condition~\ref{item:supp-TFAE2} implies the inclusion of the trivial module is a quasi-isomorphism. The equivalence of Conditions~\ref{item:supp-TFAE2} and~\ref{item:supp-TFAE3} follows from the fact that homology commutes with localization.
\end{proof}

We can obtain the same object by quotienting by the augmentation ideal
in $\Alg(\PMC)$. That is, let $\Idem(\PMC)$ be the direct sum over
basic idempotents $\iota$ of $\FF_2$; $\Idem(\PMC)$ is
the usual ground ring for $\Alg(\PMC)$. There is an augmentation
$\Alg(\PMC)\to\Idem(\PMC)$ which sends all Reeb chords to $0$; this
makes $\Idem(\PMC)$ into an $\Ainf$ $\Alg(\PMC)$-module (where almost
all $\Ainf$-operations vanish). Given a bordered-sutured manifold $Y$ as
above, we can then form a chain complex
$\Idem(\PMC)\DT_{\Alg(\PMC)} \tBSD(Y)$ over $R$; this is the
result of quotienting $\tBSD(Y)$ by the augmentation ideal in
$\Alg(\PMC)$.

\begin{proposition}\label{prop:tCFD-supp-reinterp}
  The support of $\tBSD(Y)$ is the same as the support of the
  $R$-module $H_*\bigl(\Idem(\PMC)\DT_{\Alg(\PMC)} \tBSD(Y)\bigr)$.
\end{proposition}
\begin{proof}
  First, suppose that $I\not\in \Supp(\tBSD(Y))$. Then
  $\tBSD(Y)\otimes R_I$ is homotopy equivalent to the trivial
  module, so $(\Idem(\PMC)\otimes R)\DT \tBSD(Y)\otimes R_I$ is
  homotopy equivalent to the trivial complex, so $I$ is not in the support
  of $H_*\bigl(\Idem(\PMC)\DT \tBSD(Y)\bigr)$.

  Conversely, suppose that %
  $I\not\in \Supp H_*\bigl(\Idem(\PMC)\DT \tBSD(Y)\bigr)$.
  Decompose the algebra $\Alg(\PMC)$ as $\Idem\oplus \Alg_+$, where
  $\Alg_+$ is the augmentation ideal.  Using
  the basis given by the generators, view the differential on
  $\tBSD(Y)\otimes R_I$ as a matrix $A$ and decompose $A=A_0+A_+$
  where the entries of $A_0$ are in $\Idem\otimes R_I$ and the entries
  of $A_+$ are in $\Alg_+\otimes R_I$. The hypothesis on $I$ is
  equivalent to $A_0$ being invertible over $\Idem\otimes R_I$. (More
  precisely, there is a matrix $B_0$ so that $A_0B_0$ and $B_0A_0$ are
  diagonal with basic idempotents on the diagonal.) The matrix $A_+$
  is nilpotent (because $\Alg_+$ is), so $A$ is also invertible. Thus,
  $\tBSD(Y)\otimes R_I$ is contractible, so $I\not\in \Supp \tBSD(Y)$,
  as desired.
\end{proof}

Note that, given a bordered-sutured Heegaard diagram $\HD$ for $Y$,
$\Idem(\PMC)\DT_{\Alg(\PMC)}\tBSD(\HD)$ is isomorphic, as a chain
complex over $\Idem(\PMC)\otimes R$, to $\tBSA(\HD)$, where we view
$\tBSA(\HD)$ as a module over $\Idem(\PMC)\otimes R$ by restriction of scalars.

Because nontrivial algebra elements act by zero on
$\Idem(\PMC)\DT_{\Alg(\PMC)}\tBSD(\HD)$, the module is in fact
induced from a module over $\FF_2[H_2(Y)]$ via the inclusion
$\FF_2[H_2(Y)]\into \FF_2[H_2(Y,F)]$. This is slightly easier to
see for $\tBSA(\HD)$. Recall that as an $\Ainf$-module,
$\tBSA(\HD)$ decomposes along $\SpinC$-structures, as
\[
  \tBSA(\HD)\simeq \bigoplus_{\spinc\in\SpinC(Y)}\tBSA(\HD,\spinc).
\]
We can also consider relative $\SpinC$-structures $\SpinC(Y,\bdy Y)$
extending any given $\SpinC$-structure on $\bdy Y$, and as a chain
complex over $\Idem(\PMC)\otimes R$ we have a decomposition
\[
  \tBSA(\HD)\simeq \bigoplus_{\mathfrak{t}\in\SpinC(\bdy Y)}
  \bigoplus_{\substack{\spinc'\in\SpinC(Y,\bdy Y)\\ \spinc'|_{\bdy Y}=\mathfrak{t}}}
  \tBSA(\HD,\spinc').
\]
Any two generators of $\tBSA(\HD,\spinc')$ differ by a provincial
domain, and up to isomorphism, $\tBSA(\HD,\spinc')$ is given by
choosing a base generator and a provincial domain connecting that base
generator to every other generator, and proceeding as in
Section~\ref{sec:twisted-HF}. Then every term in the differential of
a generator has coefficient in $\FF_2[H_2(Y)]\subset
\FF_2[H_2(Y,F)]$. So, $\tBSA(\HD,\spinc')$ is induced from a complex
$\tBSA(\HD,\spinc';\FF_2[H_2(Y)])$ over $\FF_2[H_2(Y)]$. Let
\[
  \tBSA(\HD;\FF_2[H_2(Y)])
  =\bigoplus_{\spinc'}  \tBSA(\HD,\spinc';\FF_2[H_2(Y)]).
\]
Again, this is a chain complex over $\Idem(\PMC)\otimes\FF_2[H_2(Y)]$,
not over $\Idem(\PMC)\otimes\FF_2[H_2(Y,F)]$, and
\[
  \Idem(\PMC)\DT_{\Alg(\PMC)}\tBSD(\HD)\cong
  \tBSA(\HD)\cong
  \tBSA(\HD;\FF_2[H_2(Y)])\otimes_{\FF_2[H_2(Y)]}\FF_2[H_2(Y,F)].
\]

\begin{lemma}
  The support of $\tBSA(Y)$ over $\FF_2[H_2(Y,F)]$ is the
  pullback of the support of $\tBSA(Y;\FF_2[H_2(Y)])$ over
  $\FF_2[H_2(Y)]$ with respect to the projection
  $\pi\co \Spec\FF_2[H_2(Y,F)]\to \Spec\FF_2[H_2(Y)]$ induced by the
  inclusion $i\co \FF_2[H_2(Y)]\into \FF_2[H_2(Y,F)]$.
\end{lemma}
\begin{proof}
  By~\cite[\href{https://stacks.math.columbia.edu/tag/07T8}{Lemma 10.40.4}]{stacks-project},
  since $\FF_2[H_2(Y,F)]$ is flat over $\FF_2[H_2(Y)]$ and
  $H_*\tBSA(Y)$ is finitely generated, the annihilator
  $\Ann_{\FF_2[H_2(Y,F)]}H_*\tBSA(Y;\FF_2[H_2(Y,F)])$ is the ideal
  in $\FF_2[H_2(Y,F)]$ generated by the image of
  $\Ann_{\FF_2[H_2(Y)]}H_*\tBSA(Y)$. This implies the result.
\end{proof}

\begin{corollary}\label{cor:supp-tBSA-rankK}
  Let $K=\ker(H_1(F)\to H_1(Y))$. Then
  \[
    \dim\Supp(\tBSA(\HD)) = \dim\Supp(\tBSA(\HD;\FF_2[H_2(Y)]))+\rank(K).
  \]
\end{corollary}

Of course, once we have tensored with $\Idem(\PMC)$ or viewed $\tBSA$
as merely a chain complex, we are no longer really in the world of
bordered Floer homology: by Zarev's work~\cite[Section 6.1]{Zarev:JoinGlue},
the result is a sum of sutured Floer homology groups. So:

\begin{corollary}\label{cor:finite-test-mflds}
  Given a bordered-sutured 3-manifold $Y$ with bordered boundary $-F$, there are
  finitely many bordered-sutured 3-manifolds $Y'_1,\dots,Y'_m$ with bordered boundary $F$ so that
  $I$ lies in the support of $\tBSD(Y)$ if and only if there is an $i$ so that
  $\tSFH(Y'_i\cup Y;\FF_2[H_2(Y,F)])\otimes R_I\neq 0$. Further, each $Y'_i$ can be chosen to be of the
  form $[0,1]\times F$ where $\{1\}\times F$ is the bordered boundary
  and the rest is sutured boundary, with some choice of sutures.
\end{corollary}
Here, $\FF_2[H_2(Y,F)]$ is a module over $\FF_2[H_2(Y'_i\cup Y)]$ via the map $H_2(Y'_i\cup Y)\to H_2(Y'_i\cup Y,Y'_i)=H_2(Y,F)$.
\begin{proof}
  For each basic idempotent $\iota$ in $\Idem$ there is a
  corresponding bordered-sutured manifold $Y_\iota$ so that
  $\BSA(Y_\iota)\simeq \FF_2\langle\iota\rangle$, as $\Ainf$-modules
  over $\Alg(\PMC)$. By
  Proposition~\ref{prop:tCFD-supp-reinterp}, $I$ is in the support of $\tBSD(Y)$ if and only if $I$ is in the support of $\BSA(Y_\iota)\DT\tBSD(Y)$ for some $\iota$. By the twisted pairing theorem (Theorem~\ref{thm:twisted-pairing}), $\BSA(Y_\iota)\DT\tBSD(Y)\cong \tSFC(Y'_\iota\cup Y;\FF_2[H_2(Y,F)])$. Finally, as in Lemma~\ref{lem:support-homology}, the complex $\tSFC(Y'_\iota\cup Y;\FF_2[H_2(Y,F)])$ and its homology $\tSFH(Y'_\iota\cup Y;\FF_2[H_2(Y,F)])$ have the same support.
\end{proof}

\subsection{The support of Heegaard Floer homology}\label{sec:HF-support}

The goal of this section is to prove Lemma~\ref{lem:sphere-support},
that the support of twisted Heegaard Floer homology determines the
number of disjoint, homologically independent 2-spheres, and then use it
to deduce Theorem~\ref{thm:detect-hb}.

\begin{lemma}\label{lem:sphere-support}
  \begin{enumerate}[label=(\arabic*)]
  \item\label{item:ss-closed} Suppose $Y$ is a closed
    3-manifold. Then the maximum number of
    linearly independent homology classes in $H_2(Y;\QQ)$ that can be
    represented by embedded $2$-spheres is equal to
    \(
      \rank H_2(Y)-\dim \Supp(\tHFa(Y)).
    \)
  \item\label{item:ss-sutured} Let $(Y,\Gamma)$ be a connected sutured
    3-manifold (with non-empty boundary) so that
    $\SFH(Y,\Gamma)\neq 0$ and for each component of $\bdy Y$,
    $\chi(R_+)=\chi(R_-)$.  Then the maximal number of linearly
    independent homology classes in $H_2(Y;\QQ)$ that can be
    represented by embedded $2$-spheres is equal to
    \( \rank H_2(Y)-\dim \Supp(\tSFH(Y)).  \)
  \item\label{item:ss-bs} Suppose $Y$ is a special bordered-sutured
    3-manifold with bordered boundary $F$ and $k$ sutures on its
    boundary (so $R_\pm$ each consists of $k$ disks). Assume $F$ has
    $1\leq n\leq k$ connected components.  Then the maximal number of linearly
    independent homology classes in $H_2(Y;\QQ)$ that can be
    represented by embedded $2$-spheres is equal to
    \begin{equation}\label{eq:tBSD-support-dim}
      \rank H_2(Y,F)-\dim \Supp(\tBSD(Y))+n-k.
    \end{equation}
  \end{enumerate}
\end{lemma}

\begin{proof}
  \ref{item:ss-closed}
  Let $n$ be the dimension of $H_2(Y;\QQ)$ and 
  let $k$ be the maximum number of linearly independent homology
  classes in $H_2(Y,\QQ)$ which can be represented by embedded
  $2$-spheres. By Lemma \ref{lem:disj-sphere}, we may decompose $Y$ as
  \[
    Y=Y'\#(\#^kS^1\times S^2),
  \]
and thus, $\tHFa(Y)=\tHFa(Y')\otimes_{\FF_2}\tHFa(\#^kS^1\times S^2)$. Here, $Y'$ does not contain any homologically essential $2$-spheres. 
Choose an identification $\FF_2[H_2(Y)]\cong
\FF_2[x_1^{\pm1},\dots,x_n^{\pm1}]$ so that $\FF_2[H_2(\#^kS^1\times
S^2)]\cong\FF_2[x_{1}^{\pm1},\dots,x_{k}^{\pm1}]$.
A direct computation shows that $\tHFa(\#^kS^1\times S^2)\cong \FF_2[x_1^{\pm1},\dots,x_k^{\pm1}]/I$ where 
\[I=(x_1-1,\dots,x_k-1).\]
Since $I$ is a maximal ideal, $\Supp\tHFa(\#^kS^1\times S^2)=V(I)=\{I\}$.

Let $\omega$ be a closed, generic $2$-form on $Y'$,
i.e., so that $e_{\omega}\co H_2(Y')\to \RR$ defined as
$e_{\omega}(A)=\int_{A}\omega$ is injective. Let $\Lambda$ denote the universal Novikov field, a completion of $\FF_2[\RR]$. The map $e_\omega$ induces an injective ring
map $\psi_{\omega}\co \FF_2[H_2(Y')]\to\Lambda$ by setting
$\psi_\omega(e^\alpha)=T^{e_{\omega}(\alpha)}$ (making $\Lambda$ into an algebra over $\FF_2[H_2(Y')]$) and thus a field
homomorphism $\imath_{\omega}\co \FF_2(H_2(Y'))\to\Lambda$.  As
a result,
\[
  \tHFa(Y';\FF_2(H_2(Y')))\otimes\Lambda_{\omega}\cong\tHFa(Y';\Lambda)\neq 0,
\]
(where the non-vanishing is by~\cite[Theorem 1.1]{AL19:incompressible}). Hence, $\tHFa(Y')\otimes_{\FF_2[H_2(Y')]}\FF_2(H_2(Y'))\cong\tHFa(Y';\FF_2(H_2(Y')))\neq 0$.
Hence, $\Supp\tHFa(Y')=\mathrm{spec}(\FF_2[H_2(Y')])$.

Finally,
\[
  \begin{split}
\Supp\bigl(\tHFa(Y')\otimes_{\FF_2}\tHFa(\#^kS^1\times S^2)\bigr)&=\Supp\bigl(\tHFa(Y')\bigr)\times \Supp\bigl(\tHFa(\#^kS^1\times S^2)\bigr)\\
&\cong\mathrm{spec}(\FF_2[x_{k+1}^{\pm1},\dots,x_{n}^{\pm1}])\times\{I\}
  \end{split}
\]
has dimension $n-k$, as desired.

\ref{item:ss-sutured} As in the case of closed 3-manifolds, by
Lemma~\ref{lem:disj-sphere} we can decompose $(Y,\Gamma)$ as
\[
  Y'\#\overbrace{(S^2\times S^1)\#\cdots\# (S^2\times S^1)}^k\#(Y_1,\Gamma_1)\#\cdots\#(Y_{n-k+1},\Gamma_{n-k+1})
\]
where $Y'$ is closed and contains no homologically essential $2$-spheres and each $(Y_i,\Gamma_i)$ is an
aspherical balanced sutured manifold.
Here, $n$ is the number of linearly independent homology classes in
$H_2(Y)$ represented by disjoint, embedded 2-spheres.
By Lemma~\ref{lem:sut-Kunneth},
\[
  \tSFH(Y,\Gamma)\cong \tHFa(Y')\otimes\tHFa(S^2\times S^1)^{\otimes
    n}\otimes\tSFH(Y_1;\Gamma_1)\otimes\cdots\otimes\tSFH(Y_{n-k+1},\Gamma_{n-k+1}).
\]
Since $Y'$ is aspherical, the dimension of the support of $\tHFa(Y')$
is $\rank H_2(Y')$. The dimension of the support of
$\tHFa(S^2\times S^1)^{\otimes n}$ is zero.
Since $\SFH(Y,\Gamma)\neq 0$, by the K\"unneth theorem with untwisted
coefficients, $\SFH(Y_i,\Gamma_i)\neq 0$, as well. Hence,
$(Y_i,\Gamma_i)$ is taut~\cite[Proposition 9.17]{Juhasz06:Sutured}.
So, by Corollary~\ref{cor:taut-twist-SFH-nontriv}, the dimension of
the support of $\tSFH(Y_i,\Gamma_i)$ is $\rank H_2(Y_i)$, and so
\[
  \dim\Supp\tSFH(Y,\Gamma)=\rank H_2(Y)+\sum_i\rank
  H_2(Y_i)=\rank H_2(Y)-n,
\]
as claimed.

\ref{item:ss-bs} We start by reducing to the case that each component of
$\bdy Y$ has a single suture. So, suppose we know the result for some
special bordered-sutured manifold $Y$, and let $Y'$ be a
bordered-sutured manifold obtained by adding a suture to a boundary
component of $Y$. Gluing a bordered-sutured Heegaard diagram for $Y$ to
the diagram shown on the right of Figure~\ref{fig:bs-handlebody-diag}
gives a bordered-sutured Heegaard diagram for $Y'$. By
Proposition~\ref{prop:tCFD-supp-reinterp}, gluing on this diagram has
no effect on the support of $\tBSD$, and of course does not change the
number $n$ of boundary components. On the other hand,
$\rank H_2(Y,F)$ and $k$ both increase by $1$, so
Formula~\eqref{eq:tBSD-support-dim} is unchanged.

So, from now on, assume that $n=k$.
Let $s$ be the number of linearly independent homology classes in
$H_2(Y;\QQ)$ that can be represented by disjoint, embedded
$2$-spheres. By Lemma~\ref{lem:disj-sphere}, we can decompose
\[
  Y= Y'\#\overbrace{(S^2\times S^1)\#\cdots\# (S^2\times
    S^1)}^\ell\#Y_1\#\cdots\#Y_{s-\ell+1}.
\]
By the K\"unneth theorem, Theorem~\ref{thm:bs-Kunneth},
\[
  \tBSD(Y)\simeq
  \tCFa(Y')\otimes\tBSD(Y_1)\otimes\cdots\otimes\tBSD(Y_{s-\ell+1})\otimes
  \tCFa(S^2\times S^1)^{\otimes s}.
\]
Thus, the dimension of the support of $\tBSD(Y)$ is
\begin{align*}
  \dim\Supp(&\tCFa(Y'))+\dim\Supp(\tBSD(Y_1))+\cdots+\dim\Supp(\tBSD(Y_{s-\ell+1}))\\
  &\leq\rank H_2(Y')+\rank H_2(Y_1,F_1)+\cdots+\rank H_2(Y_{s-\ell+1}, F_{s-\ell+1})\\
  &=\rank H_2(Y,F)-s.
\end{align*}
Further, equality holds if and only if for each of the irreducible
special bordered-sutured manifolds $Y_i$,
$\dim\Supp\tBSD(Y_i)=\rank H_2(Y_i,F_i)$.

So, to complete the proof, suppose that $Y$ is an irreducible special
bordered-sutured manifold. By
Proposition~\ref{prop:choose-SFH-nontriv}, there is a choice of sutures $\Gamma$
on $\bdy Y$ so that $\SFH(Y,\Gamma)\neq 0$, and so that for each
component of $\bdy Y$, $\chi(R_+)=\chi(R_-)$. Zarev~\cite[Section
6.1]{Zarev:JoinGlue} showed that there is a module $M$ so that
$\SFC(Y,\Gamma)\simeq M\DT\BSD(Y,F)$, corresponding to
$[0,1]\times F$ with bordered boundary on one side and sutured
boundary on the other.

By Part~\ref{item:ss-sutured} of the lemma,
$\dim\Supp(\tSFC(Y,\Gamma))=\rank H_2(Y)$.  By the twisted pairing
theorem, $\tSFC(Y,\Gamma;\FF_2[H_2(Y,F)])\simeq M\DT\tBSD(Y,F)$,
where $\FF_2[H_2(Y,F)]$ is viewed as an $\FF_2[H_2(Y)]$-module via the
inclusion map $H_2(Y)\to H_2(Y,F)$. So, $\dim\Supp\tBSD(Y,F)\geq\dim\Supp_{\FF_2[H_2(Y,F)]}\tSFC(Y,\Gamma;\FF_2[H_2(Y,F)])$. Write
$H_2(Y,F)\cong H_2(Y)\oplus K$, where $K=\ker(H_1(F)\to H_1(Y))$.
As in Corollary~\ref{cor:supp-tBSA-rankK}, 
\begin{equation}\label{eq:support-change-ring}
\dim\Supp_{\FF_2[H_2(Y,F)]}(\tSFC(Y,\Gamma;\FF_2[H_2(Y,F)]))
=\dim\Supp_{\FF_2[H_2(Y)]}(\tSFC(Y,\Gamma))+\rank(K),
\end{equation}
which is equal to $\rank H_2(Y,F)$.
This proves the result.
\end{proof}

Recall that Theorem~\ref{thm:detect-hb} asserts that the bordered
Floer modules detect handlebodies among
irreducible homology handlebodies, via the
dimension of the support of a twisted endomorphism space.

\begin{proof}[Proof of Theorem~\ref{thm:detect-hb}]
  From the duality result for $\CFDa$ (\cite[Theorem 2]{LOTHomPair}, a special case of Theorem~\ref{thm:bs-dual}),
  \begin{multline*}
    \Mor_{\Alg(F)}\bigl(\lsup{\Alg(F)}\CFDa(Y),\lsup{\Alg(F)}\tCFDAa(\Id)_{\Alg(F)}\DT\lsup{\Alg(F)}\CFDa(Y)\bigr)\\
    \simeq\CFAa(-Y)_{\Alg(F)}\DT\lsup{\Alg(F)}\tCFDAa(\Id)_{\Alg(F)}\DT\lsup{\Alg(F)}\CFDa(Y).
  \end{multline*}
  By the twisted pairing theorem (Theorem~\ref{eq:twisted-pairing}
  and, in particular, Formula~\eqref{eq:twist-pair-Id}), this is
  homotopy equivalent to
  \[
    \CFAa(-Y)_{\Alg(F)}\DT\lsup{\Alg(F)}\tCFDa\bigl(Y;\FF_2[H_1(F)]\bigr)\simeq \tCFa\bigl(D(Y);\FF_2[H_1(F)]\bigr)
  \]
  where $D(Y)=-Y\cup_\bdy Y$ is the double of $Y$ across its
  boundary. Choose a splitting $H_1(F)\cong H_2(Y,F)\oplus
  H_1(Y)\cong \ZZ^g\oplus\ZZ^g$. Then
  \[
    \FF_2[H_1(F)]\cong
    \FF_2[H_2(Y,F)]\otimes_{\FF_2}\FF_2[H_1(Y)]\cong \FF_2[H_2(D(Y))]\otimes_{\FF_2}\FF_2[H_1(Y)]
  \]
  as modules over $\FF_2[H_2(Y,F)]\cong\FF_2[H_2(D(Y))]$. So,
  \[
    \tCFa\bigl(D(Y);\FF_2[H_1(F)]\bigr)\cong
    \tCFa(D(Y))\otimes \FF_2[H_1(Y)].
  \]
  In particular,
  \begin{align*}
    \dim\Supp_{\FF_2[H_1(F)]}\tCFa\bigl(D(Y);\FF_2[H_1(F)]\bigr)&=\dim\Supp\tCFa(D(Y))+g\\
    \dim\Supp_{\FF_2[H_2(Y,F)]}\tCFa\bigl(D(Y);\FF_2[H_1(F)]\bigr)&=\dim\Supp\tCFa(D(Y)).
  \end{align*}
  Now, by Lemma~\ref{lem:sphere-support}, $H_2(D(Y))$ is
  generated by 2-spheres if and only if the dimension of the support
  of $\tCFa(D(Y))$ is zero. By
  Lemma~\ref{lem:double-handlebody}, $H_2(D(Y))$ is
  generated by 2-spheres if and only if $Y$ is a handlebody, so this
  implies the result.
\end{proof}

Recall that the bordered modules $\CFDa(Y)$ associated to 3-manifolds
with boundary are algorithmically computable~\cite{LOT4}, as are their
bimodule analogues (see also~\cite{AL19:incompressible}). As we will
discuss in Section~\ref{sec:compute-support}, the dimension of the
support is also computable, so Theorem~\ref{thm:detect-hb} gives an
algorithm to test whether an irreducible manifold is a handlebody. 
(This problem probably also has a well-known solution using
normal surface theory.)

We conclude the section by proving the version of the fact that bordered Floer homology detects handlebodies announced in the abstract:
\begin{proof}[Proof of Corollary~\ref{cor:detect-hb}]
  We first use the gradings to deduce that $Y$ is a homology handlebody.
  (This is the only part of the argument that uses the gradings.) The
  set of orbits in the grading set for $\CFDa(Y)$ is in bijection with
  the $\SpinC$-structures on $Y$, and hence with $H^2(Y)$. Since
  $\CFDa(Y)$ is graded homotopy equivalent to $\CFDa(H)$, $H^2(Y)=0$. So, by Lefschetz duality, $H_1(Y,\bdy Y)=0$. Thus, from the long exact sequence for the pair $(Y,
  \bdy Y)$, $H_1(\bdy Y)$ surjects onto $H_1(Y)$, so $Y$ is a homology handlebody.

  Now, decompose $Y$ as $Y'\# Y''$ where $Y'$ is irreducible and $\bdy
  Y'=\bdy Y$ (so $Y''$ is closed). Since $Y$ is a homology handlebody,
  $Y''$ is an integer homology sphere and $Y'$ is a homology handlebody.
  From the K\"unneth theorem for bordered Floer homology (a special
  case of Theorem~\ref{thm:bs-Kunneth}), $\CFDa(Y)\simeq
  \CFDa(Y')\otimes_{\FF_2}\CFa(Y'')$. Thus,
  \begin{multline*}
    H_*\Mor_{\Alg(F)}\bigl(\lsup{\Alg(F)}\CFDa(Y),\lsup{\Alg(F)}\tCFDAa(\Id)_{\Alg(F)}\DT\lsup{\Alg(F)}\CFDa(Y)\bigr)\\
    \cong H_*\Mor_{\Alg(F)}\bigl(\lsup{\Alg(F)}\CFDa(Y'),\lsup{\Alg(F)}\tCFDAa(\Id)_{\Alg(F)}\DT\lsup{\Alg(F)}\CFDa(Y')\bigr)\otimes_{\FF_2} \HFa(Y'').
  \end{multline*}
  Thus, since the support of the left side is zero-dimensional (by Theorem~\ref{thm:detect-hb}), the support of $$
  H_*\Mor_{\Alg(F)}\bigl(\lsup{\Alg(F)}\CFDa(Y'),\lsup{\Alg(F)}\tCFDAa(\Id)_{\Alg(F)}\DT\lsup{\Alg(F)}\CFDa(Y')\bigr)
  $$
  is also zero-dimensional. So, by Theorem~\ref{thm:detect-hb}, $Y'$ is a handlebody.

  It remains to see that $Y''$ is an $L$-space. The homologies of $\Mor_{\Alg(F)}(\CFDa(Y),\CFDa(Y))$ and $\Mor_{\Alg(F)}(\CFDa(Y),\CFDa(Y))$ both have dimension $2^g$ (since both morphism complexes compute $\HFa(\#g(S^2\times S^1))$, the result of doubling a handlebody). So, by the K\"unneth theorem, $\HFa(Y'')\cong \FF_2$.
\end{proof}

\subsection{Computability of the support}\label{sec:compute-support} 

We discuss briefly how one can compute the support of a module like
$\tHFa(Y)$ or $\tSFH(Y,\Gamma)$. For definiteness, we will focus on the
case of $\tHFa(Y)$.

To keep notation short, let $R=\FF_2[x_1^{\pm 1},\dots,x_n^{\pm 1}]$. Fix any identification of $H_2(Y)$ with $\ZZ^n$, so $\FF_2[H_2(Y)]$ is
identified with $R$.

Suppose $Y\cong \#^k(S^2\times S^1)\# Y'$
where $Y'$ does not contain any homologically essential
2-spheres. This decomposition induces an identification of $\FF_2[H_2(Y)]$ with $\ZZ[y_1^{\pm 1},\cdots,y_n^{\pm 1}]$ so that the support of
$\tHFa(Y)$ is given by $\{y_1=\cdots=y_k=1\}$. Hence, returning to
viewing $\tHFa(Y)$ as a module over $R$, it follows that the support
of $\tHFa(Y)$ is a smooth, codimension-$k$ subvariety of $\Spec R$
containing the point $(1,1,\dots,1)$ (or the ideal
$(y_1-1,\dots,y_k-1)$). Our goal is to compute $k$ without knowing the decomposition inducing the coordinates $y_1,\dots,y_n$.

Write the differential on $\tCFa(Y)$ as an $N\times N$ matrix with
coefficients in $R$. Of course, if $N$ is not even then $Y$ does not
contain any homologically essential $2$-spheres, so we may assume $N$ is
even. Let $A$ be the matrix for the differential on $\tCFa(Y)$, viewed
as a single $N\times N$ matrix.

\begin{lemma}
  An ideal $I\subset R$ is in 
  $\Supp(\tHFa(Y))$ if and only if $I$ contains all the
  $N/2\times N/2$ minors of $A$.
\end{lemma}
\begin{proof}
  If $I$ contains all the $N/2\times N/2$ minors of $A$, then over the
  field $R_I/I$, all $N/2\times N/2$ minors of $A$ vanish so $A$ has
  determinantal rank, and hence rank, less than $N/2$. Hence,
  $\tHFa(Y;R_I/I)\neq 0$. So, by the universal coefficient spectral
  sequence, 
  $$\Tor_{R_I}(\tHFa(Y;R_I), (R_I/I))\neq 0.$$
  Thus,
  $\tHFa(Y;R_I)=\tHFa(Y)_I\neq 0$ and $I\in\Supp(\tHFa(Y))$.

  Conversely, suppose $I\in\Supp(\tHFa(Y))$.  Then in the coordinates
  $x_i$, $I$ contains $y_i-1$ for all $1\le i\le k$. Let $\mathbb{F}$ be the field of
  fractions of $R/(y_1-1,\dots,y_k-1)$. If $I$ does not contain some
  $N/2\times N/2$ minor of $A$ then $(y_1-1,\dots,y_k-1)$ also does
  not contain that minor, so the minor is a non-zero element of
  $\mathbb{F}$. Hence, $A$ has determinantal rank, and hence rank, $N/2$ over
  $\mathbb{F}$, so $\tHFa(Y;\mathbb{F})=0$. On the other hand, we showed previously that
  $\tHFa(Y;\mathbb{F})\neq 0$, a contradiction.
\end{proof}

By the previous lemma, the support of $\tHFa(Y)$ is exactly the variety
associated to the set of $N/2\times N/2$ minors of $A$. In particular,
the dimension of the support is computable by familiar algorithms in
commutative algebra, using Gr\"obner bases. (Alternatively, from the
form of $\tHFa$, this variety is smooth and contains $(1,\dots,1)$, so
one can compute its dimension as the dimension of the tangent space at
$(1,\dots,1)$, if that is faster.)

The discussion above applies without changes to twisted sutured Floer
homology. 
Further, by
Corollary~\ref{cor:finite-test-mflds}, say, if one can compute the
support of twisted sutured Floer homology then one can also compute the
support of the twisted bordered-sutured invariants. So, the supports of
all modules discussed in this paper are computable.

\section{Detecting whether maps extend over a given compression body}\label{sec:HF-detect}

The goal of this section is to prove Theorem~\ref{thm:detect-over-specific}.

Recall that a \emph{bordered handlebody} is a pair $(H,\phi)$ where $H$ is a
handlebody and $\phi\co \Sigma\to \bdy H$ is a diffeomorphism from a
standard, reference surface (usually coming from a pointed matched
circle) to the boundary of $H$. A map $\psi\co \Sigma\to\Sigma$
\emph{extends over $H$} if there is a diffeomorphism $\Psi\co H\to H$
extending $\phi\circ\psi\circ \phi^{-1}$, or equivalently so that
$\phi\circ\psi = \Psi\circ \phi$:
\begin{equation}\label{eq:extends}
  \xymatrix{
    \Sigma\ar[r]^\psi\ar[d]_\phi & \Sigma\ar[d]_\phi\\
    H \ar[r]_{\Psi}  & H.
    }
\end{equation}
That is, $(H,\phi)$ and $(H,\phi\circ\psi)$ are diffeomorphic
(equivalent) bordered manifolds.
Similarly, given a half-bordered compression body $(C,\phi)$, a
diffeomorphism $\psi\co \Sigma\to \Sigma$ (of the outer boundary) \emph{extends
  over $(C,\phi)$} if there is a diffeomorphism $\Psi\co C\to C$
extending $\phi\circ\psi\circ\phi^{-1}$.

\begin{proposition}\label{prop:mc-S2S1}
  Let $(C,\phi)$ be a half-bordered compression body, with outer boundary of genus $g$ and inner boundary with $k$ connected components of genera $g_1,g_2,\dots,g_k$. Then a diffeomorphism $\psi\co\Sigma_g\to\Sigma_g$ extends over $C$ if and only if 
  \begin{equation}\label{eq:ext-comp}
    Y\coloneqq C
    \phipsiphi (-C)\cong Y_1\#Y_2\#\cdots\#Y_{k}\#(S^2\times S^1)^{g-g'}
  \end{equation}
  where $g'=g_1+\cdots+g_k$ and each $Y_i\cong [0,1]\times\Sigma_{g_i}$.
\end{proposition}
In particular, $\psi$ extends over a bordered handlebody
$(H,\phi)$ if and only if
\begin{equation}\label{eq:extends-is-sum}
  H\phipsiphi(-H)\cong (S^2\times S^1)^{\#g}.
\end{equation}

\begin{proof}
Suppose $\psi$ extends over $C$. Consider a maximal set of pairwise disjoint meridians $\amalg_{i=1}^n\gamma_i$ in $\Sigma_g$, along with pairwise disjoint, properly embedded disks $\amalg_{i=1}^{n}D_i$ in $C$ such that $\bdy D_i=\phi(\gamma_i)$ and 
\[C\setminus\left(\amalg_{i=1}^{n}\nbd(D_i)\right)\cong\amalg_{i=1}^k \left([0,1]\times \Sigma_{g_i}\right).\]
Here, $n=g-g'+k-1$. Since $\psi$ extends over $(C,\phi)$, each $\psi(\gamma_i)$ is a meridian for $C$ with $\phi\circ\psi(\gamma_i)=\bdy \Psi(D_i)$, where $\Psi$ denotes the extension of $\psi$ over $C$. Further, the disks $\Psi(D_1), \cdots, \Psi(D_n)$ split $C$ into product $3$-manifolds, so we are done. 

Conversely, suppose Equation~\eqref{eq:ext-comp} holds. Then 
\[H_1(Y)\cong \ZZ^{g+g'}\cong \frac{H_1(\Sigma)}{L+\psi_*L}\]
where $L=\ker\left(\phi_*:H_1(\Sigma)\to H_1(C)\right)$. Note that $H_1(\Sigma)/L\cong H_1(C)\cong \ZZ^{g+g'}$. Thus, $L=\psi_*L$, and so the inclusion $C\subset Y$ induces an isomorphism from $H_1(C)$ to $H_1(Y)$. Consequently, in the long exact sequence for the pair $(Y,-C)$ the map 
\[H_2(Y)\to H_2(Y,-C)\cong H_2(C,\bdy_{\mathit{out}}C)\]
is surjective. The kernel of this map is equal to the image of $H_2(-C)$, and so $H_2(\bdy_{\mathit{in}}(-C))$, under the inclusion map.   
Write $\bdy_0Y=\bdy_{\mathit{in}}(-C)$ and $\bdy_1Y=\bdy_{\mathit{in}}(C)$.
Similarly, considering the long exact sequence for the pair $(Y,C)$, the map from $H_2(Y)$ to $H_2(Y,C)$ is surjective, and its kernel is equal to $i_*(H_2(\bdy_1Y))$.

Let $S_1,S_2,\dots, S_{n}$ be pairwise disjoint embedded $2$-spheres in $Y$ satisfying the properties listed in Lemma \ref{lem:cb-Haken}, and $l=g-g'$. By sliding the spheres $S_{l+1},S_{l+2},\dots, S_{n}$ over the spheres $S_1,\dots, S_l$ we can assume $S_{i}$ is a separating sphere for any $l+1\le i\le n$. Thus, these spheres give a decomposition of $Y$ as 
\[Y=Y_1'\#Y_2'\#\cdots\#Y_{k}'\]
where $\bdy Y_i'\neq\emptyset$ for every $i$.  The inclusion maps from $H_1(\bdy_0 Y)$ and $H_1(\bdy_1 Y)$ to $H_1(Y)$ are injective so, for each $i$, $\bdy Y_i'$ contains at least one component of $\bdy_0Y$ and $\bdy_1Y$. Hence, by counting, each $\bdy Y'_i$ consists of exactly one component of $\bdy_0Y$ and one component of $\bdy_1Y$.
A similar consideration of the kernel of $H_1(\bdy Y)\to H_1(Y)$ implies that $\bdy Y_i'=\bdy Y_{j}$ for some $j$.
Further,  
\[
  H_2(Y)\cong \ZZ^{n}\oplus\left(\bigoplus_{i=1}^ki_*(H_2(\bdy_0 Y_i))\right),
\]
The complement $Y\setminus(S_1\cup\cdots\cup S_l)$ is connected, so $S_1,\dots,S_l$ induce a decomposition $Y\cong Y''\#(S^2\times S^1)^l$, and $H_2(Y'')/i_*H_2(\bdy Y'')=0$. So, $[S_1],\dots,[S_l]$ form a basis for $H_2(Y)/i_*H_2(\bdy Y)$. Thus,
\(
\{[S_i] \mid 1\le i\le n\}
\)
is a basis for $H_2(Y)/i_*(H_2(C))$. Let $D_i=S_i\cap C$ and $D'_i=S_i\cap -C$ for $1\le i\le n$. Then the long exact sequence for the pair $(Y,-C)$ implies that $[D_1],\dots,[D_n]$ form a basis for $H_2(C,\bdy_{\mathit{out}}C)$. Similarly, $[D_1'],\dots,[D'_n]$ also form a basis for $H_2(C,\bdy_{\mathit{out}}C)$. 

 Let $C'\subset C$ be the compression body obtained from $\nbd(\bdy_{\mathit{out}}C)\cup \left(\coprod_{i=1}^{n}\nbd(D_i)\right)$ by filling any inner sphere boundary components with balls. By Lemma~\ref{lem:max-set}, $C\setminus C'\cong [0,1]\times \bdy_{\mathit{in}}C$. Similarly, one can define a compression body $C''\subset C$ using $D_1',\dots, D_n'$ and $C\setminus C''\cong [0,1]\times \bdy_{\mathit{in}}C$. We extend $\psi$ to a diffeomorphism of $C$ by first defining $\Psi$ from $C'$ to $C''$ such that $\Psi(D_i)=D'_i$. Since $C\setminus C'$ and $C\setminus C''$ are products, $\Psi$ extends to a diffeomorphism on $C$.
\end{proof}

\begin{remark}
  In a somewhat related vein, Casson-Gordon showed that a mapping class extends over a compression body if and only if it preserves the subgroup of $\pi_1$ generated by a set of meridians for the compression body~\cite[Lemma 5.2]{CG83:fibered-ribbon}.
\end{remark}

The following is essentially a reformulation of Theorem~\ref{thm:detect-over-specific}; we deduce Theorem~\ref{thm:detect-over-specific} immediately after proving it.

\begin{proposition}\label{prop:tHF-detect-extend} With notation as in
  Proposition~\ref{prop:mc-S2S1}, equip $\bdy_{\mathit{in}}C$ with a special
  bordered-sutured structure, with bordered boundary $\phi'\co F'\to
  \bdy_{\mathit{in}}C$. So, $Y:=C\phipsiphi(-C)$ inherits a bordered-sutured
  structure, with bordered boundary $F'\amalg (-F')$. Then $\psi$ extends over
  $(C,\phi)$ if and only if $\psi$ preserves $\ker(\phi_*\co H_1(\Sigma)\to
  H_1(C))$ and the dimension of the support of $\tBSD(Y)$ (over
  $\FF_2[H_2(Y,-F'\amalg F')]$) is $2g'+k$. 
\end{proposition}

\begin{proof}
Suppose $\psi$ extends over $C$. Then commutativity of the Diagram~\eqref{eq:extends} implies that $\psi_*$ preserves the kernel of $\phi_*$. Further, by Proposition \ref{prop:mc-S2S1},
\[Y\cong Y_1\#Y_2\#\cdots\#Y_k\#(S^2\times S^1)^{g-g'},\]
where $Y_i\cong [0,1]\times\Sigma_{g_i}$. By the K\"{u}nneth theorem (Theorem \ref{thm:bs-Kunneth}),
\[
  \tBSD(Y)\simeq\tBSD(Y_1)\otimes \cdots\otimes\tBSD(Y_k)\otimes\tCFa(S^2\times S^1)^{\otimes g-g'+k-1}.
\]
By Lemma~\ref{lem:sphere-support}, since $Y_i$ is aspherical,
$\dim\Supp(\tBSD(Y_i))=\rank H_2(Y_i,F_i)$. Thus, 
\begin{align*}
\dim \Supp(\tBSD(Y))&=\dim\Supp(\tBSD(Y_1))+\cdots+\dim\Supp(\tBSD(Y_k))\\
&=\rank H_2(Y_1,F_1)+\cdots+\rank H_2(Y_k,F_k)+2k-2m\\
&=(2g'+2m-k)+2k-2m=2g'+k.
\end{align*}
Here, $m$ denotes the number of sutured arcs on the inner boundary of $C$ and $F_i$ denotes the bordered part  of $\bdy Y_i$. 

Turning to the converse, since $\psi_*$ preserves $\ker(\phi_*)$, the Mayer-Vietoris sequence for $Y=C\phipsiphi(-C)$ implies that $H_2(Y)/i_*(H_2(\bdy Y))\cong \ZZ^{g-g'}$. Since $\ker(H_1(\Sigma)\to H_1(C,\bdy_{\mathit{in}}(C)))$ is the symplectic orthogonal complement to $\ker(H_1(\Sigma)\to H_1(C))$ inside $H_1(\Sigma)$, $\psi_*$ preserves $\ker(H_1(\Sigma)\to H_1(C,\bdy_{\mathit{in}}(C)))$, as well. So, 
the relative Mayer-Vietoris sequence for $(Y,\bdy Y)=\left(C\phipsiphi(-C), \bdy_{\mathit{in}}(C)\cup\bdy_{\mathit{in}}(-C)\right)$ implies that $H_2(Y,\bdy Y)=\ZZ^{g+g'}$. Consequently, by the long exact sequence for the triple $(Y,\bdy Y, -F'\amalg F')$, we have 
\[
  \rank H_2(Y,-F'\amalg F')=\rank H_2(Y,\bdy Y)+2m-1=g+g'+2m-1.
\]
It follows from Lemma \ref{lem:sphere-support} that the maximal number of linearly independent homology classes that can be represented by embedded $2$-spheres is $g-g'+k-1$. Let $L'\subset H_2(Y)$ denote the linear subspace generated by these homology classes. 

On the other hand, the inclusion map from $H_1(C)$ to $H_1(Y)$ is an isomorphism, and so the long exact sequence for the pair $(Y,C)$ implies that the quotient map from $H_2(Y)$ to $H_2(Y,C)$ is surjective. Since $C$ is built from $\bdy_{\mathit{in}}(C)$ by attaching $1$-handles, $H_2(Y,C)\cong H_2(Y,\bdy_{\mathit{in}}C)$. So, the long exact sequence for the pair $(Y,\bdy_{\mathit{in}}C)$ implies that the inclusion map from $H_1(\bdy_{\mathit{in}}C)$ to $H_1(Y)$ is injective. Thus, 
\begin{equation}\label{eq:i-cap-L}
  \left(i_*H_2(\bdy_{\mathit{in}}C)\right)\cap L'=\{0\}
\end{equation}
and so 
\[
  H_2(Y;\QQ)\cong i_*H_2(\bdy_{\mathit{in}}C;\QQ)\oplus (L'\otimes\QQ).
\]

It follows from Lemma \ref{lem:disj-sphere} that $Y$ has a decomposition as 
\[
  Y=Y'\#(S^2\times S^1)^{g-g'-\ell}\#Y_1\#Y_2\#\cdots\#Y_{k+\ell},  
\]
for some integer $\ell$,
where $Y'$ is closed, while each $Y_i$ has non-empty boundary. 
Since the intersection of $i_*H_2(\bdy_{\mathit{in}}C)$ and $i_*H_2(\bdy_{\mathit{in}}(-C))$ with $L'$ is $\{0\}$, both $\bdy Y_i\cap \bdy_{\mathit{in}}C$ and $\bdy Y_i\cap
\bdy_{\mathit{in}}(-C)$ are nonempty for every $1\le i\le k+\ell$. Therefore, the boundary of each $Y_i$ consists of exactly two components, one in
$\bdy_{\mathit{in}}(C)$ and one in $\bdy_{\mathit{in}}(-C)$. Thus, $\ell\leq 0$; but $H_2(Y)/i_*(H_2(\bdy Y))\cong\ZZ^{g-g'}\supset \ZZ^{g-g'-\ell}$, so $\ell\geq 0$, and hence $\ell=0$. In
addition, injectivity of the inclusions $H_1(\bdy_{\mathit{in}}C)$ and
$H_1(\bdy_{\mathit{in}}(-C))$ into $H_1(Y)$ implies that the
components of $\bdy Y_i$ have equal genus. Therefore, we can assume
$Y_1, Y_2,\dots, Y_k$ are labeled so that $\bdy Y_i$ contains two
components of genus $g_i$.

By Lemma \ref{lem:cb-Haken} there exist pairwise disjoint, embedded $2$-spheres $S_1,S_2,\cdots, S_{g-g'+k-1}$ such that each $S_i$ intersects $\Sigma$ in a single circle. Moreover, they represent linearly independent homology classes in $H_2(Y)$ and $[S_1],\cdots, [S_{g-g'}]$ span $H_2(Y,\QQ)/i_*(H_2(\bdy Y,\QQ))$. Let $D_i=S_i\cap C$ and $D_i'=S_i\cap(-C)$ for all $i$. The long exact sequence for $(Y,-C)$ implies that the kernel of $H_2(Y)\to H_2(Y,-C)$ is equal to $i_*(H_2(-C))=i_*(H_2(\bdy_{\mathit{in}}(-C)))$, and thus by Formula~\eqref{eq:i-cap-L}, $[D_1],[D_2],\cdots, [D_{g-g'+k-1}]$ represent linearly independent homology classes in $H_2(Y,-C)=H_2(C,\bdy_{\out}C)$, and so they generate $H_2(C,\bdy_{\out}C)$. Similarly, $[D_1'], [D_2'],\cdots, [D_{g-g'+k-1}]$ generate $H_2(C,\bdy_{\out}(-C))$. Thus, by Lemma~\ref{lem:max-set}, $C$ is $\Sigma[\bdy D_1,\dots,\bdy D_{g-g'+k-1}]$ and similarly for $C'$. Thus, $Y$ has the form of Proposition~\ref{prop:mc-S2S1}, giving the result.
\end{proof}

\begin{proof}[Proof of Theorem~\ref{thm:detect-over-specific}]
  Write $Y=C\phipsiphi(-C)$. We will relate the morphism complex from
  Equation~\eqref{eq:detect-over-specific-formula} with $\tBSD(Y)$
  and then apply Proposition~\ref{prop:tHF-detect-extend}. (Specifically, we will show that the dimensions of their supports differ by $m$.)

  From the twisted pairing theorem (Theorem~\ref{eq:twisted-pairing}) and the
  duality theorem for bordered-sutured Floer homology
  (Theorem~\ref{thm:bs-dual}), the morphism complex in
  Formula~\eqref{eq:detect-over-specific-formula} is chain homotopy equivalent to $\tBSD(Y;\FF_2[H_2(C,F\cup F')])$, where $\FF_2[H_2(C,F\cup F')]$ is a module over $\FF_2[H_2(Y,F'\cup(-F'))]$ via the map $H_2(Y,F'\cup(-F'))\to H_2(Y,F'\cup(-C))=\FF_2[H_2(C,F\cup F')]$ induced by the evident map of pairs and excision.

  The bordered boundary of $Y$ is $F'\amalg (-F')$.
  The long exact sequence for the triple $(Y,F'\amalg (-C), F'\amalg (-F'))$ gives
  \begin{align*}
  0\to \ZZ^m=H_2(F'\cup(-C),F'\cup(-F'))&\to H_2(Y,F'\cup(-F'))\to H_2(Y,F'\cup(-C))\\&\to H_1(F'\cup(-C),F'\cup(-F'))\to H_1(Y,F'\cup(-F')).
  \end{align*}
  The last map is injective and, by excision, $H_2(Y,F'\cup (-C))=H_2(C,F\cup F')$, so this gives
  \[
    0\to \ZZ^m=H_2(F'\cup(-C),F'\cup(-F'))\to H_2(Y,F'\cup(-F'))\to H_2(C,F\cup F')\to 0.
  \]
  The first $\ZZ^m$ is generated by the sutured components of $\bdy_{\mathit{in}}-C$. So, the image $\ZZ^m$ in $H_2(Y,F'\cup(-F'))$ comes from the kernel of the map $H_1(F'\cup(-F'))\to H_1(Y)$, and hence is disjoint from the subspace generated by spheres (which lies in the image of $H_2(Y)\to H_2(Y,F'\cup(-F'))$). So, the exact sequence implies that
  \[
  \dim\Supp \tBSD(Y)=\dim\Supp(\tBSD(Y;\FF_2[H_2(C,F\cup F')]))+m,
  \]
  as desired.
\end{proof}

\begin{corollary}\label{cor:CFDA-detect-extension}
  Let $\psi$ be a mapping class and $(C,\phi)$ be a bordered
  handlebody or half-bordered compression body. Assume that $\psi$
  preserves $\ker(\phi_*\co H_1(\Sigma)\to H_1(C))$. Then there
  is an algorithm to determine from $\BSDA(\psi)$ whether $\psi$
  extends over $(C,\phi)$.
\end{corollary}
\begin{proof}
  As we described in our previous paper~\cite{AL19:incompressible}, the
  algorithm for computing $\HFa$ using bordered Floer theory~\cite{LOT4}
  extends easily to give an algorithm for computing the bordered-sutured
  Floer invariants of arbitrary bordered-sutured manifolds. The
  algorithm for computing $\HFa$ depends on computing the type \DD\
  invariants of the identity map, arcslide diffeomorphisms, and
  compression bodies with outer boundary a surface of genus $g$, inner
  boundary a surface of genus $g-1$, and particularly simple
  parameterizations. The extension to bordered-sutured manifolds involves
  computing the invariants of seven other simple pieces, mostly
  corresponding to changing the sutures on the boundary. In each case,
  computing the totally twisted bordered-sutured bimodule is no extra
  work. Indeed, in each of these cases, there are no provincial periodic
  domains, so the twisted invariant is determined by the boundary
  twisting $\tBSDA(Y;\FF_2[H_1(F)])$, which one can compute by taking
  the tensor product of the untwisted invariant with $\tBSDA(\Id)$, as
  in Formula~\ref{eq:twist-pair-Id}.

  Tensoring these pieces together in turn and removing excess twisting
  at each stage as in Formula~\eqref{eq:drop-extra-twisting} computes
  the twisted invariant $\tBSD(C,\phi)$; and the algorithms from earlier
  papers~\cite{LOT4,AL19:incompressible} determine $\BSD(C,\phi)$ and
  $\BSDA(\psi)$. 
  With these invariants in hand computing the morphism complexes and
  tensor products in Formula~\ref{eq:detect-over-specific-formula} is
  clearly combinatorial. As explained in
  Section~\ref{sec:compute-support}, the support of this complex is
  computable. By Theorem~\ref{thm:detect-over-specific}, this determines
  whether $\psi$ extends over $(C,\phi)$.
\end{proof}

\section{Train tracks and bordered Heegaard diagrams}\label{sec:tt-arc-diagram-slides} 

The goal of this section is to show how train track splitting sequences
give bordered-sutured Heegaard diagrams for diffeomorphisms. The first
step, in Section~\ref{sec:ad-from-tt}, is to see how a train track on a
surface gives an arc diagram for that surface (minus some disks). The
second, in Section~\ref{sec:slide-seq-loop}, is to show that a periodic
splitting sequence for a mapping class gives a sequence of arcslides
whose composition is a strongly-based representative for that mapping
class.

\subsection{Arc diagrams from train tracks}\label{sec:ad-from-tt} 
The goal of this section is to observe that, given a filling train track
on a surface (see Section~\ref{sec:traintracks}), there is an associated
arc diagram for the surface (see Section~\ref{sec:bs-background}); and
an additional choice of some switches in the train track induces a
special arc diagram for the surface.

\begin{figure}
  \centering
  \includegraphics{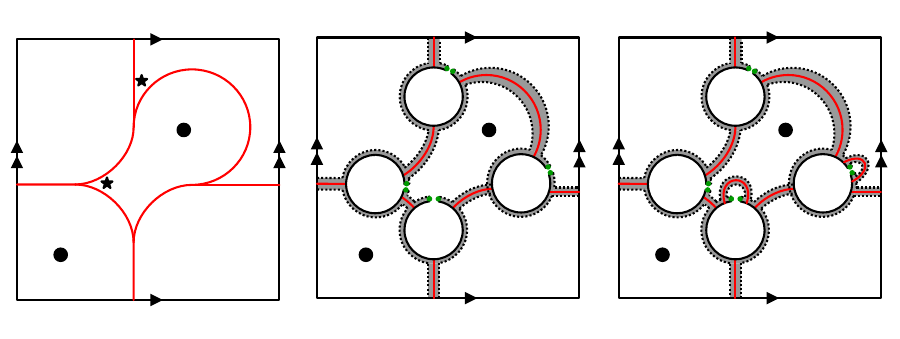}
  \caption{\textbf{Arc diagrams from train tracks.} Left: a \textcolor{red}{train track} on a twice-punctured torus. Center: the corresponding arc diagram $\PMC$ from Construction~\ref{const:tt-to-arc-diag}. The cores of the handles ($\alpha$-arcs in an associated Heegaard diagram) are \textcolor{red}{solid}, the surface $F(\PMC)$ is shaded, the intervals $S_-$ on its boundary are dotted, and the sutures $\bdy S_+=\bdy S_-$ are indicated with \textcolor{darkgreen}{small circles}. Right: the arc diagram $\PMC_\sigma$ from Construction~\ref{const:tt-to-arc-diag-sp}, where $\sigma$ chooses the two cusps marked with stars.}
  \label{fig:ad-to-tt}
\end{figure}

\begin{construction}\label{const:tt-to-arc-diag}
Let $\tau$ be a generic, filling train track on $\Sigma$. Associated to $\tau$ one can construct an arc diagram for $\Sigma$ as follows. Let $Z'_i$ be a small circle
  around the $i\th$ switch in the train track, $J_i$ an open interval
  in $Z'_i$ in the cusp region, and $Z_i=Z'_i\setminus J_i$. Orient
  each circle $Z'_i$ as the boundary of a small disk around the
  switch, and give $Z_i$ the induced orientation. Let $\mathbf{a}$ be
  the intersection of $Z=\amalg_iZ_i$ with the train track (which is the same as
  the intersection of $Z'=\amalg_iZ_i'$ with the train track). The matching $M$ exchanges the pairs of points on the same edge of the train track. Then, for $\PMC=(Z,\mathbf{a},M)$, the surface $F(\PMC)$ has an obvious embedding in $\Sigma$ as a neighborhood of $\tau$.   The intervals $S_+$ are contained in small disks around the switches, while the intervals $S_-$ are large, in the sense that they contain the boundaries of the neighborhoods of the edges of the train track. Moreover, since $\tau$ is filling, every component of $\Sigma\setminus F(\PMC)$ is either a disk or an annulus around a component of $\bdy\Sigma$. See Figure~\ref{fig:ad-to-tt}.
 \end{construction}

\begin{lemma}
  The arc diagram $\PMC$ along with the embedding $F(\PMC)\into\Sigma$ in Construction~\ref{const:tt-to-arc-diag} is, in fact, an arc diagram for $\Sigma$. That is, every component of $\bdy F(\PMC)$ intersects $S_+$.
  \end{lemma}

\begin{proof}
Every component $K$ of $\bdy F(\PMC)$ corresponds to the boundary of a component $C$ of $\Sigma\setminus\tau$, and $K$ will intersect the $S_+$ arcs corresponding to the cusps on the boundary of $C$. Thus, $\tau$ being a train track implies that $\PMC$ is an arc diagram. 
\end{proof}

Note that the arc diagram $\PMC$ constructed in Construction \ref{const:tt-to-arc-diag} is not necessarily special. In fact, the number of positive (and so negative) arcs on every component of $\bdy F(\PMC)$ is equal to the number of cusps on the boundary of the corresponding region of $\Sigma\setminus \tau$.  We modify  $\PMC$  to a special arc diagram for $\Sigma$ as follows. 

\begin{construction}\label{const:tt-to-arc-diag-sp}
With notation as in Construction~\ref{const:tt-to-arc-diag}, consider a subset $\sigma$ of the switches of $\tau$ such that if we add a star in the cusp of each switch in $\sigma$ then every component of $\Sigma\setminus\tau$ contains exactly one star. 
 For any $i$, if the $i\th$ switch of $\tau$ is not in $\sigma$ then add a pair of points on $Z_i$ such that they bound a sub arc of $Z_i$ containing all three points of $Z_i\cap \tau$. Let     
  $\mathbf{a}'$ be the union of these pairs of points, and set $\mathbf{a}_\sigma=\mathbf{a}\cup \mathbf{a}'$. Define the matching $M_\sigma$ such that it is equal to $M$ on $\mathbf{a}$ and exchanges the points of $\mathbf{a}'$ that are on the same component of $Z$. Let $\PMC_\sigma=(Z,\mathbf{a}_\sigma,M_\sigma)$.
  
For any $i$, if the $i\th$ switch is not in $\sigma$, then connect the switch to itself with a loop disjoint from $\tau$ and embedded in a small neighborhood of the switch in the cusp region. Denote the union of such loops by $\ell$. The surface $F(\PMC_\sigma)$ can be embedded into $\Sigma$ as a neighborhood of $\tau\cup\ell$.  Then, $\PMC_\sigma$ along with this embedding is a special arc diagram for $\Sigma$. Again, see Figure~\ref{fig:ad-to-tt}.
\end{construction}

If $\Sigma$ is closed, then for any $\sigma$, the embedded sutured surface $F(\PMC_\sigma)$ in $\Sigma$ is isotopic to $\Sigma\setminus \amalg_iD_i$ where $D_i\subset \Sigma$ is the open disk neighborhood of the $i\th$ switch, with $\bdy D_i=Z_i'$. Under this identification, $S_+$ and $S_-$ are identified with $\amalg_iZ_i$ and $\amalg_iJ_i$, respectively. If $\Sigma$ has some boundary components, some switches in $\sigma$ correspond instead to punctures of $\Sigma$ rather than extra disks which are deleted.

\subsection{Arc slide sequences from periodic splitting sequences}\label{sec:slide-seq-loop}

In this section, we observe that splitting a train track (see
Section~\ref{sec:traintracks}) corresponds to performing a pair of
arcslides (Definition~\ref{def:arcslide}). So, a periodic splitting
sequence for a mapping class $\psi$ induces a factorization of some
strongly based representative of $\psi$ into
arcslides (Lemma~\ref{lem:splits-to-slides-1}). If we restrict to special arc
diagrams, the arc diagrams before and after the sequence of arcslides
may not be the same, but this can be remedied by performing a few more
arcslides supported near the boundary
(Corollary~\ref{cor:factor-into-arcslides}).

\begin{figure}
  \centering
  \includegraphics[scale=.8]{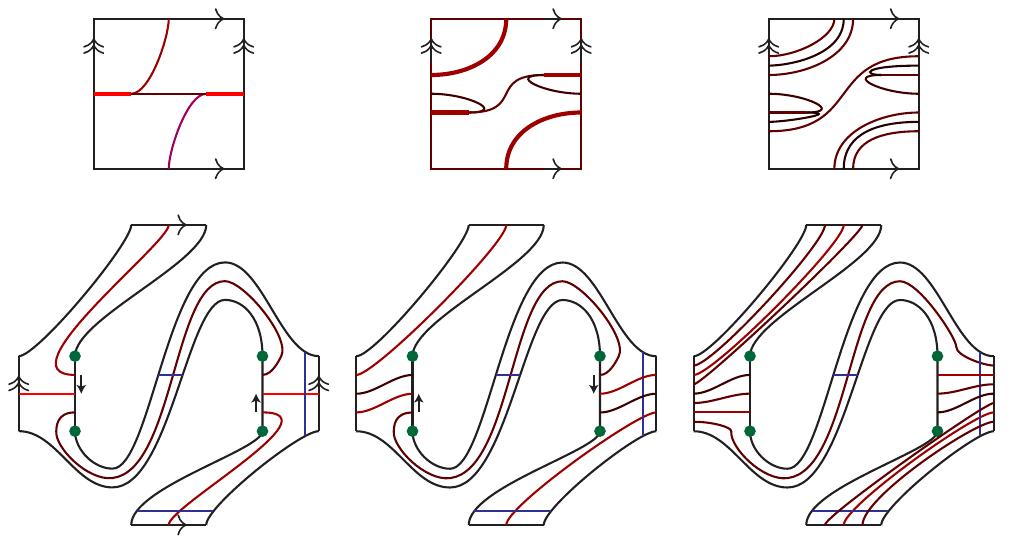}
  \caption{\textbf{Trivalent train tracks and bordered-sutured
      Heegaard diagrams.} Top: the periodic splitting sequence for the
    map $\tau_a\tau_b^{-1}$ on the torus. The arcs being split are
    thick; the splits are drawn in a particular way to correspond to
    the bottom row. Bottom: a corresponding sequence of arcslides
    starting from the half-identity diagram and
    ending at a half diagram for $\tau_a\tau_b^{-1}$. The dots
    indicate sutures and small arrows indicate where the arcslides
    occur.}
  \label{fig:splits}
\end{figure}

\begin{construction}\label{const:diff-arc-slide}
Suppose $\tau_1$ and $\tau_2$ are generic, filling train tracks on $\Sigma$ such that $\tau_2$ is obtained from $\tau_1$ by a split and they have the same number of branches.  Let $\PMC_1$ and $\PMC_2$ be the arc diagrams corresponding to $\tau_1$ and $\tau_2$, respectively, from Construction~\ref{const:tt-to-arc-diag}. Then, as depicted in Figure~\ref{fig:splits}, $\PMC_2$ is obtained from $\PMC_1$ by a pair of arcslides. There is an induced diffeomorphism from $F(\PMC_2)$ to $F(\PMC_1)$, and the half Heegaard diagram associated to that diffeomorphism is the result of performing these arcslides on the $\alpha$-arcs in the $\HalfHD(\Id_{\PMC_1})$.

\begin{figure}
  \centering
  \includegraphics{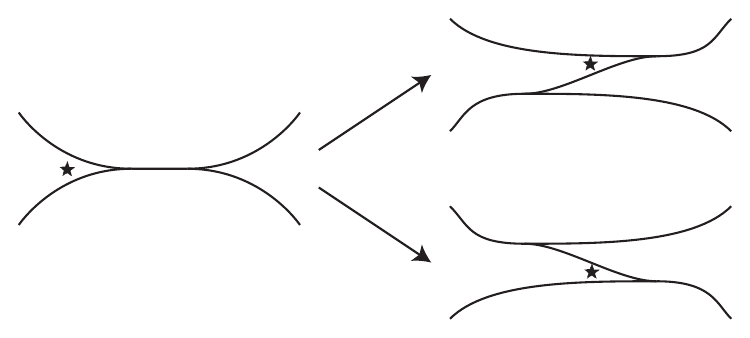}
  \caption[Induced choice of switches after a split]{\textbf{Induced choice of switches after a split}. The switch marked with a star is one of the chosen switches, before and after the split.}
  \label{fig:induced-switch}
\end{figure}

Fix a subset $\sigma_1$ of switches in $\tau_1$ as in Construction~\ref{const:tt-to-arc-diag-sp}. This induces a subset $\sigma_2$ of switches in $\tau_2$ such that the boundary of each component of $\Sigma\setminus\tau_2$ contains exactly one starred cusp (see Figure~\ref{fig:induced-switch}). (Note that the middle case of Figure~\ref{fig:tt-split} does not occur in the periodic splitting sequence.) Let $\PMC_{\sigma_1}$ and $\PMC_{\sigma_2}$ be the special arc diagrams corresponding to $\tau_1$ and $\tau_2$  as in Construction~\ref{const:tt-to-arc-diag-sp}, respectively. Then $\PMC_{\sigma_2}$ is obtained from $\PMC_{\sigma_1}$ by the pair of arcslides corresponding to the aforementioned arcslides from $\PMC_1$ to $\PMC_2$.
 \end{construction}
 
Let $\psi$ be a pseudo-Anosov diffeomorphism of a surface $\Sigma$ and 
$(\tau,\mu)$ be a measured train track suited to the unstable lamination of $\psi$, from the periodic splitting sequence (Theorem~\ref{thm:Agol-cycle}). That is $(\psi(\tau),\lambda^{-1}\psi(\mu))$ is obtained from $(\tau,\mu)$ by a sequence of maximal splits. Let $\PMC$ be the arc diagram for $\Sigma$ associated to $\tau$ by Construction~\ref{const:tt-to-arc-diag}. The abstract arc diagram associated to $\psi(\tau)$ by Construction~\ref{const:tt-to-arc-diag} is the same as $\PMC$. By Construction~\ref{const:diff-arc-slide}, the sequence of maximal splits from $\tau$ to $\psi(\tau)$ gives a sequence of arcslide pairs from $\PMC$ to $\PMC$ and so induces a strongly based diffeomorphism $\psi'$ of $F(\PMC)$. If we let $\phi_R$ denote the original embedding of $F(\PMC)$, as a neighborhood of $\tau$, and $\phi_L$ the embedding after the arcslides, then the diffeomorphism $\psi'$ is characterized by $\psi'\circ \phi_R=\phi_L$. But by construction, the diffeomorphism $\psi$ also has this property (after fixing $\psi$ suitably near the boundary), so $\psi'$ is isotopic to $\psi$ as diffeomorphisms of $\Sigma$. That is, we have proved:

\begin{lemma}\label{lem:splits-to-slides-1}
  The composition of the arcslide diffeomorphisms associated to the
  periodic splitting sequence for $\psi$ is a strongly based
  diffeomorphism which is isotopic, as a (not strongly based)
  diffeomorphism of $\Sigma$ to $\psi$.
\end{lemma}

Similarly, one can fix a subset $\sigma$ of the switches and construct a special arc diagram $\PMC_{\sigma}$ for $\Sigma$ as in Construction \ref{const:tt-to-arc-diag-sp}. As we discuss in Lemma~\ref{lem:arc-slide-diff} below, the composition of the arcslides associated to the periodic splitting sequence may not give a strongly based diffeomorphism of $F(\PMC_\sigma)$. The following lemma will allow us to adjust the diffeomorphism near the boundary to make it strongly based, by composing with a simple sequence of arcslides.

\begin{figure}
  \centering
  \includegraphics{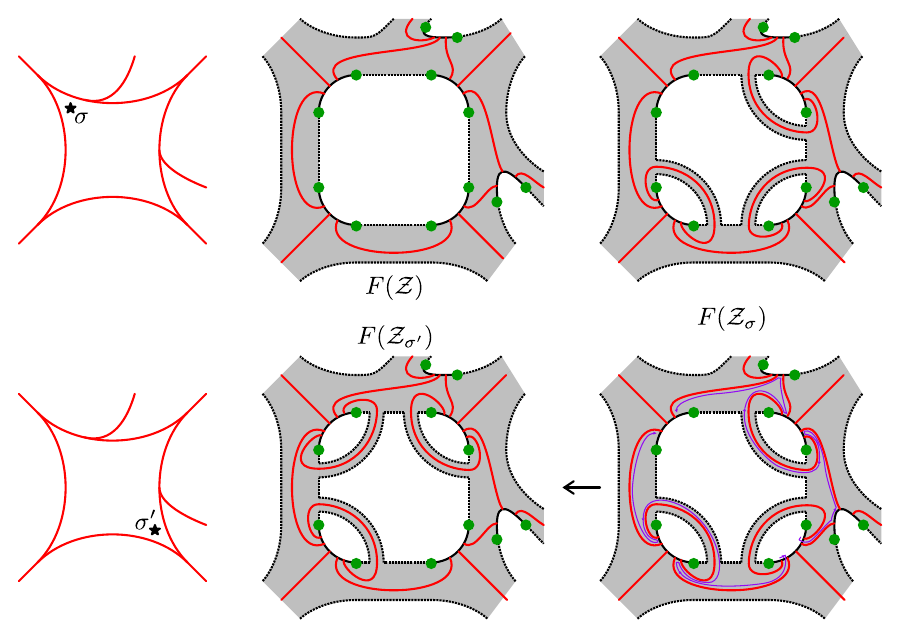}
  \caption{\textbf{Diffeomorphism connecting different choices of preferred punctures.} The figure shows two different choices of distinguished corner for a quadrilateral in a train track and the associated special arc diagrams. The thin arrows in the figure on the lower-right indicate a sequence of arc slides taking the arc diagram $F(\PMC_\sigma)$ to $F(\PMC_{\sigma'})$ (up to isotopy).}
  \label{fig:change-Sigma}
\end{figure}

\begin{lemma}\label{lem:bdy-tw-diff}
Let $\sigma_1$ and $\sigma_2$ be two subsets of the cusps of $\tau$ and $\PMC_{\sigma_1}$ and $\PMC_{\sigma_2}$ be the corresponding special arc diagrams as in Construction \ref{const:tt-to-arc-diag-sp}. There is a sequence of arcslides from $\PMC_{\sigma_1}$ to $\PMC_{\sigma_2}$ such that the points corresponding to the sub arc diagram $\PMC$ associated to $\tau$ stay fixed. 
The induced diffeomorphism from $F(\PMC_{\sigma_2})$ to $F(\PMC_{\sigma_1})$ extends to a diffeomorphism of $\Sigma$ which is isotopic to the identity. 
\end{lemma}

\begin{proof}
Each positive arc (component of $S_+$) in $\bdy F(\PMC)$ corresponds to a switch in $\tau$. So, a choice of subset $\sigma$ specifies a subset of positive arcs that includes exactly one positive arc on each component of $\bdy F(\PMC)$. The sutured surface $F(\PMC_\sigma)$ is obtained from $F(\PMC)$ by attaching a $1$-handle to $S_-$ for each component of $S_+$ that is not associated to a switch in $\sigma$. (See Figure~\ref{fig:change-Sigma}.) Consider a component of $\bdy F(\PMC)$ on which $\sigma_1$ and $\sigma_2$ specify distinct positive arcs. In the construction of $F(\PMC_{\sigma_1})$, a $1$-handle has been attached to $\bdy F(\PMC)$ at the switch distinguished by $\sigma_2$ (since by assumption $\sigma_1$ does not choose this arc). Sliding the feet of this $1$-handle as in Figure~\ref{fig:change-Sigma} gives a $1$-handle attached to $\bdy F(\PMC)$ corresponding to the arc distinguished by $\sigma_1$. This specifies a sequence of arcslides from $F(\PMC_{\sigma_1})$ to $F(\PMC_{\sigma_1'})$ where $\sigma_1'$ coincides with $\sigma_2$ on the boundary component of $F(\PMC)$ under consideration and coincides with $\sigma_1$ on other components. Repeating this process gives a sequence of arcslides from $F(\PMC_{\sigma_1})$ to $F(\PMC_{\sigma_2})$. Since the induced diffeomorphism from $F(\PMC_{\sigma_2})$ to $F(\PMC_{\sigma_1})$ is equal to identity on the subsurface $F(\PMC)\setminus \nbd(\bdy F(\PMC))$, the extension of this diffeomorphism on $\Sigma$ is isotopic to identity.
\end{proof}
 
As discussed in the proof of Lemma~\ref{lem:carry-diag-ext}, the diffeomorphism $\psi$ induces a correspondence between the cusps on the boundaries of the complementary regions of $\tau$. So, any subset $\sigma$ will get mapped to another subset, denoted by $\psi(\sigma)$.

\begin{lemma}\label{lem:arc-slide-diff}
Let $\sigma$ be a subset of switches as in Construction \ref{const:tt-to-arc-diag-sp}. The sequence of maximal splits from $\tau$ to $\psi(\tau)$ induces a sequence of arcslides from $\PMC_{\sigma}$ to $\PMC_{\psi^{-1}(\sigma)}$. The strongly based diffeomorphism represented by the composition of these arcslides extends to a diffeomorphism of $\Sigma$ representing the mapping class of $\psi$. 
\end{lemma}

\begin{proof}
 Let $\PMC$ be the arc diagram associated to $\tau$. As in the proof of Lemma~\ref{lem:bdy-tw-diff}, each positive arc in $\bdy F(\PMC)$ corresponds to a switch and so a cusp on the boundary of a complementary region of $\tau$. The sequence of arcslide pairs associated to the maximal split induces a strongly based diffeomorphism on $F(\PMC)$ such that its restriction to $S_+$, under the identification with cusps, is equal to the induced map by $\psi$. Therefore, the result of the arcslide moves on $\PMC_\sigma$ is $\PMC_{\psi^{-1}(\sigma)}$ and the induced diffeomorphism from $F(\PMC_{\psi^{-1}(\sigma)})$ to $F(\PMC_\sigma)$ extends to a diffeomorphism isotopic to $\psi$ on $\Sigma$.
\end{proof}

\begin{corollary}\label{cor:factor-into-arcslides}
For any subset $\sigma$ of switches as in Construction \ref{const:tt-to-arc-diag-sp}, the composition of 
the arcslide diffeomorphism from $F(\PMC_{\sigma})$ to $F(\PMC_{\psi^{-1}(\sigma)})$ as in Lemma~\ref{lem:bdy-tw-diff} and
the arcslide diffeomorphism 
associated to maximal splits from $F(\PMC_{\psi^{-1}(\sigma)})$ to $F(\PMC_{\sigma})$ as in Lemma~\ref{lem:arc-slide-diff} is a strongly based diffeomorphism from $F(\PMC_\sigma)$ to $F(\PMC_\sigma)$ that extends to a diffeomorphism isotopic to $\psi$ on $\Sigma$.
\end{corollary}

\section{Compressing surface diffeomorphisms}\label{sec:compress-diffeo}
The goal of this section is to prove Theorem~\ref{thm:phi-extend-bound} (restated as Theorem~\ref{thm:phi-extend-bound-v2}, below), that bordered-sutured Floer homology detects when a mapping class $\phi$ extends over some compression body and, in particular, that there is an explicit bound on the number of generators of the bordered-sutured module involved. The general strategy of the proof is similar to Casson-Long's proof that there is an algorithm to determine whether a diffeomorphism extends over some compression body (and, in fact, we use some of their results). The main difference is that their bounds are in terms of lengths of geodesics with respect to some hyperbolic metric, while we need bounds adapted to Heegaard diagrams. So, Sections~\ref{sec:length-via-triang} and~\ref{sec:bounded-complexity} give analogues of some of their results in terms of the number of intersections of curves with the triangulation dual to a train track. Section~\ref{sec:length-via-triang} focuses on the lengths of curves, while Section~\ref{sec:bounded-complexity} constructs compression bodies represented by bases of bounded length. Section~\ref{sec:tt-to-HD} then uses the bounded-length compression bodies to build a particular Heegaard diagram in which we can understand the number of intersection points, proving Theorem~\ref{thm:phi-extend-bound}. Unlike in the previous section, we require $\psi$ to be a diffeomorphism of a closed surface here; this is needed for the proof of Lemma~\ref{lem:scc-length-bound}.

\begin{notation}\label{not:Ts-etc}
  Given a pseudo-Anosov diffeomorphism $\psi$ of a closed surface $\Sigma$, let:
  \begin{itemize}
  \item $\tau$ be a train track, suited to the unstable foliation of $\psi$,
    from the periodic splitting sequence (so repeatedly applying maximal splits to
    $\tau$ eventually gives $\psi(\tau)$),
  \item $\mathcal{T}$ be the triangulation of $\Sigma$ dual to $\tau$,
  \item $v(\mathcal{T})$ denote the set of vertices of $\mathcal{T}$,
  \item $s$ denote the number of switches of $\tau$, 
  \item $l$ denote the number of edges of $\tau$, and
  \item $\kappa=|v(\mathcal{T})|$ denote the number of connected components of
    $\Sigma\setminus\tau$.
  \end{itemize}
\end{notation}

\subsection{Discrete length via triangulation}\label{sec:length-via-triang}
In this section, we use the triangulation $\mathcal{T}$ to define a notion of length for curves in $\Sigma$, and then study how the length grows when one applies $\psi$, as well as the number of intersections between a curve and its image under $\psi$.

For any simple closed curve $\gamma\subset\Sigma$, let
$\ell(\gamma)=\imath(\gamma,\mathcal{T})$ be the geometric intersection
number of $\gamma$ with the edges of $\mathcal{T}$, that is, the minimum
number of times that a simple closed curve
$\gamma'\subset\Sigma\setminus v(\mathcal{T})$ isotopic to $\gamma$ in $\Sigma$ intersects $\mathcal{T}$.

Recall from Section~\ref{sec:compression-bodies} that a basis 
for a compression body $C$ is a collection of curves in $\Sigma$ 
specifying $C$. Let $B(C)$ denote the set of all bases for $C$, and define 
\begin{equation}\label{eq:ls-Y}
  \ell(C)=\min_{\left\{\gamma_1,\dots,\gamma_n\right\}\in B(C)}\sum_{i=1}^n \ell(\gamma_i).
\end{equation}
  
We pause to show that for $\gamma=\cup_{i=1}^n\gamma_i$, the minimal length of an embedded representative of $\gamma$ is equal to $\ell(\gamma)=\sum_{i=1}^n\ell(\gamma_i)$. Certainly $\sum\ell(\gamma_i)$ is a lower bound for the length of an embedded representative for $\gamma$. Conversely, suppose $\gamma'_i$ is a minimal-length representative of $\gamma_i$ and $\gamma'=\cup_{i=1}^n\gamma'_i$. The curve $\gamma'$ may be immersed rather than embedded, but is regularly homotopic to $\gamma$. Hence, $\gamma'$ can be made embedded by a sequence of isotopies supported in bigons on $\Sigma$ with boundary on $\gamma'$ and interior disjoint from $\gamma'$. Since $\#(\gamma'\cap\mathcal{T})=\ell(\gamma)$, both sides of each bigon intersect $\mathcal{T}$ in the same number of points, so such isotopies do not change $\#(\gamma'\cap\mathcal{T})$.
  
Since $\tau$ is from the periodic splitting sequence, $\psi(\tau)$ is obtained from
$\tau$ by a sequence of maximal splits. So, $\psi(\tau)$ is
carried by $\tau$ with carrying map induced by the splitting sequence. Denote the associated incidence matrix by $M=(m_{ij})$.  

Denote the $l$ branches of $\tau$ by $b_1,
b_2,\dots, b_l$. Let
\begin{equation}
  \label{eq:s-of-psi}
  r(\psi)=\max_j\left\{\sum_{i=1}^lm_{ij}\right\}.
\end{equation}
That is, $r(\psi)$ is the largest entry of $(1,\dots,1)M$.

\begin{lemma}\label{lem:lbound-1}
For any simple closed curve $\gamma$ on $\Sigma$ we have
\[\ell(\psi(\gamma))\le r(\psi) \ell(\gamma).\]
\end{lemma}
\begin{proof} 
Consider a representative of $\gamma$ in a small tubular neighborhood $\nbd(\tau)$ of $\tau$ such that $\gamma$ intersects $\mathcal{T}$ in $\imath(\gamma,\mathcal{T})$ points. Any representative of $\gamma$ with minimal length can be isotoped into $\nbd(\tau)$ (not necessarily carried by $\tau$), so such a representative exists. Let
\[
  V_\gamma=\left(\#(T_1\cap\gamma),\dots,\#(T_l\cap\gamma)\right)^T,
\]
where $T_i$ denotes the edge of $\mathcal{T}$ dual to the branch $b_i$ of $\tau$.
Then $\psi(\gamma)$ has a representative in $\nbd(\tau)$ so that its
intersections with the edges of $\mathcal{T}$ are given by the entries
of the vector $MV_{\gamma}$. This representative does not necessarily
have minimal geometric intersection number with $\mathcal{T}$, but
this still gives
\[
  \ell(\psi(\gamma))\le (1,1,\dots,1)MV_{\gamma}\le r(\psi)\ell(\gamma),
\]
as desired.
\end{proof}

\begin{lemma}\label{lem:intbound}
For any simple closed curve $\gamma$ on $\Sigma$, we have
\[\imath(\gamma, \psi(\gamma))\le r(\psi)\left(\ell(\gamma)\right)^2\]
\end{lemma}

\begin{proof}
As in the proof of Lemma \ref{lem:lbound-1}, consider a representative
of $\gamma$ in a small neighborhood of $\tau$, built from parallel copies of the edges of $\tau$, joined up near the switches. So,
$\gamma$ intersects $\mathcal{T}$ in $\imath(\mathcal{T},\gamma)$ points.
Let $V_{\gamma}=(v_1,v_2,\dots, v_l)^T$ be the corresponding vector,
so $v_i=\#(\gamma\cap T_i)$. Then $\psi(\gamma)$ has a
representative also built from parallel copies of the edges of $\tau$,
such that $\#(\psi(\gamma)\cap T_i)=\sum_{j=1}^lm_{ij}v_j$.
By a small isotopy (removing some bigons), we can arrange that
$\#(\psi(\gamma)\cap b_i)\leq\#(\psi(\gamma)\cap T_i)$ (and these intersections are transverse).
So, for a representative of $\gamma$ in a sufficiently small neighborhood of $\tau$ so that $\gamma$ goes over the branch $b_i$ $v_i$-many times, each intersection point of $\psi(\gamma)$ with $b_i$ will result in $v_i$ intersection points in $\gamma\cap\psi(\gamma)$, and thus we have 
\[
  \#\left(\gamma\cap\psi(\gamma)\right)\le V_{\gamma}^TMV_{\gamma}\le |V_{\gamma}||MV_{\gamma}|\le   \ell(\gamma)\left((1,1,\dots,1) MV_{\gamma}\right)\le r(\psi)\left(\ell(\gamma)\right)^2,
\]
as desired.
\end{proof}

\subsection{Existence of a compression with bounded complexity}\label{sec:bounded-complexity}

In this section, we prove that if $\psi$ extends over some compression body, then $\psi$ extends over a compression body $C$ for which $\ell(C)$ is bounded above by a specific constant defined in terms of $\psi$ and $\tau$ (Lemma~\ref{lem:bound}).

Let $K>0$ be the smallest integer such that for any $k>K$, every non-infinitesimal branch or non-infinitesimal diagonal of $\tau$ goes over every branch of $\tau$ under the carrying map $\psi^k(\tau)<\tau$ and its extension to diagonals as in Lemma \ref{lem:carry-diag-ext}. Denote the number of infinitesimal branches and infinitesimal diagonals of $\tau$ by $l_{I}(\tau)$ and $d_{I}(\tau)$, respectively.

For any maximal diagonal extension $\wt{\tau}$ of $\tau$, fix a maximal diagonal
extension $\wt{\tau}'$ of $\tau$ that carries $\psi^K(\wt{\tau})$, so that the carrying map for $\psi^K(\wt{\tau})<\wt{\tau}'$ extends the carrying map for $\psi^K(\tau)<\tau$. Let $N_{\widetilde{\tau}}=(\widetilde{n}_{ij})$ be the corresponding incidence matrix. Define
\[c(N_{\widetilde{\tau}})=2\max_i\left\{\sum_{j=1}^{\widetilde{l}}\widetilde{n}_{ij}\right\},\quad\quad\quad c(\psi)=\max_{\widetilde{\tau}}\left\{c(N_{\widetilde{\tau}})\right\}+l_I(\tau)+d_{I}(\tau).\]
Here, $\widetilde{l}$ denotes the number of branches in
$\widetilde{\tau}$, and in the definition of $c(\psi)$ we take the
maximum over all maximal diagonal extensions $\wt{\tau}$ of $\tau$ (where for each $\wt{\tau}$ we choose some $\wt{\tau}'$).

\begin{lemma}\label{lem:bound-1} Let $\gamma\subset \Sigma$ be a simple closed curve. For any sufficiently large $n$, there exists a maximal diagonal extension $\widetilde{\tau}$ of $\tau$ such that
  \begin{enumerate}
  \item $\psi^n(\gamma)$ is carried by $\widetilde{\tau}$, and
  \item for any subarc $\alpha$ of $\psi^n(\gamma)$, if the image of
    $\alpha$ under the carrying map goes over at least $c(\psi)$
    branches of $\widetilde{\tau}$, then $\alpha$ goes over every
    branch of $\tau\subset\widetilde{\tau}$.
  \end{enumerate}
\end{lemma}
\begin{proof}
The sequence
$\{\psi^n(\gamma)\}$ converges to the unstable foliation of $\psi$ in
$\mathcal{PML}(\Sigma)$. By Lemma \ref{lem:unstableconv}, there exists an $N>0$ such that for any $n\ge N$ the curve $\psi^{n}(\gamma)$ is carried by some maximal diagonal extension $\widetilde{\tau}_n$ of
$\tau$. (The $\widetilde{\tau}_n$ might depend on $n$.) Let $\widetilde{\tau}_n'$ be the fixed maximal diagonal extension of $\tau$ that carries $\psi^K(\widetilde{\tau}_n)$ with a carrying map that extends the carrying $\psi^K(\tau)<\tau$ as above. The carrying map for $\psi^K(\widetilde{\tau}_n)<\widetilde{\tau}_n'$ gives a carrying of $\psi^{n+K}(\gamma)$ by $\widetilde{\tau}_n'$. Suppose $\alpha$ is a subarc of $\psi^{n+K}(\gamma)$ so that the image of
$\alpha$ under the carrying map goes over $c(\psi)$ branches of $\widetilde{\tau}_n'$. Lemma~\ref{lem:long-paths} and the definition of $c(\psi)$ imply that the image of $\alpha$ under the carrying map contains the image of $\psi^K(b)$ under the carrying map for some non-infinitesimal branch or diagonal $b\subset\widetilde{\tau}_n$. On the other hand, $K$ has been chosen such that the image of every branch of $\widetilde{\tau}_n$ which is not infinitesimal goes over every branch of $\tau$ under the carrying map $\psi^K(\widetilde{\tau}_n)<\widetilde{\tau}_n'$.
Therefore, the claim holds for any $n\ge N+K$ with $\widetilde{\tau}=\widetilde{\tau}_{n-K}'$ along with the described carrying map.
\end{proof}

For any diagonal $d$ of $\tau$, let $\ell(d)$ be the minimum,
over arcs $d'$ isotopic to $d$ rel endpoints, of the number of times
that $d'$ intersects the dual triangulation $\mathcal{T}$. 
Let
\[
c'(\psi)=\max\{\ell(d)\mid d\text{ is a diagonal of }\tau\}.
\]

\begin{lemma}\label{lem:scc-length-bound}
Suppose $\psi$ extends over some compression body $C$ with $\bdy C=\Sigma$. Then there exists an essential simple closed curve $\gamma\subset \Sigma$ such that $\ell(\gamma)\le c(\psi)+c'(\psi)$ and $\gamma$ bounds a disk in $C$.
\end{lemma}

\begin{proof}
The proof is similar to the proof of  \cite[Theorem 1.2]{CassonLong85:compression}. Let $\gamma\subset\Sigma$ be an essential simple closed curve that bounds a disk in $C$. As $n$ goes to infinity, $\psi^{-n}(\gamma)$ converges to the stable foliation of $\psi$ in $\mathcal{PML}(\Sigma)$. So, by Lemma~\ref{lem:stableconv}, for $n\gg0$, $\psi^{-n}(\gamma)$ is carried by some maximal diagonal extension $\widetilde{\tau}^*$ of the dual bigon track $\tau^*$. Thus, choose $n$ large enough so that:
\begin{enumerate}
\item By Lemma~\ref{lem:bound-1}, $\psi^n(\gamma)$ is isotopic to a simple closed curve $\gamma_n$ so that $\gamma_n$ is contained in a small fibered neighborhood $\nbd(\widetilde{\tau})$ of some maximal diagonal extension $\widetilde{\tau}$ of $\tau$, is carried by $\widetilde{\tau}$, and any subarc of $\gamma_n$ that goes over $c(\psi)$ branches of $\widetilde{\tau}$ goes over every branch of $\tau$.
\item $\psi^{-n}(\gamma)$ is isotopic to an essential simple closed curve $\gamma_{-n}$ in a small fibered neighborhood $\nbd(\widetilde{\tau}^*)$ of some maximal diagonal extension $\widetilde{\tau}^*$ of the dual bigon track $\tau^*$ and $\psi^{-n}(\gamma)$ is carried by $\wt{\tau}^*$.
\end{enumerate}
Choose the neighborhoods $\nbd(\widetilde{\tau})$ and $\nbd(\widetilde{\tau}^*)$ small enough that the intersection points of $\gamma_n$ and $\gamma_{-n}$ correspond to the intersection points of $\widetilde{\tau}$ and $\widetilde{\tau}^*$.
Since $\psi$ extends over $C$, both $\gamma_n$ and $\gamma_{-n}$ bound properly embedded disks in $C$, which we denote $D$ and $D'$, respectively. We assume that $D$ and $D'$ intersect transversely, and we remove circle components of $D\cap D'$ by isotopy. So, $D\cap D'$ is a disjoint union of properly embedded arcs. Define a pairing $\sim$ on the intersection points  $\gamma_n\cap \gamma_{-n}$ by setting $x\sim y$ if there exists an arc $\eta$ in $D\cap D'$ so that $\bdy \eta=\{x,y\}$. 

Choose an arc $\alpha\subset \gamma_n$ satisfying the following:
\begin{enumerate}[label=(\arabic*)]
\item\label{item:scc-lb-1} for any $x,y\in\gamma_n\cap\gamma_{-n}$, if $x\sim y$ then $x\in\alpha$ if and only if $y\in \alpha$,
\item\label{item:scc-lb-2} $\alpha$ goes over at least $c(\psi)$ branches of $\widetilde{\tau}$, and
\item $\alpha$ is minimal with respect to conditions~\ref{item:scc-lb-1} and~\ref{item:scc-lb-2}, i.e., there is no subarc of $\alpha$ satisfying~\ref{item:scc-lb-1} and~\ref{item:scc-lb-2}.
\end{enumerate}
(See Figure~\ref{fig:length-bound}.)

\begin{figure}
  \centering
  \includegraphics{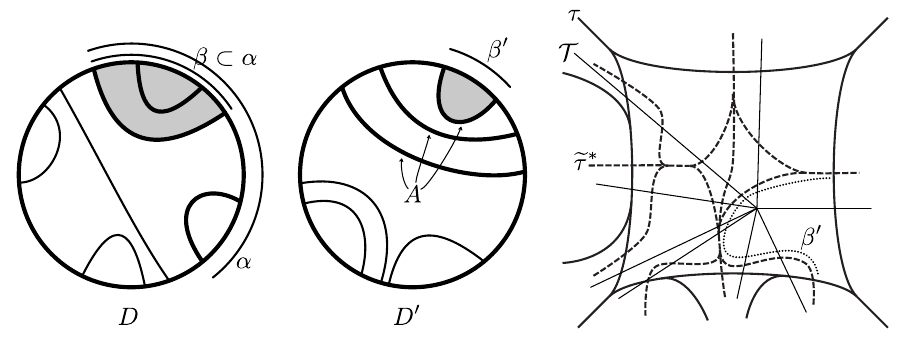}
  \caption[Bounding the length of a curve that bounds a disk]{\textbf{Bounding the length of a curve that bounds a disk.} Left and center: the disks $D$ and $D'$. The new disk $C\subset D\cup D'$ is shaded. The arcs in $A$ are thick. Right: part of the train tracks $\tau$ (solid) and $\wt{\tau}^*$ (dashed), the dual triangulation (thin) $\mathcal{T}$, and the curve $\beta'$ (dotted). The contribution of this region to $c'(\psi)$ is $4$ (via the diagonal from the top-left to the bottom-right). Since the curve $\beta'$ shown intersects $\mathcal{T}$ in this region five points, it is straightened to an arc intersecting $\mathcal{T}$ in one point, instead.}
  \label{fig:length-bound}
\end{figure}

 Consider the subset $A$ of $D\cap D'$ consisting of the arcs whose boundary is in $\alpha$. An innermost arc of $A\subset D'$ decomposes $\bdy D'=\gamma_{-n}$ into two arcs so that the interior of one of them, denoted $\beta'$, is disjoint from $\alpha$. Since the interior of $\beta'$ is disjoint from $\alpha$, which goes over every branch of $\tau$, $\beta'$ is in one region of $\Sigma\setminus\tau$. Moreover, $\beta'$ is carried by $\widetilde{\tau}^*$, so $\beta'$ can be modified by an isotopy that fixes its end points and is supported in exactly one region of $\Sigma\setminus \tau$ so that $\#(\beta'\cap\mathcal{T})\le c'(\psi)$.

 Let $\beta$ be the subarc of $\alpha$ with $\bdy \beta=\bdy
 \beta'$. This arc satisfies~\ref{item:scc-lb-1} and so minimality of $\alpha$
 implies that  the number of intersection points in
 $\beta\cap\mathcal{T}$ is less than or equal to 
 $c(\psi)$. (Equality holds when $\beta=\alpha$.) Let $\gamma=\beta\cup\beta'$. It is easy to see that
 $\gamma$ bounds a disk in $C$.
 
 The dual bigon track for the train track $\widetilde{\tau}$ is obtained by collapsing $\tau^*$ along admissible arcs, as in Figure~\ref{fig:trivial-collapse}. Splitting this dual bigon track at the large branches associated to these collapses results in a (combed) maximal diagonal extensions of $\tau^*$. From the proof of Lemma~\ref{lem:stableconv}, we may assume that $\widetilde{\tau}^*$ is one of the maximal diagonal extensions obtained this way. Since $\widetilde{\tau}$ intersects its dual efficiently, it intersects $\widetilde{\tau}^*$ efficiently as well. Thus, $\gamma$ is an essential curve in $\Sigma$. Moreover, $\ell(\beta\cup\beta')\le c(\psi)+c'(\psi)$, so we are done.
\end{proof}

\begin{definition} Given an embedded $1$-dimensional submanifold $\gamma$ in $\Sigma$ consisting of essential simple closed curves, and a compression body $C$ with outer boundary $\Sigma$, we say $\gamma$ \emph{compresses} in $C$ if every connected component of $\gamma$ bounds a disk in $C$.
\end{definition}

Given two 1-dimensional
submanifolds $\gamma$ and $\gamma'$ of $\Sigma$, Casson and
Long~\cite[Lemma 2.2]{CassonLong85:compression} construct a finite
(possibly empty) collection of compression bodies
$B(\gamma,\gamma')=\{C_1,\dots,C_n\}$ satisfying the following:
\begin{enumerate}[label=(\arabic*)]
 \item\label{item:B1} For every $i$, each curve in $\gamma$ or $\gamma'$ bounds a disk in $C_i$.
 \item\label{item:B2} If every closed curve in $\gamma$ or $\gamma'$ bounds a disk in some compression body $C$, then some $C_i\subseteq C$.
 \item\label{item:B3} The set $B(\gamma,\gamma')$ is minimal with respect to the properties~\ref{item:B1} and~\ref{item:B2}, i.e., no proper subset of $B(\gamma,\gamma')$ satisfies both properties. 
\end{enumerate}

We recall the construction of $B(\gamma,\gamma')$. Suppose $\gamma=\coprod_{i=1}^k\gamma_i$ and $\gamma'=\coprod_{j=1}^{l}\gamma'_j$ where each $\gamma_i$ and $\gamma'_j$ is an essential simple closed curve. Suppose there exists some compression body $C$ such that every $\gamma_i$ and $\gamma_j'$ bounds a disk in $C$ otherwise, $B(\gamma,\gamma')=\emptyset$.  Let $D=\coprod_{i=1}^kD_i$ and $D'=\coprod_{j=1}^{l}D'_j$ be disjoint unions of properly embedded disks in $C$ so that $\bdy D_i=\gamma_i$, $\bdy D'_j=\gamma'_j$ and $D_i$ intersects $D'_j$ in a set of pairwise disjoint arcs for any $i$ and $j$. Then $D\cap D'$ induces a pairing on $\gamma\cap \gamma'$ by setting $x\sim y$ if and only if $\{x,y\}$ is the boundary of an arc in $D\cap D'$. Clearly, this pairing has the following properties: 
\begin{itemize}
\item if $x\sim y$ then $x$ and $y$ lie on the same connected components of $\gamma$ and $\gamma'$
\item for any subarc $\alpha$ in $\gamma$ or $\gamma'$ where $\bdy\alpha$ consists of paired points, and any pair $x\sim y$, either $\{x,y\}\cap \mathrm{int}(\alpha)=\emptyset$ or $\{x,y\}\cap \mathrm{int}(\alpha)=\{x,y\}$.
\end{itemize}
On the other hand, given any pairing on the intersection points $\gamma\cap\gamma'$ satisfying
the above properties, one can construct a compression body that both $\gamma$ and $\gamma'$ compress in, as follows. Consider paired points $x\sim y$ on $\gamma$ such that there exists an arc
$\alpha\subset \gamma$ with $\bdy\alpha=\{x,y\}$ and
$\mathrm{int}(\alpha)\cap\gamma'=\emptyset$. Let $\gamma'_j$ be the
component of $\gamma'$ containing $x$ and $y$.  
Modify $\gamma'$ and remove the pair of intersection points $\{x,y\}\subset \gamma\cap\gamma'$ 
by doing surgery on $\gamma'_j$ at $x$ and $y$ along $\alpha$. One can remove all of the intersections points between $\gamma$ and $\gamma'$ by repeating this process. The resulting set of pairwise disjoint simple closed curves defines a compression body that every component of $\gamma$ or $\gamma'$ bounds a disk in. Moreover, if both $\gamma$ and $\gamma'$ compress in some compression body $C$, then $C$ contains the compression body associated to the pairing induced by $C$ on $\gamma\cap \gamma'$. Then $B(\gamma,\gamma')$ will be a subset of the finite set of compression bodies corresponding to these pairings. To make the set $B(\gamma,\gamma')$ minimal, if the compression body associated to one pairing is contained in the compression body associated to another then one drops the larger of the two. (Checking if one compression body is contained in another can be done using the fundamental group, say, or using twisted Floer homology, to check if the curves in a basis for the smaller compression body bound disks in the bigger compression body.)

\begin{lemma}\label{lem:basisbound}
Assume $\gamma$ and $\gamma'$ are embedded $1$-manifolds consisting of essential circles on
$\Sigma$ and $B(\gamma,\gamma')\neq\emptyset$. Then for any
$C_i\in B(\gamma,\gamma')$ we have 
\[\ell(C_i)\le \ell(\gamma)+\ell(\gamma')+\imath(\gamma,\gamma')\min\{\ell(\gamma),\ell(\gamma')\}.\]
\end{lemma}

\begin{proof}
Suppose $\gamma$ and $\gamma'$ intersect $\mathcal{T}$ minimally, i.e., in $\ell(\gamma)$ and $\ell(\gamma')$ points, respectively. If $\#(\gamma\cap\gamma')>\imath(\gamma,\gamma')$ there exists an embedded bigon on $\Sigma$ whose interior is disjoint from $\gamma\cup\gamma'$ and whose boundary consists of an arc on $\gamma$ and an arc on $\gamma'$. Since $\gamma$ and $\gamma'$ intersect $\mathcal{T}$ minimally, both arcs on the boundary of this bigon intersect $\mathcal{T}$ in the same number of points, and so modifying $\gamma'$ with an isotopy supported in a small neighborhood of this bigon to remove the corresponding pair of intersection points does not change $\#(\gamma'\cap\mathcal{T})$. Repeating this process, we get representatives of $\gamma$ and $\gamma'$ that intersect $\mathcal{T}$ minimally and satisfy $\#(\gamma\cap\gamma')=\imath(\gamma,\gamma')$.

Suppose $\ell(\gamma)\le \ell(\gamma')$. Then given a pairing on $\gamma\cap\gamma'$, each step of removing a pair of intersection points by modifying $\gamma'$ will add at most $2\ell(\gamma)$ to the total number of intersection points between the curves and the edges of the ideal triangulation $\mathcal{T}$.  If $\ell(\gamma)\ge \ell(\gamma')$ we can remove the intersection points by modifying $\gamma$ instead of $\gamma'$. So, we are done.
\end{proof}

\begin{lemma}\label{lem:bound}
If $\psi$ extends over some compression body, then it extends over a compression body $C$ with 
\[
  \ell(C)\le (\overbrace{F_{\psi}\circ\cdots\circ F_{\psi}}^{2g})\left(c(\psi)+c'(\psi)\right)
\]
where $F_{\psi}(x)=(1+r(\psi))x+r(\psi)x^3$.
\end{lemma}

\begin{proof}
Let $C'$ be a compression body that $\psi$ extends over. By Lemma
\ref{lem:scc-length-bound}, there exists an essential simple closed
curve $\gamma$ that compresses in $C'$ with $\ell(\gamma)\le
c(\psi)+c'(\psi)$. We construct a rooted tree, such that each vertex
corresponds to a set of pairwise disjoint essential curves on
$\Sigma$, $\gamma$ is the root, and the children of each vertex
$\gamma'$ are in one-to-one correspondence with compression bodies in
$B(\gamma',\psi(\gamma'))$, unless $B(\gamma',\psi(\gamma'))$ has one element. (Cf.~\cite[Proof of Theorem
2.1]{CassonLong85:compression}.) Specifically, if $B(\gamma',\psi(\gamma'))$ has more than one element, each child of
$\gamma'$ is a basis for an element of $B(\gamma',\psi(\gamma'))$
achieving the minimum in Formula~\eqref{eq:ls-Y} (i.e., a minimal
length basis). If $B(\gamma',\psi(\gamma'))$ has one element (i.e., $\psi$ extends over $\Sigma[\gamma']$) or $B(\gamma',\psi(\gamma'))=\emptyset$ ($\gamma'$ and $\psi(\gamma')$ cannot be compressed in the same compression body) then $\gamma'$ does not have a child.

If $\gamma''$ is a child of
$\gamma'$, by Lemmas \ref{lem:basisbound},~\ref{lem:lbound-1}, and~\ref{lem:intbound}, we have
\[\begin{split}
\ell(\gamma'')&\le \ell(\gamma')+\ell(\psi(\gamma'))+\imath(\gamma',\psi(\gamma'))\ell(\gamma')\\
&\le \ell(\gamma')+r(\psi)\ell(\gamma')+r(\psi)\ell(\gamma')^3=F_\psi(\ell(\gamma')).
\end{split}\]

Since $\psi$ extends over $C'$, for some vertex $\gamma'$ of this tree, $\psi$ extends over the compression body defined by $\gamma'$. Let $C=\Sigma[\gamma']$. By \cite[Lemma 2.3]{CassonLong85:compression} this tree has height at most $2g$. Thus, 
\[
  \ell(\gamma')\le (F_{\psi}\circ \cdots\circ F_{\psi})(\ell(\gamma))\le (F_{\psi}\circ \cdots\circ F_{\psi})(c(\psi)+c'(\psi)),
\]
where $F_{\psi}$ is composed with itself $2g$-times. 
\end{proof}

Let 
\begin{equation}\label{eq:def-M-psi}
  M(\psi):=(\overbrace{F_{\psi}\circ\cdots\circ F_{\psi}}^{2g})\left( c(\psi)+c'(\psi) \right).
\end{equation}

\subsection{Special bordered-sutured diagrams and a bound on the rank of \texorpdfstring{$\BSD$}{BSD}}\label{sec:tt-to-HD}
In this section, we show how to turn a train track for a diffeomorphism $\psi$ and a compression body which $\psi$ extends over into a bordered-sutured Heegaard diagram, while maintaining an explicit bound on the number of generators of the bordered module associated to the diagram. This immediately gives Theorem~\ref{thm:phi-extend-bound}. The process is outlined in Figure~\ref{fig:HD-from-tt}.

Before turning to the construction of the Heegaard diagram, we start with two lemmas about bases for compression bodies that $\psi$ extends over.
\begin{lemma}\label{lem:basisconditions}
Let $C$ be a compression body with outer boundary $\Sigma$. There exists an embedded, closed $1$-manifold $\gamma\subset\Sigma$ which is a basis for $C$ and satisfies the following conditions:
\begin{enumerate}[label=(\arabic*)]
\item\label{item:bc-1} $\#(\gamma\cap\mathcal{T})=\ell(\gamma)=\ell(C)$.
\item\label{item:bc-2} There exist at most two intersection point of $\gamma\cap\tau$ between any two consecutive intersection points of $\gamma\cap \mathcal{T}$.
\item\label{item:bc-3} Every connected component of $\Sigma\setminus (\mathcal{T}\cup\gamma)$ which is a triangle with one side on $\gamma$ and two sides on $\mathcal{T}$ or a rectangle with two sides on $\gamma$ and two sides on $\mathcal{T}$ contains no switch of $\tau$.
\end{enumerate}
\end{lemma}

\begin{proof} 
  Let $\gamma=\coprod_{i=1}^n\gamma_i$ be a basis for $C$ such that $\gamma\pitchfork\mathcal{T}$ and $\#(\gamma\cap\mathcal{T})=\ell(\gamma)=\ell(C)$, so $\gamma$ satisfies Condition~\ref{item:bc-1}. Every triangle $T$ of the triangulation $\mathcal{T}$ contains exactly one switch of $\tau$. Since $\gamma$ intersects $\mathcal{T}$ minimally, $T\setminus\gamma$ consists of one special region that has nonempty intersection will all three sides of $T$, and some triangles and rectangles that intersect two sides of $T$. The curve $\gamma$ can be isotoped in the interior of $T$ such that after the isotopy the switch of $\tau$ in $T$ lies in the special region. (See Figure~\ref{fig:gamma-in-triangle}.) After doing this for every triangle of $\mathcal{T}$ we get an isotopic translate of $\gamma$ that satisfies Condition~\ref{item:bc-3}. Turning to Condition~\ref{item:bc-2}, since $\gamma$ satisfies Condition~\ref{item:bc-3}, if there are more than two intersection points of $\gamma\cap\tau$ between two consecutive intersection points of $\gamma\cap\mathcal{T}$ there is a bigon in the triangle $T$ containing the intersection points with boundary on $\gamma$ and $\tau$. We change $\gamma$ with a small isotopy to remove the bigon and reduce the number of intersection points. Continue this process until no such bigon is left, and the resulting $\gamma$ will satisfy Condition~\ref{item:bc-2} as well. (Again, see Figure~\ref{fig:gamma-in-triangle}.)
\end{proof}

\begin{figure}
  \centering
  \includegraphics{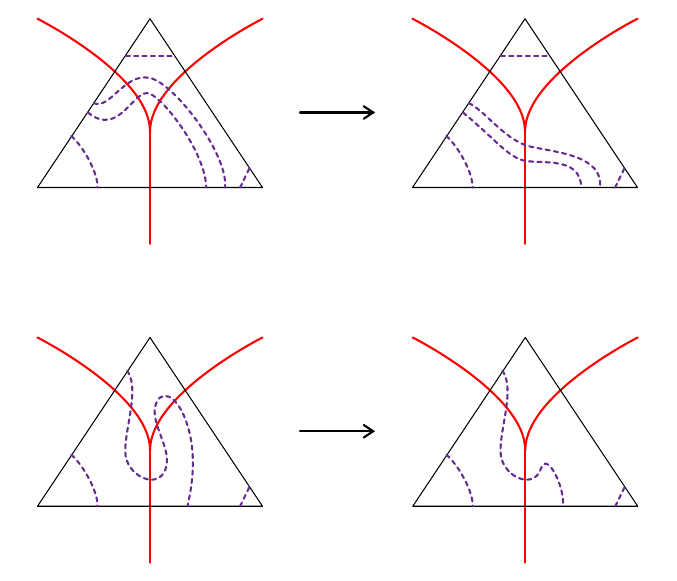}
  \caption[Local form for $\gamma$ in a triangle]{\textbf{Local form for $\gamma$ inside a triangle.} The triangle in $\mathcal{T}$ is thin, the train track $\tau$ is \textcolor{red}{solid}, and the 1-manifold $\gamma$ is \textcolor{darkpurple}{dashed}. Top: arranging that the switch lies in the unique region which is not a rectangle or triangle. Bottom: arranging that each arc of $\gamma$ intersects the train track in at most two points.}
  \label{fig:gamma-in-triangle}
\end{figure}

\begin{lemma}\label{lem:graph}
For any $\gamma$ satisfying the conditions of Lemma~\ref{lem:basisconditions}, we have $\#(\gamma\cap\tau)\le 2 \ell(\gamma)$. Moreover,  every connected component of $\Sigma\setminus \gamma$ contains at least one switch of $\tau$.
\end{lemma}
\begin{proof}
 The inequality follows from condition~\ref{item:bc-2}. To prove the second part, if $C$ is a handlebody, then $\Sigma\setminus\gamma$ is connected and the result holds trivially. Otherwise, we define a dual graph for each connected component $A$ of $\Sigma\setminus\gamma$ as follows. Put a vertex in the interior of each connected component of $\Sigma\setminus(\mathcal{T}\cup\gamma)$ that is in $A$. If two such components share an edge that lies in $\mathcal{T}$, connect the corresponding vertices by an arc dual to this edge. If $A$ does not contain any switch of $\tau$, every component of $\Sigma\setminus(\mathcal{T}\cup\gamma)$ that is included in $A$ is either a triangle or a rectangle, and in fact $A$ is either entirely composed of triangles or entirely of rectangles. Therefore, the dual graph is a disjoint union of circles on $\Sigma$. By removing the vertices of $\mathcal{T}$ from $A$ (if $A$ is composed of triangles), we get a subset of $\Sigma$ that is homotopy equivalent to the dual graph, and so $A$ has genus $0$. However, since $\gamma$ is a basis for some compression body $C$ such that $\ell(\gamma)=\ell(C)$, every connected component of $\Sigma\setminus \gamma$ has genus strictly greater than $0$, a contradiction.
 \end{proof}

Corresponding to the train track $\tau$ and any embedded $1$-manifold $\gamma\subset \Sigma$ satisfying the conditions of Lemma \ref{lem:basisconditions}, we define an $(\alpha,\beta)$-bordered--sutured Heegaard diagram 
\[
  \mathcal{H}=\left(\widetilde{\Sigma},\alphas^a,\betas^c\cup\betas^a,\Gamma\right).
\]
(An example is shown in Figures~\ref{fig:HD-from-tt} and~\ref{fig:HD-from-tt-2}.)

First, we modify the dual graph $G=\coprod_AG_A$ defined in the proof of Lemma \ref{lem:graph} to construct an embedded graph $\widetilde{G}\subset \Sigma$ whose vertices coincide with the switches of $\tau$ and so that every connected component of $\Sigma[\gamma]\setminus\widetilde{G}$ is a disk with at least three vertices (not necessarily distinct) on its boundary. (The latter condition will be needed to show the sutures can be distributed appropriately to give a bordered-sutured diagram.)
For every connected component of $\Sigma\setminus(\mathcal{T}\cup\gamma)$ containing a switch of $\tau$, isotope the corresponding vertex of $G$ to coincide with the switch. The rest of the vertices of $G$ have degree two. Remove them and concatenate their adjacent edges into one edge. Every vertex in the resulting graph, which we still denote by $G$, has degree three and coincides with a switch of $\tau$. Since the inclusion of $G$ into $\Sigma\setminus(\gamma\cup v(\mathcal{T}))$ is a homotopy equivalence, each connected component of  $\Sigma\setminus (\gamma\cup G)$ has one of the following types:
\begin{enumerate}[label=(\arabic*)]
\item\label{item:disk} a disk containing one vertex of $\mathcal{T}$ or
\item\label{item:annulus} an annulus with one boundary on $\gamma$. (The part of the boundary of the annulus along $G$ may not be embedded.)
\end{enumerate}
So, $G$ divides $\Sigma[\gamma]$ into disk regions.  Every type~\ref{item:disk} region corresponds to a component of $\Sigma\setminus (\tau\cup \gamma)$ that contains a vertex of $\mathcal{T}$, as in Figure~\ref{fig:new-pic-1}. For any such component, the switches on the boundary correspond to the vertices of $G$, and there is at least one switch between any two consecutive boundary arcs that lie on $\gamma$. Since every connected component of $\Sigma\setminus \tau$ contains more than one switch on the boundary, minimality of the number of intersection points between $\gamma$ and $\mathcal{T}$ implies that type~\ref{item:disk} disk regions cannot be monogons; see Figure~\ref{fig:no-type-1-monogon}. If a type~\ref{item:disk} disk is a bigon, then again minimality of the number of intersection points between $\gamma$ and $\mathcal{T}$ implies that its corresponding component of $\Sigma\setminus (\tau\cup \gamma)$ contains two $\gamma$ arcs on the boundary; see Figure~\ref{fig:one-arc-on-gamma}. Isotoping one of these boundary arcs (and possible parallel copies) over the vertex of $\mathcal{T}$ in this component will remove this bigon, without creating a new type \ref{item:disk} bigon; see Figure~\ref{fig:type-1-bigon}.

\begin{figure}
  \centering
  \includegraphics{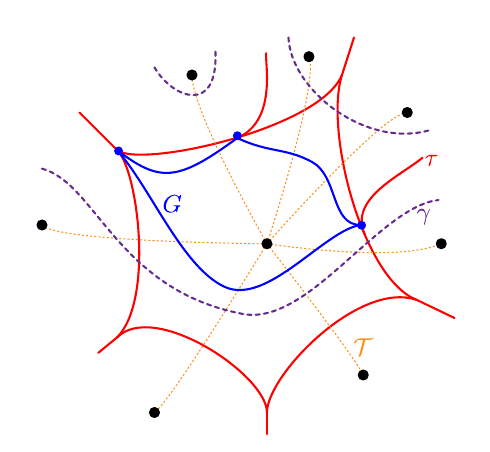}  
  \caption{\textbf{An example of a type~\ref{item:disk} region.} The train track $\tau$ is red, the dual triangulation $\mathcal{T}$ is thin, dashed, orange, the graph $G$ is blue, and the curves $\gamma$ are dark purple and dashed.}
  \label{fig:new-pic-1}
\end{figure}

\begin{figure}
  \centering
  \includegraphics[scale=.95]{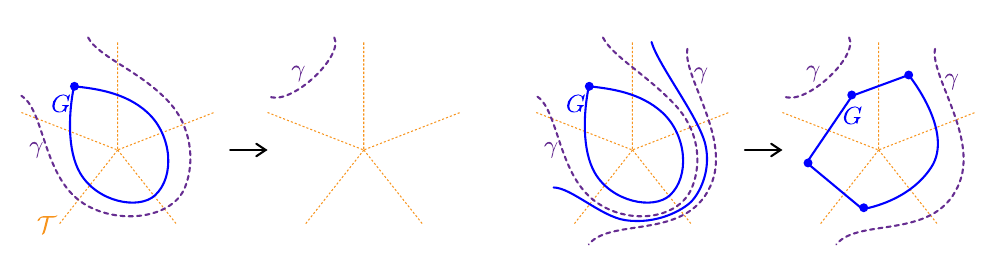}
  \caption{\textbf{Nonexistence of type 1 monogons.} In both figures, the existence of a type~\ref{item:disk} monogon implies that the number of intersections of $\gamma$ and $\mathcal{T}$ can be reduced, by the indicated move. The figure on the right shows how this affects the graph $G$ in a slightly larger region; in particular, it does not introduce a new monogon. Conventions are the same as in Figure~\ref{fig:new-pic-1}.}
  \label{fig:no-type-1-monogon}
\end{figure}

\begin{figure}
  \centering
  \includegraphics{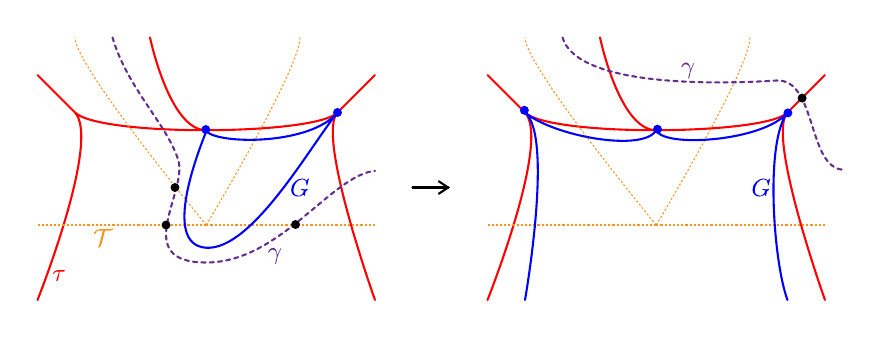}
  \caption{\textbf{Nonexistence of bigons with one arc on $\gamma$.} If there is a bigon with one arc on $\gamma$ then the number of intersections of $\gamma$ and $\mathcal{T}$ is not minimal. Conventions are as in Figure~\ref{fig:new-pic-1}; intersections of $\gamma$ with $\mathcal{T}$ are marked with dots (as are vertices of $G$).}
  \label{fig:one-arc-on-gamma}
\end{figure}

\begin{figure}
  \centering
  \includegraphics{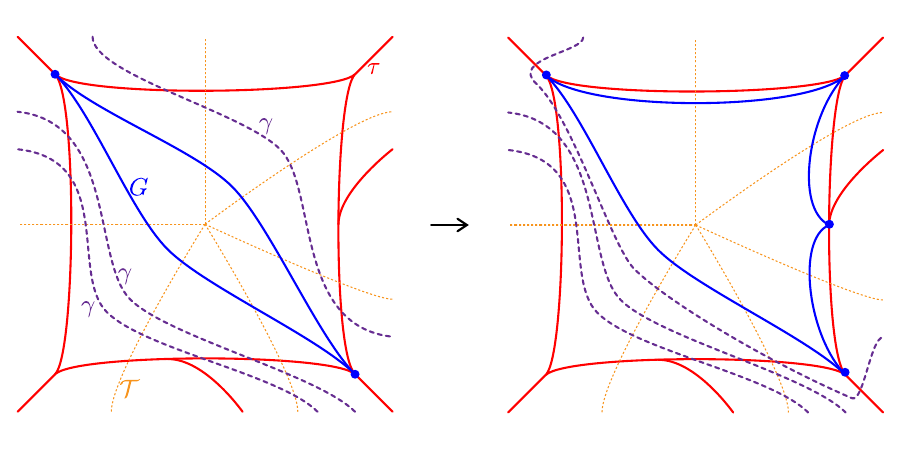}
  \caption{\textbf{Removing type-1 bigons.} Moving one of the arcs of $\gamma$ parallel to the bigon across the vertex eliminates the bigon. Conventions are as in Figure~\ref{fig:new-pic-1}.}
  \label{fig:type-1-bigon}
\end{figure}

Next, we will remove type~\ref{item:annulus} monogons and bigons by removing edges from $G$.   Let $A$ be a type~\ref{item:annulus} component. Write $\bdy A=L\coprod \gamma_A$ where $\gamma_A$ is a component of $\gamma$, and $L$ is a loop in $G$. If $A$ is a monogon, so $L$ contains exactly one vertex of $G$, there exists a distinct connected component of $\Sigma\setminus(\gamma\cup G)$ that is adjacent to $A$ and contains $L$ on its boundary; denote it by $A'$. Removing the single edge in $L$ from $G$, as in Figure~\ref{fig:type-2-monogon}, will result in concatenating $A$ and $A'$, giving a disk region in $\Sigma[\gamma]$ with at least three vertices (a vertex is repeated).  We repeat this until no monogons are left. 

\begin{figure}
  \centering
  \includegraphics{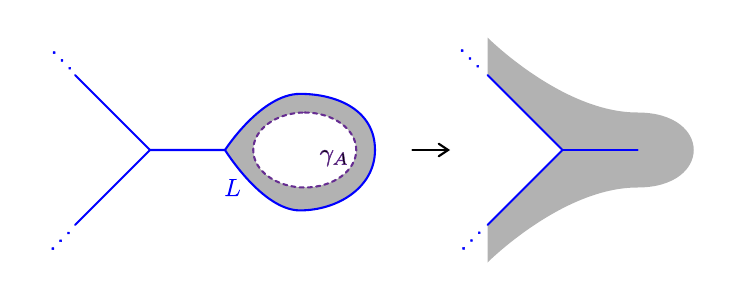}
  \caption{\textbf{Removing a type~\ref{item:annulus} monogon.} On the left, the annulus $A$ is shaded; $A'$ is the unshaded region outside $A$. The picture on the right is in $\Sigma[\gamma]$, not in $\Sigma$. Removing the round edge $L$ concatenates $A$ and $A'$; the result is still a disk in $\Sigma[\gamma]$.}
  \label{fig:type-2-monogon}
\end{figure}

Next, if $A$ is a bigon, its adjacent region(s) (regions that share an edge with $A$) are not bigons. Indeed, since vertices of $G$ either have degree three or one, if $A$ is adjacent to just one region then that region is not a bigon. Moreover, if $A$ is adjacent to a bigon then the other adjacent region must be a bigon as well.  Then $\ell(\gamma)$ is not equal to $\ell(C)$, a contradiction. Thus, no two bigons are adjacent. We remove exactly one side of each bigon to make sure that none of the regions are bigons. Since no two bigons are adjacent, this process will not create a region that is not a disk. The final graph is $\widetilde{G}$.

Second, given a graph $\widetilde{G}$ as above, we use Hall's marriage theorem to show that there exists an injective map $\sigma'$ from $\pi_0(\Sigma[\gamma]\setminus \widetilde{G})$ to $v(\widetilde{G})$, where $v(\widetilde{G})$ denotes the set of vertices of $\widetilde{G}$, taking a component to a vertex on its boundary. Suppose $P_1, P_2, \dots, P_m$ are $m$ distinct components of $\Sigma[\gamma]\setminus \widetilde{G}$. Denote the number of sides in $P_i$ by $k_i$. Further, assume $\cup_{i=1}^m\overline{P_i}$ contains $v$ vertices and $e$ edges of $\widetilde{G}$. Then, since none of the $P_i$ are monogons or bigons,
\[3m\le k_1+k_2+\dots+k_m\le 2e\le 3v\]
and so $v\ge m$, and therefore the claim follows from Hall's marriage theorem. 

\begin{figure}
  \centering
  \includegraphics{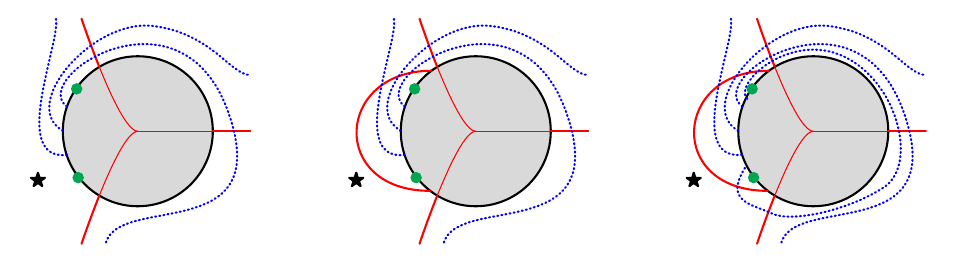}
  \caption[The form of the Heegaard diagram near the boundary]{\textbf{The form of the Heegaard diagram near the boundary.} The disk $D_i$ is shaded; it is not part of the Heegaard diagram. The $\alpha$-arcs are \textcolor{red}{solid} and the $\beta$-arcs are \textcolor{blue}{dotted}. The \textcolor{red}{thin} arcs indicate the intersection of the train track with $D_i$ and are not part of the $\alpha$-arcs. In the left figure, this switch is the switch chosen by $\sigma$ and $\sigma'$ for the region marked $\star$. In the center, it is the switch chosen by $\sigma'$ but not $\sigma$, so there is an extra $\alpha$-arc. On the right, this switch is not the chosen by $\sigma$ or $\sigma'$, so there is an extra $\alpha$- and an extra $\beta$-arc.}
  \label{fig:psi-HD-near-bdy}
\end{figure}

Now, we are ready to define $\HD$ as follows:
\begin{itemize}
  \item Consider pairwise disjoint small disk neighborhoods $D=\coprod_{i}D_i$ of the switches (and so the vertices of $\widetilde{G}$) in $\Sigma\setminus \gamma$ and let $\widetilde{\Sigma}=\Sigma\setminus D$.
  \item On each boundary component $\bdy D_i\subset \bdy \widetilde{\Sigma}$, consider two disjoint arc sutures, both in the boundary of the cusp region of $D_i\cap \tau$, and let $\Gamma$ be the union of all these sutures. Let $\bdy_L\widetilde{\Sigma}$ (respectively $\bdy_R\widetilde{\Sigma}$) be the union of the arcs in $\bdy\widetilde{\Sigma}\setminus \Gamma$ that are not disjoint (respectively are disjoint) from $\tau$. See Figure~\ref{fig:psi-HD-near-bdy}. Note that each component of $\bdy_L\widetilde{\Sigma}$ intersects $\tau$ at three points.
  \item Let $\betas^c=\gamma$.
  \item Let $\alphas^a=\alphas^{a,1}\coprod\alphas^{a,2}$ where $\alphas^{a,1}=\widetilde{\Sigma}\cap \tau$, and $\alphas^{a,2}$ extends $\alphas^{a,1}$ to a parameterization of the sutured surface $(\widetilde{\Sigma},\Gamma)$. Every component of $\wt{\Sigma}\setminus \alphas^{a,1}$ has at least one pair of sutures on its boundary, since every region in the complement of the train track has at least one cusp. For each connected component of $\widetilde{\Sigma}\setminus \alphas^{a,1}$ that contains more than one pair of sutures on its boundary, $\alphas^{a,2}$ consists of arcs parallel to $\bdy\widetilde{\Sigma}$ for all except one pair of these sutures, as in Figure~\ref{fig:psi-HD-near-bdy}. (This corresponds to making a choice $\sigma$ of one cusp for each region in the complement of the train track, as in Section~\ref{sec:ad-from-tt}.)

  \item Let $\betas^a=\betas^{a,1}\coprod\betas^{a,2}$, where  $\betas^{a,1}$ is obtained as follows. The map $\sigma'$ assigns a disk $D_i$, and so a boundary component of $\widetilde{\Sigma}$, to each connected component of $\widetilde{\Sigma}\setminus (\betas^c\cup \widetilde{G})$. The arcs $\betas^{a,1}$ are obtained by isotoping the arcs $\widetilde{G}\cap \widetilde{\Sigma}$ in a small neighborhood of $\bdy\widetilde{\Sigma}$ so that:
  \begin{enumerate}
    \item $\bdy (\widetilde{G}\cap\widetilde{\Sigma})\subset \bdy_R\widetilde{\Sigma}$.
    \item The sutures on the boundary components of $\widetilde{\Sigma}$ distinguished by $\sigma'$ are on the boundary of this connected component of $\widetilde{\Sigma}\setminus (\betas^c\cup\betas^{a,1})$, as well.
  \end{enumerate}
  Then every connected component of $\widetilde{\Sigma}\setminus (\betas^c\cup \betas^{a,1})$ has genus zero, and contains at least one pair of sutures on its boundary. 
  \item For any connected component of $\bdy \widetilde{\Sigma}$ whose corresponding vertex is not in the image of $\sigma'$, add a $\beta$ arc as in Figure~\ref{fig:psi-HD-near-bdy} and let $\betas^{a,2}$ be the union of these arcs.
\end{itemize}

\begin{example}
  Figure~\ref{fig:HD-from-tt} shows the construction of the Heegaard diagram associated to the mapping class $\tau_a\tau_b\tau_c\tau_d^{-1}$ of a genus $2$ surface with one puncture, where $a,b,c,d$ are the curves shown, and a compression body where a single curve (passing through the edge labeled $D$ once) is being compressed. (Even though we have required surfaces to be closed in this section, except for Lemma~\ref{lem:scc-length-bound}, the construction works equally well for punctured surfaces. The train track in this example was produced by Mark Bell's Flipper~\cite{flipper}, which gives train tracks on punctured surfaces.)
\end{example}

\begin{figure}
  \centering
  \begin{tabular}{cc}
  \includegraphics[scale=.94]{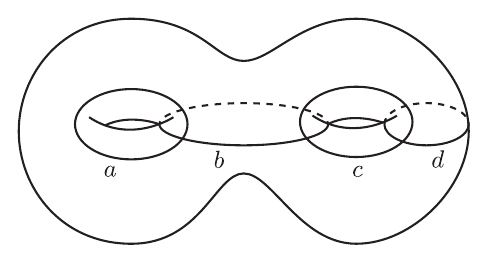} &
  \includegraphics[scale=.81]{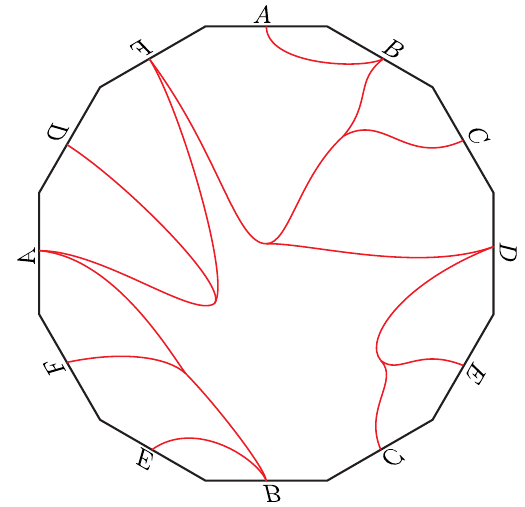}\\
  (a) & (b)\\[2em]
  \includegraphics[scale=.81]{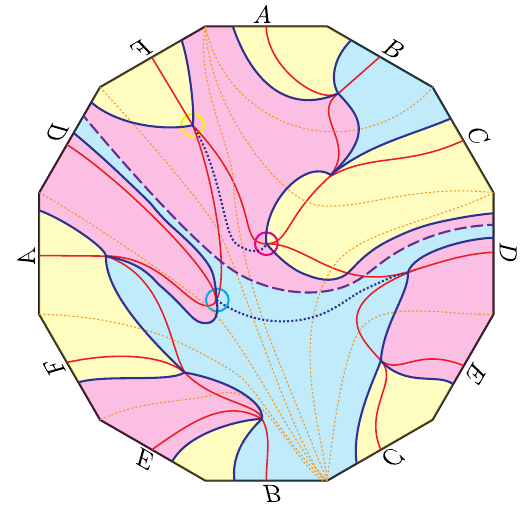} &
  \includegraphics[scale=.81]{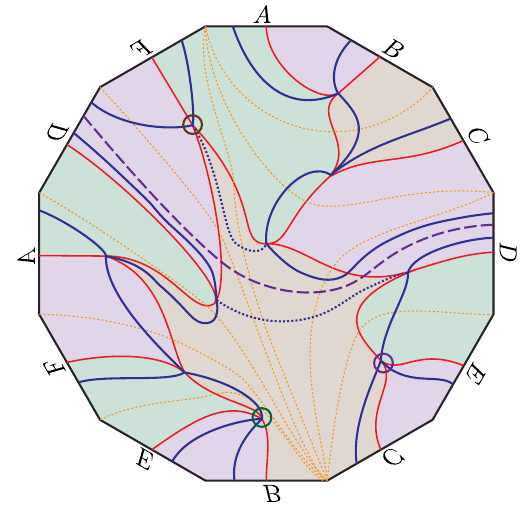}\\
  (c) & (d)
  \end{tabular}
  \caption[Constructing the Heegaard diagram associated to a train track]{\textbf{The
      Heegaard diagram associated to a train track and compression body.} (a) Curves on a genus $2$ surface; the mapping class under consideration is $\tau_a\tau_b\tau_c\tau_d^{-1}$. 
      (b) The output
    of Flipper applied to this mapping class. The letters indicate the edge identifications. (c) The train track, a small perturbation of Flipper's
    output so that the switches are not on the polygon, is in solid red. The dual triangulation is in thin, dashed, orange. The curve $\gamma$ being compressed is purple and dashed. The graph $G'$ is in solid blue; the graph $G$ is the union of the solid and dotted blue curves. The three components of $G\cup\gamma$ are shaded different colors. The small circles indicate a choice of vertices corresponding to these three components. (d) A coloring of the components of $\Sigma\setminus\tau$, and circles indicating a choice of components corresponding to the switches. The green region is a bigon, and contains the puncture. This figure continues as Figure~\ref{fig:HD-from-tt-2}.}
  \label{fig:HD-from-tt}
\end{figure}

\begin{figure}
  \centering
  \begin{tabular}{cc}
  \includegraphics[scale=.88]{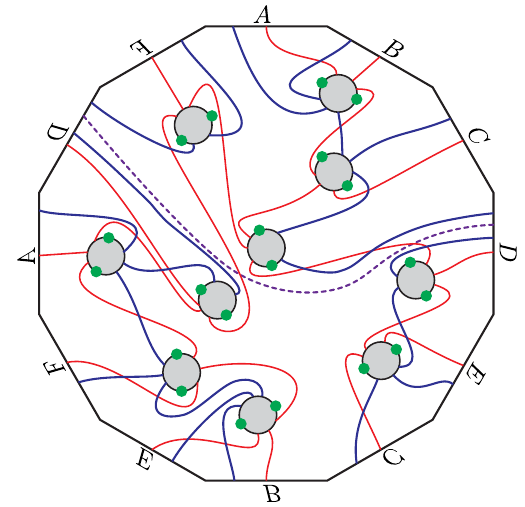} & \includegraphics[scale=.88]{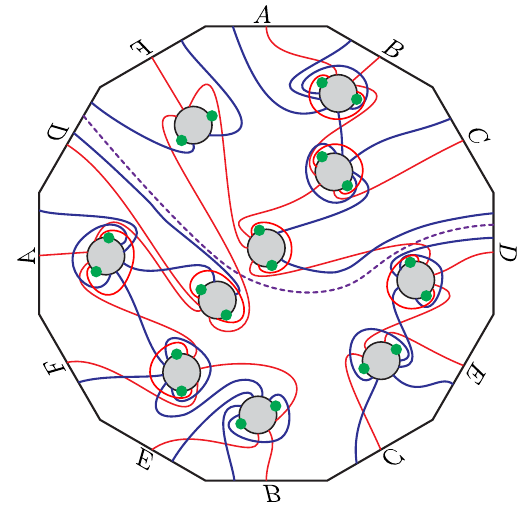}\\
  (e) & (f)\\
  \end{tabular}
    
  \caption[The Heegaard diagram associated to a train track]{\textbf{The Heegaard diagram associated to a train track}. This is a continuation of Figure~\ref{fig:HD-from-tt}. (a) The Heegaard surface (complement of the shaded disks), sutures (small green circles), and the curves $\alphas^{a,1}$ (red) $\betas^c$ (purple, dashed), and $\betas^{a,1}$ (blue). (b) The whole Heegaard diagram, including the curves $\alphas^{a,2}$ and $\betas^{a,2}$. There is an arc in $\betas^{a,2}$ for each boundary component of the diagram except the three circled in Figure~\ref{fig:HD-from-tt}(c) and an arc in $\alphas^{a,2}$ for each boundary component except the three circled in Figure~\ref{fig:HD-from-tt}(d).}
  \label{fig:HD-from-tt-2}
\end{figure}

Corresponding to $\HD$, let $\PMC_L$ and $\PMC_R$ be the arc diagrams:
\[
\begin{split}
&\PMC_L^{\alpha}=\left(\bdy_L\widetilde{\Sigma}, \alphas^a\cap\bdy_L\widetilde{\Sigma},M_L\right)\\
&\PMC_R^{\beta}=\left(\bdy_R\widetilde{\Sigma},\betas^a\cap\bdy_R\widetilde{\Sigma},M_R\right)
\end{split}
\]
where $M_L$ and $M_R$ are the matchings that pair the end points of each arc in $\alphas^a$ and $\betas^{a}$, respectively. 
Consider the bordered sutured Heegaard diagram
$\HD\RcupL\HalfHD(\Id_{\PMC_R^\beta})$
corresponding to gluing a half-identity diagram from~\cite[Construction 8.18]{LOTHomPair} to the right boundary of $\HD$. The bordered-sutured $3$-manifold $Y_{\HD}$ associated to  
$\HD\RcupL\HalfHD(\Id_{\PMC_R^\beta})$ 
is obtained from the compression body $C=C_{\gamma}$ and the train track $\tau$ as follows. By definition, $C$ is constructed by attaching $2$-handles to $[0,1]\times\Sigma$ along $\{1\}\times\gamma_i$ for all $i$. By Lemma \ref{lem:graph}, every connected component of $\Sigma\setminus\gamma$ contains  at least one switch of $\tau$. Let $D\subset \Sigma$ be a union of small pairwise disjoint disk neighborhoods of the switches of $\tau$ in $\Sigma$. Then, $Y_{\HD}=C\setminus \left([0,1]\times D\right)$. 
The sutured part of the boundary is $[0,1]\times\bdy D$ with two parallel, horizontal sutures connecting $\{0\}\times \bdy D$ to $\{1\}\times \bdy D$ on each connected component of $[0,1]\times\bdy D$. The bordered boundary components $\{0\}\times\left(\Sigma\setminus D\right)$ and $\{1\}\times\left(\Sigma[\gamma]\setminus D\right)$ are identified with $F(\PMC_L^{\alpha})$ and $F(\PMC_R^{\alpha})$. As discussed in Section~\ref{sec:bs-background} the bordered-sutured manifold $Y_{\HD}$ can be turned into a special bordered-sutured manifold by attaching $2$-handles to $\{1/2\}\times \bdy D$
to get the compression body and modifying each component of $R_+$ and $R_-$ to become a bigon. This corresponds to attaching tube-cutting pieces (Figure~\ref{fig:arced-HD}) to 
$\HD\RcupL\HalfHD(\Id_{\PMC_R^\beta})$, 
one corresponding to each inner boundary component of $C$.

\begin{theorem}\label{thm:phi-extend-bound-v2} 
  Continuing to use Notation~\ref{not:Ts-etc}, suppose $\psi$ extends
  over some compression body. Then there is a half-bordered compression
  body $Y$ so that $\psi$ extends over $Y$, and a way of extending the
  parameterization to the internal boundary of $Y$ so that the rank of
  $\BSD(Y)$ is less than or equal to
  \begin{equation}\label{eq:DDbound}
    (20(g+s)-18)^s\left((2M(\psi))^{2g}+(2M(\psi)+8)^{2(g+s-1)}\right).
  \end{equation}
\end{theorem}

\begin{proof}
By Lemma~\ref{lem:bound}, there exists a compression body $C$ that $\psi$ extends over and $\ell(C)\le M(\psi)$. Let $\gamma$ be a basis for $C$ satisfying the conditions of Lemma~\ref{lem:basisconditions} and 
let $Y$ be the special bordered-sutured manifold corresponding to $C$, $\gamma$ and $\tau$ as described above. We will show that the number of generators in the bordered-sutured Heegaard diagram for $Y$ constructed above is bounded above by Formula~\eqref{eq:DDbound}.

Consider the $(\alpha,\beta)$-bordered--sutured Heegaard diagram $\HD$ associated to $\gamma$ and $\tau$ as described above. Then the bordered-sutured Heegaard diagram for $Y$ is obtained from 
$\HD\RcupL\HalfHD(\Id_{\PMC_R^\beta})$
by attaching tube-cutting pieces, one for each inner boundary component of $C$.

In the Heegaard diagram $\HD$, the number of components in $\betas^c$, $\betas^a$ and $\alphas^a$ are equal to
\[
  |\alphas^a|=2(g+s-1),\quad\quad|\betas^c|=m,\quad\quad |\betas^a|=2(g+s-m-1).
\]
The total number of intersection points between $\betas^c$ and $\alphas^a$ is less than or equal to $2\ell(C)$, and thus $2M(\psi)$. Moreover, we may arrange for the intersection points of $\alphas^a$ and $\betas^a$ that are not in a neighborhood of $\bdy \widetilde{\Sigma}$ to be on the $\beta$ arcs in $\betas^{a,1}$ that correspond to the edges of $\widetilde{G}$ on the boundary of type $(2)$ annulus components of $\Sigma\setminus (\gamma\cup G)$. After fixing an orientation on the $\gamma$ curves one can define a bijection between these intersection points and the points in $\tau\cap \gamma$. So the number of such intersection points is bounded above by $2M(\psi)$. Therefore, each $\beta$ arc will have at most $2M(\psi)+8$ intersection points with $\alpha$ arcs; see Figure~\ref{fig:psi-HD-near-bdy}. 
Thus, the number of generators in $\HD$ is bounded above by
\[
  (2M(\psi))^m+(2M(\psi)+8)^{2(g+s-m-1)}.
\]
The number of generators in $\HD$ and
$\HD\RcupL\HalfHD(\Id_{\PMC_R^\beta})$ 
are the same, so this quantity bounds the number of generators in 
$\HD\RcupL\HalfHD(\Id_{\PMC_R^\beta})$, 
as well.

\begin{figure}
  \centering
  \includegraphics[scale=.66667]{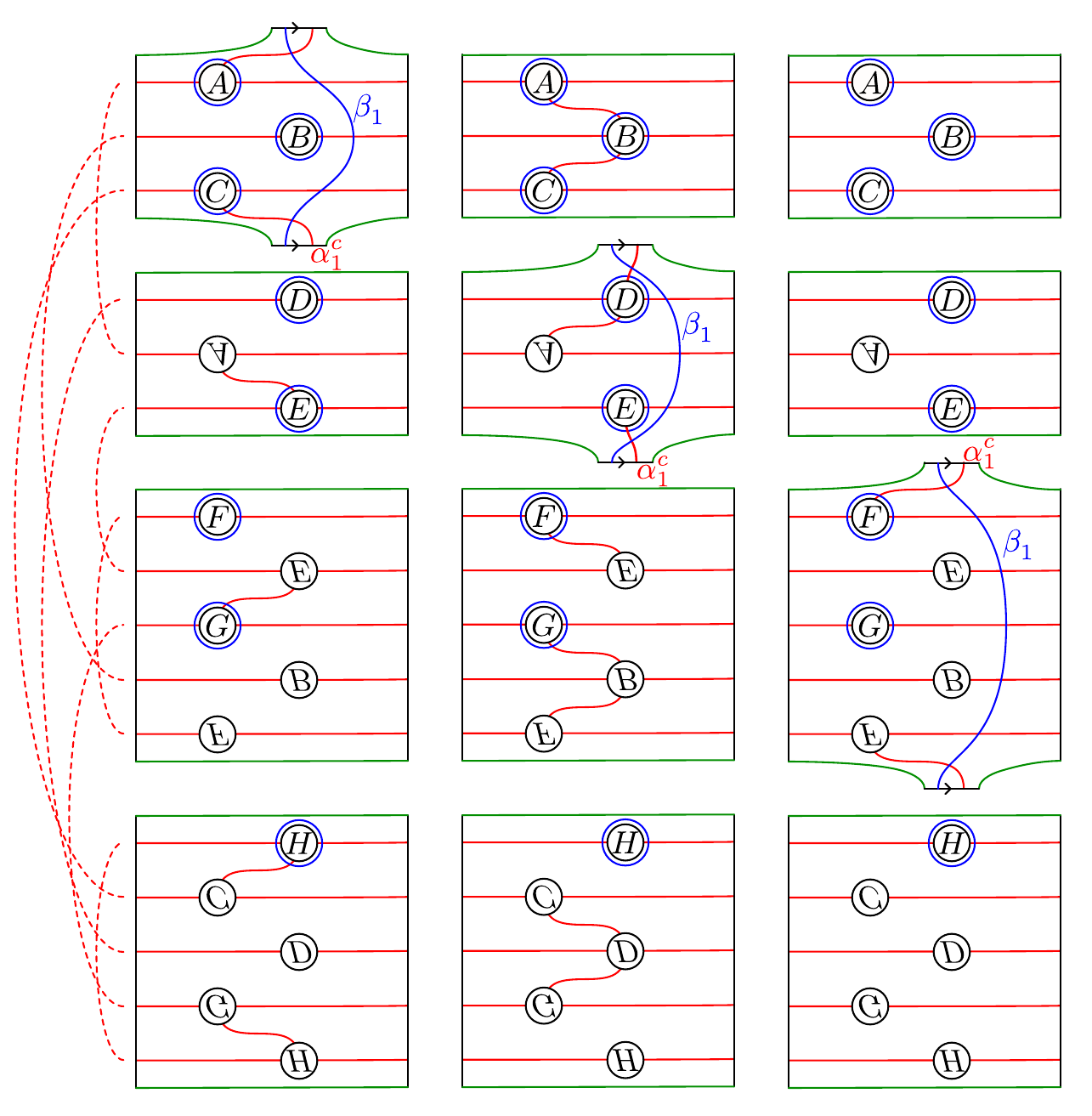}
  \caption{\textbf{The tube-cutting piece.} An arc diagram and three of the four associated tube-cutting pieces. There is one tube-cutting piece associated to each switch; the one not drawn is similar to the third one drawn. The two switches corresponding to the two bigons in the track have more complicated $\alpha$-circles than the two not corresponding to bigons. The arc diagram is the one from the right of Figure~\ref{fig:ad-to-tt}.}
  \label{fig:TC-again}
\end{figure}

In each tube-cutting piece $\TC$, there are two kinds of generators. Each piece
has a single $\alpha$-circle $\alpha_1^c$. There are 
$2(g+s-m-1)$
$\beta$-circles as in the identity
bordered Heegaard diagram, plus one more, distinguished $\beta$-circle
$\beta_1$. In order to make the diagram admissible, we do a finger move so that
$\alpha_1^c$ and $\beta_1$ intersect in two points; see
Figure~\ref{fig:TC-again}. Consequently, given a generator $\x$ for a Heegaard
diagram $\HD$ before gluing on $\TC$, there are two kinds of extensions of $\x$
to a generator of $\HD\cup\TC$: extensions using a point in $\alpha_1^c\cap
\beta_1$ and extensions not using those points. For the former kind of
extension, which $\alpha$-arcs $\x$ occupies completely determines which points
on $\alphas^a\cap \betas$ are chosen: each $\beta$-circle is occupied on the
left if the $\alpha$-arc is not occupied by $\x$ and on the right if the
$\alpha$-arc is occupied by $\x$. So, $\x$ extends to two generators of this kind (corresponding to the
two points in $\alpha_1^c\cap\beta_1$). For the second kind of extension, the
way we have drawn the diagrams in Figure~\ref{fig:TC-again}, which
$\beta$-circles are occupied on the left is still determined, but now one of the
other $\beta$-circles is occupied by a point in $\alpha_1^c$, and one of the
$\alpha$-arcs on the right is occupied by a point in $\beta_1$. So, $\x$ extends
to at most 
$20(g+s-m-1)$
generators of this kind. 
(We have estimated this as $5$ for the maximal number of intersection points between $\beta_1$ and an $\alpha$-arc times twice the number of $\beta$-arcs, for the maximum possible number of intersections of $\alpha_1^c$ with a $\beta$-arc; this is typically a gross
overestimate.)
Thus, the number of generators is multiplied by at most
$20(g+s-m)-18$
for each tube-cutting piece. The number of tube-cutting pieces is
$s$, so in all the number of generators is multiplied by 
\[
  (20(g+s-m)-18)^s.
\]

Therefore, the rank of $\BSD(Y)$ is less than or equal to 
\[
(20(g+s-m)-18)^{s}\left((2M(\psi))^{m}+(2M(\psi)+8)^{2(g+s-m-1)}\right).
\]
So, the claim follows for the fact that $m\le 2g$. 
\end{proof}

Recall that Corollary~\ref{cor:alg-extend} states, somewhat imprecisely, that bordered-sutured Floer homology gives an algorithm to determine if $\psi$ extends over some compression body. We give the algorithm we have in mind as its proof:
\begin{proof}[Proof of Corollary~\ref{cor:alg-extend}]
  First, compute the bordered-sutured bimodule associated to $\psi$ by
  Corollary~\ref{cor:factor-into-arcslides}; its description as a
  composition of arcslides is exactly the input to the algorithm
  in~\cite{LOT4} (or its easy extension to arc diagrams
  in~\cite{AL19:incompressible}).
  Next, construct all bordered Heegaard diagrams for compression bodies
  with at most the number of generators in Formula~\eqref{eq:DDbound}.
  Compute their bordered-sutured modules, as in the proof of
  Corollary~\ref{cor:CFDA-detect-extension}, and then apply
  Theorem~\ref{thm:detect-over-specific} to test if $\psi$ extends over
  each of them.
\end{proof}

Corollary~\ref{cor:alg-extend} is not quite as impractical as it might
seem at first. Assume that we already know the train track $\tau$, e.g.,
as discussed in~\cite{Bell:pA-NP}. Then Lemma~\ref{lem:bound} implies
that $\beta$-circles in the associated Heegaard diagram intersect the
triangles in the dual triangulation $\mathcal{T}$ in at most $M(\psi)$
arcs. A train track for a closed, genus $g$ surface has at most $18g-18$ edges.
The number of ways to distribute the $M(\psi)$ endpoints of these
arcs on the $18g-18$ edges of the triangulation is
$O(M(\psi)^{18g-19})$ (as a function of $M(\psi)$; $g$ is
fixed).
Given the number of endpoints on each edge of a triangle, there is at most one way the edges can lie in the
triangle (since the region containing the switch must be adjacent to all
three sides of the triangle). 
So, there are at most $O(M(\psi)^{18g-19})$ choices of $\beta$-circles. Once the $\beta$-circles are fixed, the $\beta$-arcs are determined by the algorithm above. Thus, there are at most $O(M(\psi)^{18g-19})$ bordered-sutured modules to compute. In particular, the number of modules to compute is polynomial in the length $M(\psi)$.

(Of course, our bound $M(\psi)$ itself grows quickly in terms of, say, the word length of $\psi$, because $r(\psi)$ does. The time to compute the bordered-sutured modules via the algorithm in~\cite{LOT4,AL19:incompressible} also seems to be exponential in the complexity of the manifold involved. One then also needs the computation of the bordered-sutured bimodule from Section~\ref{sec:slide-seq-loop}, which takes exponential time in the length of the splitting sequence via the algorithm in~\cite{LOT4}. So, the algorithm is slow mostly because of the growth rate of $r(\psi)$ and the cost of computing bordered-sutured invariants.)

Theorem~\ref{thm:phi-extend-bound-v2} gives a more algebraic obstruction to $\psi$ extending over a compression body, in that if $\psi$ extends then there must be some bordered-sutured module with at most the number of generators in Formula~\eqref{eq:DDbound} and satisfying the conditions of Theorem~\ref{thm:detect-over-specific}. (Moreover, the module must look like the invariant of a handlebody in various senses, like that if one computes Formula~\eqref{eq:detect-over-specific-formula} with $\Id$ instead of $\psi$, the support should be $0$-dimensional.) Arguably, this is a little like saying that the bimodule $\BSDA(\psi)$ has an eigenvector of bounded length.

\bibliographystyle{hamsalpha}

\bibliography{heegaardfloer}
\end{document}

%% file: defs.tex
\newcommand{\co}{\nobreak\mskip2mu\mathpunct{}\nonscript
  \mkern-\thinmuskip{:}\penalty300\mskip6muplus1mu\relax}
\renewcommand{\th}{^{\text{th}}}

\newcommand{\lsub}[2]{{}_{#1}#2}
\newcommand{\lsup}[2]{{}^{#1}\mskip-.6\thinmuskip#2}

\newcommand{\wt}{\widetilde}

\newcommand{\ZZ}{\mathbb{Z}}
\newcommand{\RR}{\mathbb{R}}
\newcommand{\QQ}{\mathbb{Q}}
\newcommand{\FF}{\mathbb{F}}

\newcommand{\op}{\mathrm{op}}
\newcommand{\Ainf}{A_\infty}

\newcommand{\bdy}{\partial}

\DeclareMathOperator{\Ann}{Ann}
\DeclareMathOperator{\Supp}{Supp}
\DeclareMathOperator{\Spec}{Spec}
\newcommand{\dg}{\textit{dg }}
\newcommand{\Id}{\mathit{Id}}

\DeclareMathOperator{\Mor}{Mor}

\DeclareMathOperator{\Tor}{Tor}
\DeclareMathOperator{\Ext}{Ext}
\DeclareMathOperator{\rank}{rank}
\newcommand{\into}{\hookrightarrow}

\newcommand{\HF}{\mathit{HF}}
\newcommand{\HFa}{\widehat{\mathit{HF}}}
\newcommand{\CF}{{\mathit{CF}}}
\newcommand{\CFa}{\widehat{\mathit{CF}}}

\newcommand{\tHF}{\underline{\HF}}
\newcommand{\tHFa}{\widehat{\tHF}}
\newcommand{\tCF}{\underline{\CF}}
\newcommand{\tCFa}{\widehat{\tCF}}

\newcommand{\Alg}{\mathcal{A}}
\newcommand{\Idem}{\mathcal{I}}

\newcommand{\CFD}{\mathit{CFD}}

\newcommand{\CFA}{\mathit{CFA}}

\newcommand{\CFDA}{\mathit{CFDA}}
\newcommand{\CFDAa}{\widehat{\CFDA}}

\newcommand{\CFDa}{\widehat{\CFD}}

\newcommand{\CFAa}{\widehat{\CFA}}
\newcommand{\tCFDa}{\underline{\CFDa}}

\newcommand{\tCFDAa}{\underline{\CFDAa}}

\newcommand{\DD}{$\mathit{DD}$}
\newcommand{\DA}{$\mathit{DA}$}
\newcommand{\AAm}{$\mathit{AA}$}
\newcommand{\DT}{\boxtimes}
\newcommand{\BSD}{\mathit{BSD}}
\newcommand{\BSA}{\mathit{BSA}}
\newcommand{\BSDA}{\mathit{BSDA}}

\newcommand{\tBSD}{\underline{\BSD}}

\newcommand{\tBSDA}{\underline{\BSDA}}

\newcommand{\tBSA}{\underline{\BSA}}

\newcommand{\SFH}{\mathit{SFH}}
\newcommand{\SFC}{\mathit{SFC}}
\newcommand{\tSFC}{\underline{\SFC}}
\newcommand{\tSFH}{\underline{\SFH}}

\newcommand{\HD}{\mathcal{H}}
\newcommand{\alphas}{\boldsymbol{\alpha}}
\newcommand{\betas}{\boldsymbol{\beta}}

\newcommand{\x}{\mathbf{x}}
\newcommand{\y}{\mathbf{y}}

\newcommand{\spinc}{\mathfrak{s}}
\newcommand{\SpinC}{\mathrm{spin}^c}

\newcommand{\PMC}{\mathcal{Z}}

\DeclareMathOperator{\Sym}{Sym}
\newcommand{\TC}{\mathsf{TC}}

\newcommand{\HalfHD}{\frac{\mathcal{H}}{2}}

\newcommand{\cM}{\mathcal{M}}
\DeclareMathOperator{\ind}{ind}

\newcommand{\nbd}{\mathrm{nbd}}

\theoremstyle{plain}

\numberwithin{equation}{section}
\newtheorem{theorem}[equation]{Theorem}

\newtheorem{proposition}[equation]{Proposition}
\newtheorem{lemma}[equation]{Lemma}
\newtheorem{corollary}[equation]{Corollary}

\newtheorem{convention}[equation]{Convention}
\newtheorem{definition}[equation]{Definition}
\newtheorem{construction}[equation]{Construction}
\newtheorem{notation}[equation]{Notation}

\theoremstyle{definition}

\theoremstyle{remark}
\newtheorem{example}[equation]{Example}
\newtheorem{remark}[equation]{Remark}

\definecolor{darkgreen}{rgb}{0,.25,0}
\definecolor{darkblue}{rgb}{0,0,.5}
\definecolor{darkred}{rgb}{.25,0,0}
\definecolor{darkpurple}{rgb}{.4,.18,.57}
\definecolor{pink}{rgb}{1,.08,.575}

\makeatletter
\providecommand\@dotsep{5}
\def\listtodoname{List of Todos}
\def\listoftodos{\@starttoc{tdo}\listtodoname}
\makeatother

\newcommand{\cupplus}{{\setbox0\hbox{\large$\cup$}\rlap{\hbox to \wd0{\hss\raisebox{3pt}{\tiny$+$}\hss}}\box0}}
\newcommand{\cupminus}{{\setbox0\hbox{\large$\cup$}\rlap{\hbox to \wd0{\hss\raisebox{3pt}{\tiny$-$}\hss}}\box0}}

\newcommand{\out}{\mathit{out}}
